\documentclass[10pt,reqno,a4paper]{article} 

\usepackage{anysize}
\marginsize{3.0cm}{3.0cm}{2.2cm}{2.2cm}

\usepackage[english]{babel}
\usepackage[centertags]{amsmath}
\usepackage{amsfonts,amssymb,amsthm}
\usepackage[svgnames]{xcolor}
\usepackage{comment}
\usepackage{hyperref} 
\usepackage{mathrsfs}
\usepackage[sans]{dsfont}
\usepackage{accents}

\let\OLDthebibliography\thebibliography
\renewcommand\thebibliography[1]{
  \OLDthebibliography{#1}
  \setlength{\parskip}{0pt}
  \setlength{\itemsep}{0pt plus 0.3ex} }

\numberwithin{equation}{section}

\theoremstyle{plain}
\newtheorem{theorem}{Theorem}[section]
\newtheorem{proposition}[theorem]{Proposition}
\newtheorem{lemma}[theorem]{Lemma}

\theoremstyle{definition}
\newtheorem{remarkbase}[theorem]{Remark}
\newenvironment{remark}{\pushQED{\qed} \remarkbase}{\popQED\endremarkbase}

\newcommand{\N}{{\mathbb N}}
\newcommand{\R}{{\mathbb R}}
\newcommand{\C}{{\mathbb C}}

\newcommand{\T}{{\mathbb T}}
\renewcommand{\S}{{\mathbb S}}

\newcommand{\mA}{\mathcal{A}}
\newcommand{\mB}{\mathcal{B}}
\newcommand{\mC}{\mathcal{C}}
\newcommand{\mD}{\mathcal{D}}
\newcommand{\mE}{\mathcal{E}}
\newcommand{\mF}{\mathcal{F}}
\newcommand{\mG}{\mathcal{G}}
\newcommand{\mH}{\mathcal{H}}
\newcommand{\mI}{\mathcal{I}}
\newcommand{\mJ}{\mathcal{J}}
\newcommand{\mK}{\mathcal{K}}
\newcommand{\mL}{\mathcal{L}}
\newcommand{\mM}{\mathcal{M}}

\newcommand{\mP}{\mathcal{P}}
\newcommand{\mQ}{\mathcal{Q}}
\newcommand{\mR}{\mathcal{R}}
\newcommand{\mS}{\mathcal{S}}
\newcommand{\mT}{\mathcal{T}}
\newcommand{\mU}{\mathcal{U}}
\newcommand{\mV}{\mathcal{V}}
\newcommand{\mW}{\mathcal{W}}

\renewcommand{\a}{\alpha}
\renewcommand{\b}{\beta}
\newcommand{\g}{\gamma}
\renewcommand{\d}{\delta}

\newcommand{\e}{\varepsilon}
\newcommand{\ph}{\varphi}
\newcommand{\lm}{\lambda}
\newcommand{\Lm}{\Lambda}

\newcommand{\Om}{\Omega}
\newcommand{\om}{\omega}

\newcommand{\s}{\sigma}

\renewcommand{\th}{\vartheta}

\newcommand{\la}{\langle}
\newcommand{\ra}{\rangle}
\newcommand{\pa}{\partial}
\renewcommand{\div}{\mathrm{div}\,}
\newcommand{\grad}{\nabla}

\title{2D capillary liquid drops with constant vorticity: rotating waves existence and a conditional energetic stability result for rotating circles}
 
\author{\normalsize{Giuseppe La Scala}}
\date{} 

\begin{document}

\maketitle

\noindent
\textbf{Abstract.} We consider a two-dimensional, pure capillary drop of \emph{nearly-circular} shape, having constant vorticity. We write the Craig-Sulem equations on the unit circle, then on the flat torus. We show their Hamiltonian structure and we then observe symmetries and we derive constants of motions. 

After showing linear stability for rotating circles, we prove the existence of rotating waves by combining a bifurcation-theoretical approach together with critical point theory. 

Finally, by exploiting the Hamiltonian structure, we show that whenever volume and barycenter are fixed to be the same as those of rotating circle, this solution is also conditionally energetically stable. This holds in the irrotational case as well, in agreement with the stability analysis of rotating cylinder jets in Rayleigh \cite{Lord.Rayleigh.jets}.

\bigskip

\emph{MSC 2020:} 35R35, 35B32, 35C07, 76B45, 35B38.

\tableofcontents

\section{Introduction}

\medskip

In the paper \cite{Lord.Rayleigh.jets}, Lord Rayleigh studied the motion of nearly-cylinder capillary jets, and he formally show that such jets are broken if they are long enough. In particular, he underlined that such break arises from linear instability of long cylinders, due to translational effects along their symmetry axis. At the contrary, rotational effects are not responsible for such instability, but they produce small oscillations about cylinders. To compute them, Rayleigh assumed zero-vorticity for jets and he computed steady-periodic motions for a two-dimensional drop model. Such formal computation has been numerically shown in \cite{Dyachenko} and rigorously justified in \cite{L}.

\medskip

In this paper, we further consider the presence of constant vorticity for two-dimensional capillary drops. Precisely, we consider the free boundary problem
\begin{align}
&\div u 
 = 0 \quad \qquad\qquad\,\,\,\,\,\quad\text{in} \ \Om_t,\label{div.eq.01}
\\ 
&
\text{curl}\,u 
 = \a_0 \quad \qquad\qquad\quad\text{in} \ \Om_t, \label{curl.eq.01}
 \\
 &
 \pa_t u + u \cdot \grad u + \grad p 
 = 0 \quad \text{in} \ \Om_t, \label{dyn.eq.01}
\\
&
p  = \sigma_0 H_{\Om_t} \quad\qquad\qquad\quad \text{on} \ \pa \Om_t, 
\label{pressure.eq.01}
\\
&
V_t  = \la u , \nu_{\Om_t} \ra \quad\qquad\quad\,\,\,\,\, \ \text{on} \ \pa\Om_t.
\label{kin.eq.01}
\end{align}
where, $u$ is the velocity vector field, $p$ is the pressure scalar field, $\a_0$ is the vorticity coefficient, $\s_0>0$ is the capillarity coefficient, $H_{\Om_t}$ is the curvature of $\pa\Om_t$, $V_t$ is the normal boundary velocity, $\nu_{t}$ is the outward unit normal vector field of $\pa\Om_t$, and finally $\Om_t$ is the geometrical domain of the drop which we assume to be smooth, simply connected and having boundary
\begin{equation}\label{ansatz}\pa\Om_t:=\{\g_t(x):=(1+h(t,x))x\colon\,x\in\S^1\},\qquad\S^1:=\{x\in\R^2\colon\,|x|=1\},
\end{equation}
where $h$ is the elevation function which satisfies $1+h>0$. If $\|h\|_{W^{1,\infty}}$ is small, we say that $\pa\Om_t$ has a \emph{nearly circular} shape.

The set of equations \eqref{div.eq.01}-\eqref{kin.eq.01} is given by the Euler equations \eqref{div.eq.01}-\eqref{dyn.eq.01} for 2D incompressible, perfect fluids with no external forces but with nonzero vorticity, by the capillary boundary condition \eqref{pressure.eq.01} for the pressure and by the boundary velocity condition \eqref{kin.eq.01}. 

Since $\Om_t$ is simply connected, some computations (see Section 3.1) lead to the following general form for the velocity vector field
\begin{equation}\label{u.decomposition.potential}  u=\grad\Phi - \frac{\a_0}{2}\mJ x,
\end{equation}
where $\Phi\colon\,\Om_h\to\R$ satisfies
\begin{equation}\label{velocity.potential.dirichlet}\begin{aligned}&\Delta\Phi(t,\cdot)=0\qquad\qquad\qquad\qquad\text{in}\,\,\Om_t,
\\&\Phi(t,\cdot)\circ\g_t=\psi(t,\cdot)\qquad\qquad\quad\text{on}\,\,\S^1,
\end{aligned}\end{equation}
and the second term is the rotation flow (where $\mJ$ is defined in \eqref{J.matrix.symplectic}).

The equations \eqref{div.eq.01}-\eqref{kin.eq.01} turn out to be equivalent to the Craig-Sulem equations
\begin{align}& \pa_t h = \frac{J}{1 + h} \, G(h)\psi + \frac{\a_0}{2}\mM h, \label{h.eq.}
\\
& \pa_t \psi =
 \frac12 \Big( G(h)\psi 
+ \frac{\la \grad_{\S^1} \psi , \grad_{\S^1} h \ra}{J(1+h)} \Big)^2
- \frac{|\grad_{\S^1} \psi|^2}{2(1+h)^2} 
  - \s_0 H(h) \notag \\ 
  &\qquad+\frac{\a_0}{2}\mM\psi + \frac{\a_0^2}{8}(1+h)^2 + \a_0 K(h)\psi + \s_0-\frac{\a_0^2}{8}. \label{psi.eq.}
\end{align}
Here, $G(h)\psi:=\la\grad\Phi\circ\g_h,\nu_h\ra$ is the Dirichlet-Neumann operator, $H(h):=H_{\Om_h}$ is the curvature operator, $\mM f:=\la\grad_{\S^1}f,\mJ x\ra$ and $K(h)\psi:=\Psi\circ\g_h$ is the trace operator for the harmonic conjugate $\Psi$ (which satisfies $\grad\Psi=\mJ\grad\Phi$). 

Equations \eqref{h.eq.}-\eqref{psi.eq.} admit an equivalent formulation on the flat torus $\T^1$, given by
\begin{align}&\pa_t\xi=e^{-2\xi}[\bar G(\xi)\chi + \frac{\a_0}{2}e^{2\xi}\xi'], \label{xi.eq}
\\&\pa_t\chi=e^{-2\xi}\Big[\frac12\Big(\frac{\bar G(\xi)\chi + \xi'\chi'}{\sqrt{1+\xi'^2}}\Big)^2 - \frac12\chi'^2 + \s_0 e^{\xi}\Big(\Big(\frac{\xi'}{\sqrt{1+\xi'^2}}\Big)' - \frac{1}{\sqrt{1+\xi'^2}}\Big)\Big] \notag
\\&\qquad+\frac{\a_0}{2}\chi' + \frac{\a_0^2}{8}e^{2\xi} + \a_0 \bar K(\xi)\chi + \s_0-\frac{\a_0^2}{8}. \label{chi.eq}
\end{align}
Here, $\bar G(\xi)\chi$ is the Dirichlet-Neumann operator on the flat torus (see Lemma \ref{derivative.change.strip}, point $(iii)$), and $\bar K(\xi)\chi$ is the trace operator for the harmonic conjugate on the flat torus (see Lemma \ref{derivative.change.strip}, point $(v)$).

\bigskip

\emph{Hamiltonian structure for \eqref{xi.eq}-\eqref{chi.eq}.}
Equations \eqref{xi.eq}-\eqref{chi.eq} are similar to pure capillary Water Wave equations on $\T^1$ (see for instance \cite{Wahlen.1}), and thus it is rightful to expect a similar Hamiltonian structure. A natural candidate is the total energy of the drop
\begin{equation}\label{total.energy}\mE:=\frac12\int_{\Om_h}|u|^2\,dx + \s_0\mA(\pa\Om_h),
\end{equation}
where the first term is the kinetic energy, and the second is the capillary energy (where in particular, $\mA(\pa\Om_h)$ is the length of the boundary of the drop). 

However, a careful computation (see Section 2.2) shows that, because of the additive term $\s_0 - \frac{\a_0^2}{8}$ in \eqref{chi.eq}, the Hamiltonian contains also a penalization due to the drop area term. Namely, the correct Hamiltonian is
\begin{align}\mH
&=\mE - \Big(\s_0-\frac{\a_0^2}{8}\Big)\int_{\Om_h}1\,dx \notag
\\&=\frac{1}{2}\int_{\T^1}\chi\,\bar G(\xi)\chi\,d\th + \s_0\int_{\T^1}e^\xi\,\sqrt{1+\xi'^2}\,d\th - \Big(\s_0 - \frac{\a_0^2}{8}\Big)\frac12\int_{\T^1}e^{2\xi}\,d\th \notag
\\&\quad-\frac{\a_0}{4}\int_{\T^1}e^{2\xi}\,\chi'\,d\th + \frac{\a_0^2}{32}\int_{\T^1}e^{4\xi}\,d\th 
\label{Hamiltonian.function}
\end{align}
Similarly as in \cite{Wahlen.1}, \eqref{xi.eq}-\eqref{chi.eq} are formally Hamiltonian, in the sense that 
\begin{equation}\label{hamiltonian.equations}\pa_t\begin{pmatrix}\xi \\ \chi
\end{pmatrix} = J_{\a_0}(\xi)\grad\mH(\xi,\chi),\end{equation}
and $J_{\a_0}$ is
\begin{equation*}J_{\a_0}(\xi):=\begin{pmatrix}0 & e^{-2\xi} \\ -e^{-2\xi} & \a_0\pa_\th^{-1}\Pi_0^\perp
\end{pmatrix}.
\end{equation*}
Here, $\pa_\th^{-1}$ denotes the indefinite integral operator while $\Pi_0^\perp$ the $L^2-$projection operator onto the zero-average functions. Still, \eqref{hamiltonian.equations} are not Hamiltonian, as highlighted in Remark \ref{remark:hamiltonian.if.volume.conserved}: the matrix operator $J_{\a_0}$ is invertible only if we fix the volume $\mV(\xi)=\frac12\int_{\T^1}e^{2\xi}\,d\th$ of the fluid, which is a \emph{nonlinear} functional, and thus it is not easy to work over such Hamiltonian structure.
 
However, it is possible to conjugate \eqref{xi.eq}-\eqref{chi.eq} with the change of coordinates 
\begin{equation}\label{change.of.coordinates}  \begin{pmatrix}\xi \\ \chi
\end{pmatrix}=\mC(\zeta,\g):=\begin{pmatrix}\zeta \\ \g + \frac{\a_0}{4}\pa_\th^{-1}\Pi_0^\perp e^{2\zeta}
\end{pmatrix},   
\end{equation}
so as to get an equivalent formulation which is Hamiltonian on the phase space $H^{s_1}\times\dot H^{s_2}$, and its Hamiltonian structure is given by the Hamiltonian function 
\begin{equation}\label{new.hamiltonian.function}\bar\mH:=\mH\circ\mC\end{equation}
and the Poisson tensor 
\begin{equation}\label{new.Poisson.tensor}J(\zeta):=\begin{pmatrix}0 & e^{-2\zeta} \\ -e^{-2\zeta} & 0
\end{pmatrix}.\end{equation} 
These are the drop version of the well-known Wahlén coordinates for the flat torus problem (see \cite{Wahlen.1}).

\bigskip

\emph{Linear stability of rotating circles and rotating waves.} Considering the set $\mT:=\{(0,0)\}\cup\{(\ell,m)\colon\,\ell\in\N,m\in\{-1,1\}\}$ and the $L^2$ Fourier basis
\begin{equation*}\ph_{\ell,m}(x):=\begin{cases}C_{\ell}\cos(\ell x) & \text{if}\,\,\ell\in\N,\,m=1,
\\C_{\ell}\sin(\ell x) & \text{if}\,\,\ell\in\N,\,m=-1,
\\C_0 & \text{if}\,\,\ell=0,\,m=0,
\end{cases}
\end{equation*}
by considering the Fourier expansions
\begin{equation*}\xi=\sum_{(\ell,m)\in\mT}\xi_{\ell,m}\ph_{\ell,m},\qquad\chi=\sum_{(\ell,m)\in\mT}\b_{\ell,m}\ph_{\ell,m},
\end{equation*}
it is remarkable that linearization about the \emph{rotating circle solution} $(\xi,\chi)=(0,0)$ satisfies the eigenvalue equation
\begin{equation*}\lambda^2=\begin{cases}0 & \text{if}\,\,(\ell,m)=(0,0),
\\\ell\Big(-\s_0\ell^2 -\frac{\a_0^2}{4}\ell + \s_0 - \frac{\a_0^2}{4}\Big) & \text{if}\,\,(\ell,m)\ne(0,0).
\end{cases}
\end{equation*}
It always holds that $\lm^2\le0$ for all $(\ell,m)\in\mT$, which implies that all of the eigenvalues are purely imaginary, and so rotating circle is linearly stable.

\medskip

This addresses to the following points. First, close to linear centers it is reasonable to expect the onset of small oscillations about them. Second, in absence of global well-posedness results for Cauchy problems, it is senseful to ask if, more generally, there are physically relevant global-in-time solutions.
To both these aims, we do the following considerations. 

\medskip

A consequence of the Hamiltonian structure for \eqref{xi.eq}-\eqref{chi.eq} and of its invariance by rotations $\mT_\a u(\th):=u(\th+\a)$ is the conservation of the total angular momentum of the drop. In the natural Hamiltonian structure, it is expressed as
\begin{equation}\label{Angular.Momentum}\mI(\xi,\chi)=-\frac12\int_{\T^1}e^{2\xi}\chi'\,d\th + \frac{\a_0}{8}\int_{\T^1}e^{4\xi}\,d\th.
\end{equation}
The first term is due to the irrotational part of the velocity field, while the second term is due to the rotation flow induced by vorticity. In Wahlén coordinates, the total angular momentum is more compactly expressed as
\begin{equation}\label{Angular.Momentum.Wahlen}\bar\mI(\zeta,\g)=\mI\circ\mC(\zeta,\g)=-\frac12\int_{\T^1}e^{2\zeta}\g'\,d\th.
\end{equation}
In any case, such invariances and constants of motion lead to the search for the \emph{rotating waves}. These are solutions which are steady with respect to a corotating observer. For the equations \eqref{h.eq.}-\eqref{psi.eq.}, they are expressed as
\begin{equation}\label{original.rotating.waves}h(t,x):=\bar\eta(R(\om t)x),\qquad\psi(t,x):=\bar\beta(R(\om t)x),\qquad t\in\R,\,x\in\S^1,
\end{equation}
where $R(\th)$ is the rotation matrix around the axis orthogonal to the plane of the drop:
\begin{equation*}R(\th):=\begin{pmatrix}\cos\th & -\sin\th \\ \sin\th & \cos\th \end{pmatrix}.
\end{equation*}
In the equations \eqref{xi.eq}-\eqref{chi.eq}, they can be read as
\begin{equation}\label{rotating.waves}\xi(t,\th):=\eta(\th+\om t),\qquad\chi(t,\th):=\b(\th + \om t),\qquad t\in\R,\,\th\in\T^1,
\end{equation}
and in Wahlén coordinates,
\begin{equation}\label{rotating.waves.Wahlen}\zeta(t,\th):=\bar\eta(\th+\om t),\qquad \g(t,\th):=\bar\b(\th+\om t)\qquad t\in\R,\,\th\in\T^1.
\end{equation}
Working with Wahlén coordinates is equivalent to work with natural coordinates. In particular, thanks to Hamiltonian structures for drop equations, rotating waves can be characterized in natural coordinates as critical points for the action
\begin{equation}\label{action.for.natural.coordinates}\mE:=\mH - \om\mI,
\end{equation}
and in Wahlén coordinates,
\begin{equation}\label{action.for.Wahlen.coordinates}\bar\mE:=\bar\mH - \om\bar\mI.    
\end{equation}
We decide to work with Wahlén coordinates, since its Hamiltonian structure is simpler. 

We observe that $(\bar\eta,\bar\b)=(0,0)$ is always a rotating wave, no matter what value $\om$ holds. Our task is to find different rotating solutions, which we call \emph{nontrivial}, having in general the $\kappa-$fold symmetry, which means that there exists some $\kappa\in\N$ such that for all $\th\in\T^1$,
\begin{equation*}f\Big(\th + \frac{2\pi}{\kappa}\Big)=f(\th).
\end{equation*}
Physically speaking, these correspond to rigid profiles which are invariant by rotations of $\frac{2\pi}{\kappa}$. In particular, for $\kappa\ge3$, they have the same rotational symmetry of a regular $\kappa-$polygon.

To find such waves, we follow a bifurcation-theoretical approach, and by linearizing the rotating wave equations about the rotating circle, we find the resonance relation
\begin{equation}\label{resonance.relation}\om=\frac{\a_0}{2}-\frac{\a_0}{2\kappa\ell}\pm\sqrt{\Big(\frac{\a_0}{2\kappa\ell}\Big)^2 + \frac{4\s_0((\kappa\ell)^2-1)-\a_0^2}{4\kappa\ell}},
\end{equation}
where $\ell\in\N$ is fixed. Such relation has sense if the following condition is met:
\begin{equation}\label{existence.resonance.condition}\Delta:=(\kappa\ell-1)\Big(C\kappa\ell(\kappa\ell+1) - \frac14\Big)\ge0,\qquad C:=\frac{\s_0}{\a_0^2}.
\end{equation}
In particular, this is true for all $\ell,\kappa\in\N$ if $C\ge\frac1{24}$.
In correspondence of the values \eqref{resonance.relation} for $\om$, the linearized equation has a kernel $V$ of multiplicity $2$ or even $4$. As shown in Lemma \ref{lemma:simple.eigenvalue}, $\om$ is double for large values of $\ell$, while it can have multiplicity $4$ otherwise.

At this point, there are two key remarks. The first is that, once we write $u=v+w(\om,v)$ (where $v\in V$ are the kernel vectors for the linearized operator $\mL_\om$, see \eqref{variational.linearization}, and $w=w(\om,v)\in V^{\perp_{L^2}}$), the bifurcation equation is the critical point equation for the functional $\Phi:=\mE(v+w(\om,v))$. The second is that the rotating wave equation is invariant by the torus action. 
As a result, we get the following two consequences. 

\medskip

The first consequence is that when the dimension of $V$ is 2, we can reduce the bifurcation equation to a $1-$dimensional problem, and so we are reduced to check Crandall-Rabinowitz Transversality condition, which realizes if and only if $\om\ne\frac{\a_0}{2}-\frac{\a_0}{2\ell}$. Together with the fact that $\mE$ is invariant by reversibility (that is, $\mE(\eta(-\th),-\b(-\th))=\mE(\eta,\b)$), this gives that all the critical $\mT_\a-$orbits are generated by rotating waves such that $(\eta(-\th),\beta(-\th))=(\eta(\th),-\beta(\th))$ (see Theorem \ref{thm:rw.crandall.rabinowitz}). This is what happens, in particular, for the irrotational case $\a_0=0$, where the eigenvalues are always simple.

\medskip

The second consequence is that, even where the dimension of $V$ is 4, still we can find rotating waves.
There are, however, some technical remarks to underline. If $\om_*>\frac{\a_0}{2}$, then we can carry on the
classical Moser-Weinstein argument (see for instance \cite{Moser} or \cite{Weinstein}), which exploits the conservation of angular momentum, through which the rotating waves are eventually parametrized. This argument is possible because for such values of $\om_*$, the reduced angular momentum \eqref{IZ.approx} on the kernel variables is a positive-definite quadratic form up to cubic terms, and so its level sets are compact. As a result, we get at least two distinct $\mT_\a-$orbits of critical points, corresponding to one maximum and one minimum of the reduced action on the reduced angular momentum constraint. 

\medskip

If instead $\om_*<\frac{\a_0}{2}$, we could generally lose the ellipticity of the reduced angular momentum, and due to the complexity of the arithmetic of \eqref{resonance.relation}, we cannot establish if such values of $\om_*$ correspond to eigenvalues with multiplicity 2 or 4. Anyway, if we suppose the latter case, we can follow the approach in \cite{BBMM} of parametrizing the rotating waves by their angular frequency. Precisely, we exploit some topological considerations about the critical point $v=0$ of functions of the form $f_\om=(\om-\om_*)|v|^2 + O(|v|^3)$, defined on small balls in $\R^4$, and get a Minimax principle to deduce the existence of nonzero critical points (see Lemma 5.2 and 5.4 in \cite{BBMM}). This can be obtained once we show the existence of sets invariant by deformation, namely, Mountain pass critical levels if $0$ is an isolated maximum or a minimum, or Conley blocks if $f$ attains both positive and negative values close to $0$ (being still a isolated critical point).

\medskip

We remark that what happens close to the frequencies $\om=\frac{\a_0}{2} - \frac{\a_0}{2\ell}$ is an open problem. The reason is that by \eqref{resonance.equation}, we get that such frequencies are admissible if and only if
\begin{equation}\label{degenerate.integers}\ell=-\frac12 + \frac12\sqrt{1+\frac{\a_0^2}{4\s_0}}\in\N,
\end{equation}
which happens if and only if $1 + \frac{\a_0^2}{4\s_0}=k^2$ for some odd integer $k$. We also notice that for such frequencies, by \eqref{V.kernel} the kernel space $V$ is made of vectors of the form
\begin{equation*}\mathtt v_{\ell,m}=\begin{pmatrix}\ph_{\ell,m} \\ 0
\end{pmatrix},
\end{equation*}
where $m\in\{-1,1\}$ and $\ell$ is given by \eqref{degenerate.integers}. 
This implies that the above frequencies $\om$ are double eigenvalues, but the Crandall-Rabinowitz transversality condition \eqref{pa.om.L.scalar.product} is violated.

\bigskip

\emph{Conditional energetic stability for rotating circles with fixed volume.} As already seen, rotating circles are linearly stable, which addresses to the natural question of their nonlinear stability. 
The natural idea to pursue such a purpose is to take into account the Hamiltonian structure, and to show that the Hamiltonian $\bar\mH$ is in fact a Lyapunov functional. If this held, we would get a \emph{conditional energetic stability}, namely, stability with respect to the norm $\|(\zeta,\g)\|_E:=\|\zeta\|_{H^1} + \|\g\|_{H^\frac12}$, which makes $\bar\mH$ continuous. Such a result would be only conditional in absence of local well-posedness results in such norm.

Up to some further technical considerations, in the spirit of \cite{GSS.I} and \cite{VWW}, it would be enough to show that $\bar\mH$ is strictly convex close to the rotating circle, which leads to compute the spectrum of its Hessian operator and trying to show that it is uniformly bounded from below by a positive constant. However, a situation like this is, in fact, far from obvious.

Indeed, three problematic eigenvalues arise. The first twos are two zero eigenvalues due to the modes $(\ell,m):=(1,\pm1)$, and the third one is due to the mode $(\ell,m)=(0,0)$ and it has variable sign:
\begin{equation}\label{defective.eigenvalue}\lm_1(0)=-\s_0 + \frac{\a_0^2}{4}=-\a_0^2\Big(C-\frac14\Big),
\end{equation}
where $C$ is defined in \eqref{existence.resonance.condition}, and we call it \emph{modified Bond number}, because it measures how strong is the capillarity effect compared to that of vorticity. This could apparently compromise the stability of rotating circle, 
and it seems that also irrotational circles are unstable, since $\lm_1(0)=-\s_0$ in this case. However, this does not contradict the stability of purely rotating cylinders in \cite{Lord.Rayleigh.jets}, whose analysis is done by fixing the volume of the perturbations equal to that of the cylinder jet. In this spirit, by fixing the volume of the drops equal to that of the rotating circle, the average mode $(\ell,m)=(0,0)$ gets order $O(\|(\zeta,\g)\|_E^2)$, and so it can be neglected. This erases then the negative direction.

However, we cannot still claim that in such volume constraint, the rotating circle gets energetically stable, since we are left with the other two degenerate modes $(\ell,m)=(1,\pm1)$. It is reasonable to ask if we can find another constant of motion to make these directions small enough to get strict convexity for $\bar\mH$. Actually, the answer is positive, because these modes are close to the vector field associated to the drops barycenter velocity
\begin{equation*}\mB(\zeta,\g)=\int_\T e^\zeta\Big(\g' + \frac{\a_0}{4}\Pi_0^\perp e^{2\zeta} - \frac{\a_0}{6}e^{2\zeta} \Big)\begin{pmatrix}-\sin\th \\ \cos\th
\end{pmatrix}\,d\th.
\end{equation*}
Indeed, when we fix it to be the same as that of rotating circle, namely $\mB=(0,0)$, then this also make the position of the barycenter 
\begin{equation*}\mP(\zeta)=\frac12\int_{\T}e^{3\zeta}\begin{pmatrix}\cos\th \\ \sin\th\end{pmatrix}\,d\th
\end{equation*}
another constant of motion, and then if we fix $\mP(\zeta)=(0,0)$, we get smallness also for the left degenerate modes (see Lemma \ref{lemma:smallness.on.constraints}). Physically speaking, imposing such constraints is reasonable because we are essentially requiring that in average, the drop is subjected only to rotation and deformation, so we can always translate the drop reference frame in order to fix the barycenter at the origin.

\medskip

Summing everything up, we get that no matter how large is the modified Bond number, conditional energetic stability for rotating circles is guaranteed as long as we consider the class of perturbations along the volume and barycenter velocity constraints $\bar\mV=\bar\mV(0),\,\mB=(0,0),\,\mP=(0,0)$ (see Theorem \ref{thm:conditional.stability}). 

\bigskip

\emph{Further literature about pure capillary drop.}
The problem of the fluid motion of a capillary drops
dates back to the paper \cite{Lord.Rayleigh.jets} by Lord Rayleigh, who studied nearly-circular 2D drops for explaining capillary jets.
The formulation of the free boundary problem for the drop in the irrotational regime 
as a system of equations on $\S^1$ has been used in \cite{Beyer.Gunther} and for $\S^2$ 
in \cite{B.J.LM, Beyer.Gunther, Shao.initial.notes, Shao.G.paralinearization}.
For 3D drops, its Hamiltonian structure has been obtained in \cite{Beyer.Gunther} 
and, in the present formulation, in \cite{B.J.LM}. 
The Dirichlet-Neumann operator for the 3D drops 
has been studied in \cite{Shao.G.paralinearization} (paralinearization)
and \cite{B.J.LM} (linearization, analyticity, tame estimates). 
Local well-posedness results for the Cauchy problem have been obtained in 
\cite{Beyer.Gunther, CS, Shao.G.paralinearization}, 
and continuation results and a priori estimates in \cite{JL, Shatah.Zeng}. The rigorous existence of rotating traveling waves has been proved in \cite{BLS}.
As for the 2D case, in \cite{L} it has been provided the Craig-Sulem formulation for irrotational drops and the existence of rotating waves with no symmetry assumptions, whose numerical evidence is shown in \cite{Dyachenko}. 
We also mention the recent paper \cite{MNS} for the existence of 2D steady bubbles for capillary drops.

\bigskip

\emph{Further literature about free-boundary problem for perfect fluids with vorticity and bifurcation.} We start by mentioning \cite{Wahlen.1} for the Hamiltonian structure of Water Wave equations on $\T^1$ with constant vorticity, then \cite{Wahlen, Martin, BBMM} for local bifurcations of steady periodic waves and \cite{Constantin.Strauss, CSV} for global bifurcations of steady periodic waves. We then mention \cite{Wahlen.Weber, K.L} for general vorticity. 

\bigskip

\emph{Further literature about stability in Hamiltonian systems.} We mention \cite{GSS.I} and the more recent paper \cite{VWW} for stability problems for ground states in Hamiltonian systems.

\section{Geometrical and functional settings} \label{sec:preliminary}

Let us consider an open set $\Om\subset\R^2$ whose boundary is at least $C^2$-regular. Let us call $\nu_\Om$ the outward unit normal vector field of $\pa\Om$: the tangent line at $x\in\pa\Om$ is
\begin{equation}\label{tangent.line}T_x(\pa\Om)=\{y\in\R^2\colon\,\langle y,\nu_\Om(x)\rangle=0\},
\end{equation}
where we are denoting by $\la\cdot,\cdot\ra$ the Euclidean scalar product.

Thus, calling $N_x(\pa\Om)$ the normal line at $x\in\pa\Om$, we have the orthogonal decomposition $\R^2=N_x(\pa\Om)\oplus T_x(\pa\Om)$, so we can define the projection operators $\Pi_{N_x(\pa\Om)}\colon\R^2\to N_x(\pa\Om)$ and $\Pi_{T_x(\pa\Om)} = I - \Pi_{N_x(\pa\Om)}$, where $I$ is the identity map; one can notice that 
\begin{equation}\label{normal.projection}\Pi_{N_x(\pa\Om)}=\nu_\Om(x)\otimes\nu_\Om(x).
\end{equation}

We can project any vector field $F\colon\R^2\to\R^2$ over $N_x(\pa\Om)$ and $T_x(\pa\Om)$. In particular, given any function $f\colon\R^2\to\R$ and its gradient $\nabla f$, we can define the normal and the tangential gradient of $f$ respectively as
\begin{align}&\nabla_\nu f(x):=\Pi_{N_x(\pa\Om)}[\nabla f(x)]=\langle\nabla f(x),\nu_\Om(x)\rangle\nu_\Om(x), \label{normal.gradient}
\\&\nabla_{\pa\Om}f(x):=\Pi_{T_x(\pa\Om)}[\nabla f(x)]=\nabla f(x) - \nabla_\nu f(x). \label{tangential.gradient}
\end{align}
Thus we can define the tangential differential as
\begin{equation}\label{tangential.differential}D_{\pa\Om}f(x):=[\nabla_{\pa\Om}f(x)]^T,
\end{equation}
a definition that can be extended to any smooth vector field $F\colon\R^2\to\R^2$ as
\begin{equation}\label{extended.tangential.differential}D_{\pa\Om}F(x):=DF(x) - DF(x)\nu_\Om(x)\otimes\nu_\Om(x),
\end{equation}
where $DF(x)$ is the differential (or Jacobian matrix) of $F$ on $\R^2$.
For a vector field $F\colon\,\pa\Om\to\R^3$ we define the tangential divergence as the trace of
the tangential differential
\begin{equation}\label{tangential.divergence}\div_{\pa\Om}F:=\text{Tr}(D_{\pa\Om}F)=\div F - \la (DF)\nu_\Om,\nu_\Om\ra.
\end{equation} 
From this, we define the mean curvature as
\begin{equation}\label{mean.curvature.generic}H_\Om:=\div_{\pa\Om}\nu_\Om.
\end{equation}
In general, it holds the divergence theorem for hypersurfaces
\begin{equation}\label{div.theorem.hypersurfaces}\int_{\pa\Om}\div_{\pa\Om}F\,d\mH_1=\int_{\pa\Om}H_\Om\la F,\nu_\Om\ra\,d\mH^1,
\end{equation}
where $\mH^1$ is the 1-dimensional Hausdorff measure.

Given any smooth function $f\colon\,\pa\Om\to\R$ and any extension $\tilde{f}$ of it over an open neighborhood of $\pa\Om$ in $\R^2$, one can alternatively define the tangential gradient as
\begin{equation}\grad_{\pa\Om}f(x):=\grad\tilde f(x),
\end{equation}
and one can show that this definition does not depend on the choice of the extension $\tilde{f}$, and coincide with \eqref{tangential.gradient}; same holds for vector fields $F\colon\,\pa\Om\to\R^2$.

Let us now turn our attention to the unit disk $\Om:=\mathbb{D}$ whose boundary is the unit circle $\S^1$. 
Given any function $f\colon\,\S^1\to\R$, we define its $0-$homogeneous and $1-$homogeneous extensions to $\R^2\setminus\{0\}$ respectively as
\begin{equation}\label{homogeneous.extensions}\mE_0f(x):=f\Big(\frac{x}{|x|}\Big),\qquad\mE_1f(x):=|x|\mE_0 f(x),\qquad\forall x\in\R^2\setminus\{0\}
\end{equation}
and analogously we do for any vector field $F\colon\,\S^1\to\R^2$. By these definitions and by the fact that $\nu_{\mathbb{D}}(x)=x$ for all $x\in\S^1$, we get 
\begin{align}&\nabla_{\S^1}f(x)=\nabla(\mE_0f)(x)= (\grad\mE_0f)(x) - \la(\grad\mE_0f)(x),x\ra x, \label{tang.grad.S1}
\\&D_{\S^1}f(x)=(\grad_{\S^1}f(x))^T.
\end{align}
By definition, for all $x\in\S^1$ we have
\begin{equation}\label{grad.s1.orth.}\la\grad_{\S^1}f(x),x\ra = 0.
\end{equation}
Let us turn our attention to the deformed disk $\Om_h$:
\begin{equation}\label{Om.h}\Om_h:=\{(1+\mE_0h(x))x\,\colon\,x\in\mathbb{D}\},\qquad h\colon\,\S^1\to(-1,+\infty).
\end{equation}
Then, its boundary is
\begin{equation}\label{pa.Om.h}\pa\Om_h:=\{(1+h(x))x\,\colon\,x\in\mathbb{S}^1\},\qquad h\colon\,\S^1\to(-1,+\infty),
\end{equation}
which is diffeomorphic to $\S^1$ through the map \begin{equation}\label{gamma.diff}\g_h:\S^1 \to \pa \Om_h\qquad\g_h(x):=(1+h(x))x,\qquad x\in\S^1.\end{equation}
The 1-homogeneous extension of $\g_h$ on $\R^2\setminus\{0\}$ is
\begin{equation}  \label{def.tilde.h}
\mE_1 \g_h(x) :=  x \, (1 + \mE_0 h(x)) ,\qquad\forall x\in\R^2\setminus\{0\}.
\end{equation}
The Jacobian matrix, its determinant, its inverse and transpose are
\begin{align}&D (\mE_1 \g_h)  = (1 + \mE_0 h) I + x \otimes \nabla \mE_0 h, \label{Jac.tilde.gamma}
\\&\det D(\mE_1 \g_h) = (1 + \mE_0 h)^2, \label{det.D.tilde.gamma}
\\&[D (\mE_1 \g_h)]^{-1} = \frac{I}{1+ \mE_0 h}  
- \frac{x \otimes \nabla \mE_0 h}{(1 + \mE_0 h)^2}, \label{D.tilde.gamma.inv}
\\&[D (\mE_1 \g_h)]^{-T} 
= \frac{I}{1+ \mE_0 h} 
- \frac{(\grad \mE_0 h) \otimes x}{(1 + \mE_0 h)^2}. \label{D.tilde.gamma.inv.T} 
\end{align}
Then, the tangent line $T_{\g_h(x)}(\g_h(\S^1))$ to the curve $\g_h(\S^1)$ at $\g_h(x) \in \g_h(\S^1)$ can be proved to be
\begin{equation} \label{tangent.pa.Om}
T_{\g_h(x)}(\g_h(\S^1)) 
= \{ D_{\S^1} \g_h(x) v : v \in T_x(\S^1) \}
\qquad\forall x \in \S^1,
\end{equation}
while the outward unit normal vector field to $\g_h(\S^1)$ at $\g_h(x)$ is
\begin{equation}  \label{nu.h}
\nu_{h}\circ\gamma_h(x) = \frac{(1 + h(x)) x - \grad_{\S^1} h(x)}{\sqrt{(1 + h(x))^2 + |\grad_{\S^1} h(x)|^2}} = \frac{(1 + h(x)) x - \grad_{\S^1} h(x)}{J(h)}\qquad x\in\S^1,
\end{equation}
where
\begin{equation}\label{J.def}J(h):=\sqrt{(1+h)^2 + |\grad_{\S^1}h|^2}
\end{equation}

\begin{lemma}
\label{lem:formulas}
Let $\Omega_h$ be as in \eqref{Om.h},  
and let $\g_h$ be as in \eqref{gamma.diff}.
For any function $f\colon\, \pa \Omega_h \to \R$,
let $\tilde f$ be its pullback by $\g_h$, namely $\tilde f(x) = f\circ\g_h(x)$. 
Then 
\begin{equation}
\begin{split}
\nabla_{\pa \Omega_h} f\circ\gamma_h(x) &= \frac{\nabla_{\S^1} \tilde f(x)}{1+h} + \frac{\la \nabla_{\S^1}  \tilde f, \nabla_{\S^1} h \ra}{(1+h)J}\nu_{h} \circ\gamma_h(x)  \\
&=  \frac{\nabla_{\S^1}  \tilde f(x)}{1+h}  - \frac{\la \nabla_{\S^1}  \tilde f, \nabla_{\S^1} h \ra}{(1+h)J^2} \nabla_{\S^1} h + \frac{\la \nabla_{\S^1}  \tilde f, \nabla_{\S^1} h \ra}{J^2} x,
\end{split}
\label{eq.1.in.lem:formulas}
\end{equation}

\end{lemma}

\begin{proof} See \cite{L}.

\end{proof}

We observe that $\R^2\setminus\{0\}$ is diffeomorphic to the periodic infinite strip $\mathcal{S}:=\R\times\T^1$ through the diffeomorphism $\phi\colon\,\mS\to\mathbb{R}^2\setminus\{0\}$, defined as
\begin{equation}\label{phi.map}\phi(\rho,\th):=e^{\rho}\begin{pmatrix}\cos\th \\ \sin\th
\end{pmatrix}\qquad\forall\th\in\T^1,\,\rho\in\R.
\end{equation}
In particular, $\phi(\T^1\times(-\infty,0))=\mathbb{D}\setminus\{0\}$ and $\phi(\{0\}\times\T^1)=\S^1$. Moreover, setting $\xi(\th):=\log(1+h\circ\phi(0,\th))=\log(1+h(\cos\th,\sin\th))$, we have
\begin{equation}\label{phi.omegah.t1graph}\g_\xi(\th):=\phi(\xi(\th),\th)=e^{\xi(\th)}\begin{pmatrix}\cos\th \\ \sin\th\end{pmatrix} = (1 + h(\cos\th,\sin\th))\begin{pmatrix}\cos\th \\ \sin\th
\end{pmatrix}.
\end{equation}
Setting
\begin{equation}\label{omegah.diffeo.to.graph}\begin{aligned}&\Omega_\xi:=\{(\rho,\th)\in\mS\colon\,\th\in\T^1,\,-\infty<\rho<\xi(\th)\},
\\&\pa\Omega_\xi = \{(\rho,\th)\in\mS\colon\,\th\in\T^1,\,\rho=\xi(\th)=\log(1+h(\cos\th,\sin\th))\},
\end{aligned}\end{equation}
we get $\phi(\Om_\xi)=\Om_h\setminus\{0\}$, $\phi(\pa\Om_\xi)=\pa\Om_h$ and that $\g_\xi\colon\,\T^1\to\pa\Om_\xi$ is a diffeomorphism from $\T^1$ to $\pa\Om_\xi$, namely a parametrization of $\pa\Om_\xi$. For all $\th\in\T^1$, denoting by $f':=\pa_\th f$ the derivative of $f$ in the $\th$ variable, one has that the derivative of $\g_\xi$ is 
\begin{equation}\label{parametrization.derivative}\g_\xi'(\th)=\begin{pmatrix}(\g_\xi')_1(\th) \\ (\g_\xi')_2(\th)\end{pmatrix}=e^{\xi(\th)}\Big\{\xi'(\th)\begin{pmatrix}\cos\th \\ \sin\th
\end{pmatrix} - \begin{pmatrix}\sin\th \\ -\cos\th
\end{pmatrix}\Big\}.
\end{equation}
If we equip $\R^2\setminus\{0\}$ with the Euclidean metric $g_{Euc}:=dx\otimes dx + dy\otimes dy$, then we can do the pullback of it through $\g_\xi$ to induce a metric $g$ on $\pa\Om_h$, whose local representation on $\T^1$ is
\begin{equation}\begin{aligned}g(\th)&=((\g_{\xi})_*g_{Euc})(\th)
\\&:=d_\th(\g_\xi)_1\otimes d_\th(\g_\xi)_2 + d_\th(\g_\xi)_2\otimes d_\th(\g_\xi)_2
\\&=[(\g_\xi')_1^2(\th)+(\g_\xi')_2^2(\th)]d\th\otimes d\th
\\&=e^{2\xi(\th)}(1+\xi'^2(\th))d\th\otimes d\th
\\&=:g_{\th\th}(\th)d\th\otimes d\th.
\end{aligned}\end{equation}
Let us also call 
\begin{equation}\label{metric.inverse}g^{\th\th}(\th):=(g_{\th\th}(\th))^{-1}=e^{-2\xi(\th)}(1+\xi'^2(\th))^{-1}.
\end{equation}
The following lemma provides a representation on $\T^1$ of $\grad_{\S^1}$, $\nu_{h}$ and some correlated quantities:

\begin{lemma}\label{derivative.change.strip}The following facts hold.

\medskip

$(i)$ For any $f\colon\,\R^2\setminus\{0\}\to\R$, let us call $\tilde{f}:=f\circ\phi$ its representation on $\mS$. Then, 
\begin{align}&\grad_{\mS}\tilde f(\rho,\th):=\begin{pmatrix}\pa_\rho\tilde{f}(\rho,\th) \\ \pa_\th\tilde{f}(\rho,\th)\end{pmatrix} = e^\rho\begin{pmatrix}\cos\th & \sin\th \\ -\sin\th & \cos\th
\end{pmatrix}\grad f\circ\phi(\rho,\th) ,\label{grad.rho.th.to.grad.true}
\\&\grad f\circ\phi(\rho,\th) = e^{-\rho}\Big\{\pa_\rho\tilde{f}(\rho,\th)\begin{pmatrix}\cos\th \\ \sin\th
\end{pmatrix} - \pa_\th\tilde{f}(\rho,\th)\begin{pmatrix}\sin\th \\ -\cos\th\end{pmatrix}\Big\}, \label{grad.true.to.rho.th}
\\&\Delta f\circ\phi(\rho,\th)=e^{-2\rho}\Delta_{\mS}\tilde f (\rho,\th). \label{Laplace.Transformation}
\end{align}
Here, we are denoting by $\Delta_{\R^2},\,\Delta_{\mS}$ respectively the Laplace operators on $\R^2,\,\mS$.

\medskip

$(ii)$ For any $f,g\colon\,\S^1\to\R$, let us call $\tilde{f}:=\mE_0f\circ\phi,\,\tilde{g}:=\mE_0g\circ\phi$ their respective representations on $\mS$. Then, $\tilde f(\th,\rho)=\tilde f(\th)$, $\tilde g(\th,\rho)=\tilde g(\th)$ for any $(\th,\rho)\in\mS$ and
\begin{align}&\grad_{\S^1}f\circ\phi(0,\th)=\tilde{f}'(\th)\begin{pmatrix}-\sin\th \\ \cos\th \label{grad.true.to.th.S1}
\end{pmatrix},
\\&\la\grad_{\S^1}f\circ\phi(0,\th),\grad_{\S^1}g\circ\phi(0,\th)\ra = \tilde{f}'(\th)\tilde{g}'(\th), \label{scalar.products.th}
\\&|\grad_{\S^1}f\circ\phi(0,\th)|^2=\tilde{f}'^2(\th),\label{norm.grad.S}
\\&\sqrt{(1+h\circ\phi(0,\th))^2 + |\grad_{\S^1}h\circ\phi(0,\th)|^2}=e^{\xi(\th)}\sqrt{1+\xi'^2(\th)}, \label{J.th.}
\\&\mM f\circ\phi(0,\th):=\la\grad_{\S^1}f,\mJ x\ra\circ\phi(0,\th)=\tilde{f}'(\th). \label{mM.op.def}
\end{align}
Here, 
\begin{equation}\label{J.matrix.symplectic}\mJ:=\begin{pmatrix}0 & -1 \\ 1 & 0
\end{pmatrix}.\end{equation}
$(iii)$ Let $\nu_{h}$ be the outer normal vector field of $\pa\Om_h$, let $\nu_\xi:=\nu_{h}\circ\g_\xi$ be its local representations on $\mS$. Then,
\begin{equation}\label{nu.representation}\nu_\xi(\th)=(1+\xi'^2(\th))^{-\frac12}\Big\{\begin{pmatrix}\cos\th \\ \sin\th
\end{pmatrix} + \xi'(\th)\begin{pmatrix}\sin\th \\ -\cos\th
\end{pmatrix} \Big\} = -|\g_\xi'(\th)|^{-1} \mJ\g_\xi'(\th)
\end{equation}
$(iv)$ Let $\Phi\colon\,\Om_h\to\R$ be the solution of the problem
\begin{equation}\label{bp.on.h}\begin{aligned}&\Delta\Phi=0\qquad\qquad\,\,\text{in}\,\,\Om_h,
\\&\Phi\circ\g_h=\psi\qquad\quad\text{on}\,\,\S^2,
\end{aligned}\end{equation}
and let 
\begin{equation}\label{G.space}G(h)\psi:=\la\grad\Phi\circ\g_h,\nu_{h}\circ\g_h\ra,
\end{equation}
be the Dirichlet-Neumann operator.
Let us suppose that $\grad\Phi$ is bounded in $\Om_h$, and let $\tilde{\Phi}:=\Phi\circ\phi$ the representation of $\Phi$ on $\mS$. Then, for all $\th\in\T^1$ we have
\begin{equation}\label{G.into.G.on.T1}G(h)\psi\circ\phi(\th,0) = e^{-\xi(\th)}(1+\xi'^2(\th))^{-\frac12}\bar{G}(\xi)\chi(\th),\end{equation}
where $\chi:=\Phi\circ\g_\xi$ and $\bar{G}(\xi)\chi$ is the Dirichlet-Neumann operator on the torus at infinite depth, namely
\begin{equation}\label{Dirichlet.Neumann.on.torus.infinite.depth}\bar{G}(\xi)\chi=\sqrt{1+\xi'^2}\cdot\la\grad\tilde{\Phi}\circ\g_\xi,\nu_{\xi}\ra,
\end{equation}
where $\tilde{\Phi}\colon\,\bar{\Om}_\xi\to\R$ is the solution to the Dirichlet problem
\begin{equation}\begin{aligned}\label{generic.bp}&\Delta\tilde\Phi=0\qquad\qquad\qquad\text{in}\,\,\Om_\xi,
\\&\tilde\Phi\circ\g_\xi=\chi\qquad\qquad\,\,\,\,\text{on}\,\,\T^1,
\\&\pa_\rho\tilde\Phi(\rho,\th)\to0\qquad\quad\,\text{as}\,\,\rho\to-\infty
\end{aligned}\end{equation}  
$(v)$ Let $\Phi$ as in point $(iv)$ and let $\Psi\colon\,\Om_h\to\R$ such that $\grad\Psi=\mJ\grad\Phi$.
Let 
\begin{equation}\label{Neumann.Dirichlet}K(h)\psi:=\Psi\circ\g_h.
\end{equation}
Then, $\Psi$ solves the Neumann problem
\begin{align}&\Delta\Psi=0\qquad\qquad\qquad\qquad\,\,\,\text{in}\,\,\Om_h,   \notag
\\&G(h)[K(h)\psi]=-\frac{\mM\psi}{J}\qquad\text{on}\,\,\S^1.\label{Dirichlet.Neumann.acting.on.Neumann.Dirichlet} 
\end{align}
Furthermore,
\begin{equation}\label{Neumann.Dirichlet.on.torus}K(h)\psi\circ\phi(0,\th)=\Psi\circ\g_\xi(\th) =:\bar{K}(\xi)\chi,
\end{equation}
where $\chi$ is defined in point $(iv)$, and $\Psi$ solves also
\begin{align}&\Delta_S\Psi=0\qquad\qquad\qquad\qquad\,\,\,\,\text{in}\,\,\Om_\xi,   \notag
\\&\bar G(h)[\bar K(h)\chi]=-\chi'\qquad\qquad\text{on}\,\,\T^1.\label{Dirichlet.Neumann.acting.on.Neumann.Dirichlet.torus} 
\end{align}
Furthermore, it holds the identity
\begin{equation}\label{G.is.K.derivative}(\bar K(\xi)\chi)'=\bar G(\xi)\chi.
\end{equation}
\end{lemma}

\begin{proof} For the points $(i)-(iv)$, see \cite{L}. As for $(v)$, let us start by proving \eqref{Dirichlet.Neumann.acting.on.Neumann.Dirichlet}. By $\grad\Psi=\mJ\grad\Phi$, by the fact that $\mJ\nu_h\circ\g_h$ is tangential to $\pa\Om_h$, and by \eqref{nu.h}, \eqref{eq.1.in.lem:formulas}, \eqref{mM.op.def}, we have for all $x\in\S^1$
\begin{align*}G(h)[K(h)\psi](x)
&=\la\grad\Psi,\nu_{h}\ra\circ\g_h(x)
\\&=\la\mJ\grad\Phi,\nu_{h}\ra\circ\g_h(x)
\\&=-\la\grad\Phi,\mJ\nu_{h}\ra\circ\g_h(x)
\\&=-\la\grad_{\pa\Om_h}\Phi,\mJ\nu_{h}\ra\circ\g_h(x)
\\&=-\la\frac{\grad_{\S^1}\psi}{1+h},\frac{(1+h)\mJ x - \mJ\grad_{\S^1}h}{J}\ra
\\&=-\frac{\mM\psi}{J}.
\end{align*}
This proves \eqref{Dirichlet.Neumann.acting.on.Neumann.Dirichlet}. We notice that $\grad_{\S^1}\psi\perp\mJ\grad_{\S^1}h$ since $\grad_{\S^1}\psi$ is a tangent vector while $\mJ\grad_{\S^1}h$ is normal. The proof of \eqref{Neumann.Dirichlet.on.torus} is a direct check and \eqref{Dirichlet.Neumann.acting.on.Neumann.Dirichlet.torus} comes by using \eqref{Laplace.Transformation}, \eqref{mM.op.def} and \eqref{G.into.G.on.T1}. The identity \eqref{G.is.K.derivative} holds because by \eqref{Neumann.Dirichlet.on.torus}, by $\grad\Psi=\mJ\grad\Phi$ and by \eqref{nu.representation},
\begin{align*}(\bar K(\xi)\chi)'
&=\la\grad\Psi\circ\g_\xi,\g_\xi'\ra=\la\grad\Phi\circ\g_\xi,-\mJ\g_\xi'\ra=\bar{G}(\xi)\chi.
\end{align*}
\end{proof}
We use the metric $g_h$ and the representation \eqref{nu.representation} to get the mean curvature of $\pa\Om_h$:
\begin{equation}\label{curvature}(H(h))(\th):=H_\Om(\g_h(\phi(\th)))=-g^{\th\th}(\th)\la \g_\xi''(\th), \nu_\xi(\th)\ra.
\end{equation}
Geometrically speaking, it is nothing less but the opposite of the trace of the Second Fundamental Form $II(\th):=\la\g_\xi''(\th), \nu_\xi(\th)\ra d\th\otimes d\th$ of $\pa\Om_h$ through the metric $g$, see for instance \cite{GHL}. By a direct computation, one can check that
\begin{equation}\begin{aligned}\label{mean.curvature}H(h)&=-(1+\xi'^2))^{-\frac32}e^{-\xi}\{\xi'' - 1-\xi'^2\}
\\&=e^{-\xi}\Big[ \frac{1}{\sqrt{1+\xi'^2}} - \Big(\frac{\xi'}{\sqrt{1+\xi'^2}}\Big)' \Big].
\end{aligned}\end{equation}
One can notice that in the last identity, the first term is the curvature of the graph of $\xi(\th)$ up to the conformal factor $e^{-\xi(\th)}$. Let us notice that $H(0)=1$, which corresponds to the curvature of the unit circle.

We introduce here the functional setting we will work with. 
All the functions $f\in L^2$ can be expanded in Fourier series along the $L^2-$orthonormal basis $\{\ph_{\ell,m}\colon\,(\ell,m)\in\mT\}$, with $\mT:=\{(0,0)\}\cup(\N\times\{-1,1\})$:
\begin{equation*}f=\sum_{(\ell,m)\in\mT}f_{\ell,m}\ph_{\ell,m};
\end{equation*}
here,
\begin{equation}\label{Fourier.basis}\ph_{\ell,m}(\th):=\begin{cases}C_{0} & \text{if}\qquad \ell=m=0,
\\C_\ell\cos(\ell\th)& \text{if}\qquad \ell>0,m=1,
\\C_\ell\sin(\ell\th)& \text{if}\qquad \ell>0,m=-1,
\end{cases}
\end{equation}
We recall the Sobolev spaces $H^s$ and $H_0^s$, with $s\in\R$:
\begin{align*}&H^s=\Big\{f\in L^2\colon\,\|f\|_{s}^2=\sum_{(\ell,m)\in\mT}(1+|\ell|^{2s})|f_{\ell,m}|^2<+\infty\Big\},
\\&H_0^s=\{f\in H^s\colon\,f_{0,0}=0\}.
\end{align*}
In what follows, we denote by $(H^s)^2:=H^s\times H^s$.

We also recall the class of analytic functions with exponential decay rate $\mathfrak{s}\ge0$ on torus:
\begin{equation}H^{\mathfrak{s},s}:=\Big\{f\in L^2\colon\,\|f\|_{\mathfrak{s},s}^2:=\sum_{(\ell,m)\in\mT}e^{2\mathfrak{s}\ell}(1+|\ell|^{2s})|f_{\ell,m}|^2<+\infty \Big\},
\end{equation}
and we call $H_0^{\mathfrak{s},s}$ the Sobolev space made of zero-average functions (or simply, $H_0^s$ if $\mathfrak{s}=0$).

In the end, for $\kappa\in\N$, we define more in general Sobolev spaces of functions having $\kappa-$fold symmetry:
\begin{align}H_\kappa^{\mathfrak{s},s}
:&=\Big\{f\in H^{\mathfrak{s},s}\colon\,f\Big(\th 
 + \frac{2\pi}{\kappa}\Big)=f(\th)\Big\},\qquad H_{\kappa,0}^{\mathfrak{s},s}:= \{f\in H_\kappa^{\mathfrak{s},s}\colon\,\la f,1\ra_{L^2}=0\}. \label{k.fold.sobolev}
\end{align}

We mention here some useful properties of the Dirichlet-Neumann operator $\tilde G$ mentioned in \eqref{Dirichlet.Neumann.on.torus.infinite.depth}.

\begin{lemma}[from \cite{Lannes, Lannes1, BMV}]\label{lemma:G}
Calling $\dot{H}$ the space of equivalence classes of Sobolev functions differing by a constant, one has the following facts.

\medskip

$(i)$(Low norm estimate) For all fixed $\xi\in H^{s_0}$ with $s_0>\frac32$, one has that $\bar G(\xi)[\cdot]$ is a bounded operator from $\dot H^{\frac12}$ to $H^{-\frac12}$, that is, 
\begin{equation}\label{G.low.norm.estimate}\|\bar G(\xi)\chi\|_{-\frac12}\lesssim_\xi \|\chi\|_{\frac12},
\end{equation}
where $C(\xi)$ is a constant depending on $\xi$.

\medskip

$(ii)$(Symmetricity and nonnegativity) The operator $\chi\in \dot H^{\frac12}\to \bar G(\xi)\chi\in H^{-\frac12}$ is $L^2-$symmetric and nonnegative for all fixed $\xi\in W^{1,\infty}$, that is, for all $\chi_1,\chi_2\in \dot H^{\frac12}$,
\begin{align}&\la \bar G(\xi)\chi_1,\chi_2\ra_{L^2} = \la \bar G(\xi)\chi_2,\chi_1\ra_{L^2},
\label{G.is.symmetric}
\\&\la \bar G(\xi)\chi_1,\chi_1\ra_{L^2}\ge0.\label{G.is.nonnegative}
\end{align}

$(iii)$(Kernel of $\bar G(\xi)[\cdot]$) For some $\delta>0$ small enough and for $\|\xi\|_{s_0}<\delta$, one has that $\psi\in\ker(\bar G(h))$ if and only if $\psi$ is a constant function. 

\medskip

$(iv)$(Invariance by inversion of sign) Let $\xi\in H^{s_0}$ with $s_0>\frac32$, $\chi\in \dot H^{\frac12}$, and let us call $\iota$ the operator of inversion of sign, that is, for all functions $f\colon\,\T^1\to\R$ and $\th\in\T^1$, $(\iota f)(\th):=f(-\th)$. Then,
\begin{equation}\label{G.inversion}\bar G(\iota\xi)[\iota\chi]=\iota(\bar G(\xi)\chi).
\end{equation}

$(v)$(Translation invariance of $\bar G$) Let $\mT_\alpha$ be the translation operator, that is, $(\mT_\alpha f)(\th):=f(\th+\alpha)$ for any $f\colon\,\T^1\to\R$, $\th\in\T^1$ and $\alpha\in\T^1$. Then, for any $\alpha\in\T^1$, $\xi\in H^{s_0}$ with $s_0>\frac32$ and $\chi\in \dot H^{\frac12}$,
\begin{equation}\label{G.translation.inv}\mT_\alpha(\bar G(\xi)\chi)=\bar G(\mT_{\alpha}\xi)[\mT_{\alpha}\chi].
\end{equation}
$(vi)$(High norm estimates) Let $s_0>\frac12$, let $s\ge0$ and $\xi\in H^{s+\frac12}\cap H^{s_0+1}$. Then, 
\begin{align}&\|\bar G(\xi)\chi\|_{s-\frac12}\lesssim_{\|\xi\|_{s+\frac12},\|\xi\|_{s_0+1}}\|\chi'\|_{s-\frac12},\qquad\forall0\le s\le s_0+\frac32,\,\forall\chi\in \dot H^{s-\frac12}, \label{G.High.norm.est.low}
\\& \|\bar G(\xi)\chi\|_{s-\frac12}\lesssim_{\|\xi\|_{s_0+2}}\|\chi'\|_{s-\frac12} + \|\xi\|_{s+\frac12}\|\chi'\|_{s_0+1},\qquad\forall0\le s\le s_0+\frac32,\,\forall\chi\in \dot H^{s-\frac12}\label{G.High.norm.est.high}
\end{align}
$(vii)$(Analiticity in finite regularity) Let $s_0>\frac12$, $0\le s\le s_0+\frac12$ and $\|\xi\|_{s_0 + 1}<\delta_0$ for a suitable $\d_0\in(0,1)$. Then, the map $\xi\in H^{s_0+1}\to\bar G(\xi)\in\mathcal{L}(H^{s+\frac12},\dot H^{s-\frac12})$ is analytic.

\medskip

$(viii)$(Shape derivative) Let, $s_0>\frac12$, $\chi\in H^{\frac32}$, $\xi\in \dot H^{s_0+1}$. Then, for all directions $\hat\xi\in H^{s_0+1}$ corresponding to the $\xi-$variable,
\begin{equation}\label{shape.derivative}d(\bar G(\xi)\hat\xi)\chi=-\bar G(\xi)[B\hat\xi] - (V\hat\xi)', \end{equation}
where
\begin{equation}\label{B.V.}B:=\frac{\bar G(\xi)\chi + \xi'\chi'}{1+\xi'^2},\qquad V:=\chi'-B\xi'.
\end{equation}
$(ix)$(Analiticity for analytic regularity) Let $\mathfrak{s}>0$, let $s,s_0$ such that $s,s_0+\frac12\in\N$ and $s-\frac32\ge s_0>1$. Then, there exists $\e_0:=\e_0(s)>0$ such that for all $\xi$ for which $\|\xi\|_{\mathfrak{s},s_0+\frac32}<\e_0$, the map $\xi\in H^{\mathfrak{s},s}\to \bar G(\xi)\in\mathcal{L}(H^{\mathfrak{s},s},\dot H^{\mathfrak{s},s-1})$ is analytic and in particular satisfies the tame estimate
\begin{equation}\label{analiticity.for.analytic.functions}\|\bar G(\xi)\chi\|_{\mathfrak{s},s-1}\lesssim_s\|\chi\|_{\mathfrak{s},s} + \|\xi\|_{\mathfrak{s},s}\|\chi\|_{\mathfrak{s},s_0+\frac32}.
\end{equation}

\end{lemma}

We point out that by points $(iii),\,(vi),\,(vii)$ and $(ix)$, for all $\|\xi\|_{s_0+1}<\delta_0$ one has that $\bar G(\xi)$ is a homeomorphism from $H_0^{s-\frac12}$ onto its image. As a result, recalling point $(v)$ of Lemma \ref{derivative.change.strip}, we have $\bar K(\xi)\chi=-\bar G(\xi)^{-1}\chi'$ if we choose $\Psi$ such that it has zero average on $\pa\Om_h$. Therefore, $\bar K(\xi)$ inherits all the regularity properties of $\bar G(\xi)$.

Moreover, since $\bar G$ satisfies the translation invariance property \eqref{G.translation.inv}, $K$ does as well.

\section{The capillary drop equation with constant vorticity}

\subsection{Craig-Sulem formulation of the capillary drop equations}

In this section, we derive the Craig-Sulem equations for 2D drops with constant vorticity, we show their equivalence with the free-boundary equations \eqref{div.eq.01}-\eqref{kin.eq.01} and finally we rewrite them on the torus $\T^1$. We always suppose that all the functions into play are smooth enough to get classical solutions.

Let us recall the equations \eqref{div.eq.01}-\eqref{kin.eq.01}:
\begin{align*}
&\div u 
 = 0 \quad \qquad\qquad\,\,\,\,\,\quad\text{in} \ \Om_t,
\\ 
&
\text{curl}\,u 
 = \a_0 \quad \qquad\qquad\quad\text{in} \ \Om_t,
\\
&
\pa_t u + u \cdot \grad u + \grad p 
 = 0 \quad \text{in} \ \Om_t, 
\\
&
p  = \sigma_0 H_{\Om_t} \quad\qquad\qquad\quad \text{on} \ \pa \Om_t, 
\\
&
V_t  = \la u , \nu_{t} \ra \quad\qquad\quad\,\,\,\,\, \,\,\,\,\, \text{on} \ \pa \Om_t.
\end{align*}
We recall that $\Om_t$ is a star-shaped domain whose boundary is of the form
\begin{equation}\label{pa.Om.t}\pa\Om_t:=\{(1+h(t,x))x\,\colon\,x\in\S^1\},\qquad h(t,\cdot)\colon\,\S^1\to(-1,+\infty),\qquad t\in\R
\end{equation}
for $\|h\|_{W^{1,\infty}}$ small, and so the map $\g_t\colon\,\S^1\to\pa\Om_t$, which is given by 
\begin{equation}\label{Om.t.par.}\g_t(x):=(1+h(t,x))x,\end{equation}
for all $t\in\R,\,x\in\S^1$ (see \eqref{gamma.diff}), is a diffeomorphism between $\S^1$ and $\pa\Om_t$. Let us define the normal boundary velocity as
\begin{equation}\label{boundary.velocity.Om.t} V_t(\g_t(x)):=\la\pa_t\g_t(x), \nu_{t}\ra,
\end{equation}
where $\nu_t$ is the outer normal vector field of $\pa\Om_t$.

Since $\pa\Om_t$ is star-shaped, by \eqref{div.eq.01} the $1-$form $\om:=u_2\,dx - u_1\,dy$
is exact, so there exists a function $\Psi\colon\,\Om_t\to\R$ (the \emph{stream function}), which satisfies  
\begin{equation}\label{potential.u}u=-\mJ\grad\Psi_1\quad\text{in}\,\Om_t.\end{equation}
Here, $\mJ$ is defined in \eqref{J.matrix.symplectic}.
Applying curl$\,X=\pa_{x_2}X_1 - \pa_{x_1}X_2$ to both sides of \eqref{potential.u}, by \eqref{curl.eq.01} we have
\begin{equation}\Delta\Psi_1=\a_0.
\end{equation}
By linearity of the Laplace operator, we can decompose $\Psi_1=\Psi + \frac{\a_0}{4}(x_1^2 + x_2^2)$, where we call $\Psi$ the \emph{harmonic conjugate}, which satisfies
\begin{equation}\label{harmonic.conjugate.is.harmonic}\Delta\Psi=0.
\end{equation}
Then, \eqref{potential.u} turns into 
\begin{equation}\label{u.decomposition.harmonic.conjugate}u=-\mJ\grad\Psi - \frac{\a_0}{2}\mJ x.
\end{equation}
Now, observing that $\div(u+\frac{\a_0}{2}\mJ x)=0$ and $\text{curl}(u + \frac{\a_0}{2}\mJ x)=0$, by star-shapedness of $\Om_t$ we deduced that the $1-$form $\om_{\a_0}=(u_1-\frac{\a_0}{2}x_2)\,dx_1 + (u_2+\frac{\a_0}{2}x_1)\,dx_2$ is exact, and so there exists a function $\Phi\colon\,\Om_t\to\R$ (the \emph{potential}) satisfying
\begin{equation}\label{velocity.potential}-\mJ\grad\Psi=\grad\Phi,
\end{equation}
or else, $\grad\Psi=\mJ\grad\Phi$. As a result, \eqref{u.decomposition.harmonic.conjugate} turns into \eqref{u.decomposition.potential}:
\begin{equation*}u=\grad\Phi - \frac{\a_0}{2}\mJ x.
\end{equation*}
We remark that $\Phi$ is harmonic in $\Om_t$, and we suppose it to solve the Dirichlet problem
\begin{equation}\label{Laplac.problem}\begin{aligned}&\Delta\Phi(t,\cdot)=0\qquad\qquad\qquad\qquad\text{in}\,\,\Om_t,
\\&\Phi(t,\cdot)\circ\g_t=\psi(t,\cdot)\qquad\qquad\quad\text{on}\,\,\S^1.
\end{aligned}\end{equation}
For such problem, we define the Dirichlet-Neumann operator $G(h)\psi$ as done in \eqref{G.space}.

As for the harmonic conjugate $\Psi$, the boundary value of $\Psi$ is given by the operator $(K(h)\psi)(t,\cdot)=\Psi(t,\cdot)\circ\g_h$, and by Lemma \ref{derivative.change.strip}, point $(v)$, $\Psi$ is a solution of the Neumann problem
\begin{align}&\Delta\Psi(t,\cdot)=0\qquad\qquad\qquad\qquad\,\,\,\text{in}\,\,\Om_t,   \notag
\\&\la\grad\Psi,\nu_{t}\ra\circ\g_t=-\frac{\mM\psi}{J}\,\,\,\quad\quad\quad\,\,\,\,\text{on}\,\,\S^1.\label{Neumann.problem.Psi}
\end{align}
We notice that $\Psi$ is defined up to additive constants, and thus $K(h)\psi$ does.

Then, we use the notation
\begin{equation}\label{equivalence.relation}f\sim g\iff f-g\quad\text{does not depend on}\, x.
\end{equation}
We now derive the Craig-Sulem formulation for \eqref{div.eq.01}-\eqref{kin.eq.01}
\begin{theorem}\label{thm:ww.eq.}In the notations above, the equations \eqref{div.eq.01}-\eqref{kin.eq.01} are equivalent to the Craig-Sulem equations \eqref{h.eq.}-\eqref{psi.eq.}, that is,
\begin{align*}& \pa_t h = \frac{J}{1 + h} \, G(h)\psi + \frac{\a_0}{2}\mM h,
\\
& \pa_t \psi = 
 \frac12 \Big( G(h)\psi 
+ \frac{\la \grad_{\S^1} \psi , \grad_{\S^1} h \ra}{J(1+h)} \Big)^2
- \frac{|\grad_{\S^1} \psi|^2}{2(1+h)^2} 
  - \s_0H(h)  \\
  &\qquad + \frac{\a_0}{2}\mM\psi + \frac{\a_0^2}{8}(1+h)^2 + \a_0 K(h)\psi + \s_0 - \frac{\a_0^2}{8} .
\end{align*}

\end{theorem}

\begin{proof}
Let $\Om_t\subset\R^2$ and $u\colon\,\Om_t\to\R^2$ the solution to \eqref{div.eq.01}-\eqref{kin.eq.01} hold. Let us start by getting \eqref{h.eq.}.
As in \eqref{nu.h}, the explicit formula for $\nu_{t}$ is
\begin{equation}\label{normal.Om.t}\nu_{t}\circ\g_t(x)=\frac{(1 + h(t,x)) x - \grad_{\S^1} h(t,x)}{J},\qquad x\in\S^1.\end{equation}
Then, on one hand by \eqref{Om.t.par.}, \eqref{boundary.velocity.Om.t} and \eqref{normal.Om.t} we get
\begin{equation}\label{boundary.velocity}V_t\circ\g_t=\frac{1+h}{J}\pa_t h,
\end{equation}
and on the other hand, by \eqref{kin.eq.01} and \eqref{u.decomposition.potential} it holds that
\begin{equation}\label{def.boundary.vel.}\begin{aligned}V_t\circ\g_t
&=\la u(t,\cdot),\nu_{t}\ra\circ\g_t
\\&= G(h)\psi -\frac{\a_0(1+h)}{2J}\la\mJ x,(1+h)x - \grad_{\S^1}h\ra \\&=  G(h)\psi + \frac{\a_0(1+h)}{2J}\mM h
\end{aligned}
\end{equation}
putting together \eqref{boundary.velocity} and \eqref{def.boundary.vel.}, we obtain \eqref{h.eq.}.

Now, we want to get \eqref{psi.eq.}.
By \eqref{u.decomposition.potential} and the identity $u\cdot\grad u=\grad(\frac12|u|^2) - (\text{curl}\,u)\mJ u$, we have 
\begin{align*}&\pa_t u=\pa_t\Phi
\end{align*}
and
\begin{align*}u\cdot\grad u
&=\grad\Big(\frac12|\grad\Phi - \frac{\a_0}{2}\mJ x|^2\Big) - \alpha_0\grad\Big(\Psi + \frac{\a_0}{4}|x|^2\Big)
\\&=\grad\Big(\frac12|\grad\Phi|^2 - \frac{\a_0}{2}\la\grad\Phi,\mJ x\ra - \frac{\a_0^2}{8}|x|^2 - \a_0\Psi\Big)
\end{align*}
We then get the Bernoulli identity \eqref{Ber.law}:
\begin{equation*}\pa_t\Phi + \frac12|\grad\Phi|^2 - \frac{\a_0}{2}\la\grad\Phi,\mJ x\ra - \frac{\a_0^2}{8}|x|^2 - \a_0\Psi
+ p \sim0 \qquad\text{in}\,\Om_t.
\end{equation*}
By continuity this equation still holds on the boundary $\pa\Om_t$, that is,
\begin{equation}\label{Ber.law}\pa_t\Phi + \frac12|\grad\Phi|^2 - \frac{\a_0}{2}\la\grad\Phi,\mJ x\ra - \frac{\a_0^2}{8}|x|^2 - \a_0\Psi
+ \s_0 H_{\Om_t} \sim 0 \qquad\text{on}\,\pa\Om_t,
\end{equation}
and therefore
\begin{align}\pa_t\Phi\circ\g_t &+ \frac12|\grad\Phi\circ\g_t|^2 + \s_0 H(h) \notag
\\&- \frac{\a_0}{2}\la\grad\Phi,\mJ x\ra\circ\g_h - \frac{\a_0^2}{8}(1+h)^2 - \a_0K(h)\psi \sim 0\qquad\text{on}\,\S^1,\label{Bernoulli}
\end{align}

For all $t\in\R,\,x\in\S^1$ one has
\begin{align}&\pa_t\psi=\pa_t(\Phi(t,\cdot)\circ\g_t)=\pa_t\Phi(t,\cdot)\circ\g_t + \langle\grad\Phi(t,\cdot)\circ\g_t,\pa_t\g_t\rangle, \label{pa.t.Phi}
\\&|\grad\Phi(t,\cdot)\circ\g_t|^2 = |\grad_{\pa\Om_t}\Phi(t,\cdot)\circ\g_t|^2 + |\grad_{\nu_{t}}\Phi(t,\cdot)\circ\g_t|^2. \label{grad.Phi}
\end{align}
By using \eqref{Om.t.par.}, \eqref{h.eq.}, \eqref{normal.gradient}, \eqref{tangential.gradient} and \eqref{eq.1.in.lem:formulas} we observe that
\begin{equation}\label{grad.Phi.gamma.t}\begin{aligned}\langle\grad\Phi(t,\cdot)\circ\g_t,\pa_t\g_t\rangle&=\frac{J(h)}{1+h}G(h)\psi\cdot\{\la\nabla_{\pa\Omega_t}\Phi(t,\cdot)\circ\g_t,x\ra + \la\nabla_{\nu_{t}}\Phi(t,\cdot)\circ\g_t,x\ra\}
\\&=\frac{J(h)}{1+h}G(h)\psi\cdot\Big\{\frac{\la\grad_{\S^1}\psi,\grad_{\S^1}h\ra}{J(h)} + G(h)\psi\cdot\la\nu_{t}\circ\g_t,x\ra\Big\}
\\&=\frac{J(h)}{1+h}G(h)\psi\cdot\Big\{\frac{\la\grad_{\S^1}\psi,\grad_{\S^1}h\ra}{J(h)} + G(h)\psi\cdot\frac{1+h}{J(h)}\Big\}
\\&=\frac{\la\grad_{\S^1}\psi,\grad_{\S^1}h\ra\cdot G(h)\psi}{1+h} + (G(h)\psi)^2
\end{aligned}\end{equation}
in \eqref{grad.Phi}, one gets $\grad_{\nu_{t}}\Phi(t,\cdot)\circ\g_t=(G(h)\psi)\nu_{t}\circ\g_t$ and so 
\begin{equation}\label{norm.grad.nu.Phi}|\grad_{\nu_{t}}\Phi(t,\cdot)\circ\g_t|^2=(G(h)\psi)^2
\end{equation}
while by \eqref{eq.1.in.lem:formulas}, we have
\begin{align}\grad_{\pa\Om_t}\Phi(t,\cdot)\circ\g_t&=\frac{\nabla_{\S^1} \psi}{1+h} + \frac{\la\nabla_{\S^1}\psi, \nabla_{\S^1} h \ra}{(1+h)J(h)}\nu_{t}\circ\gamma_t. \label{grad.tang.Phi}
\end{align}
Therefore,
\begin{equation}\label{norm.grad.tan.Phi}|\grad_{\pa\Om_t}\Phi(t,\cdot)\circ\g_t|^2=\frac{|\grad_{\S^1}\psi|^2}{(1+h)^2} - \frac{|\la\nabla_{\S^1}\psi, \nabla_{\S^1} h \ra|^2}{(1+h)^2\cdot J^2(h)}.
\end{equation}
Using \eqref{pa.t.Phi}, \eqref{grad.Phi}, \eqref{grad.Phi.gamma.t}, \eqref{norm.grad.nu.Phi}, \eqref{grad.tang.Phi} and \eqref{norm.grad.tan.Phi}, the first line of \eqref{Bernoulli} becomes
\begin{align}\pa_t\Phi\circ\g_t &+ \frac12|\grad\Phi\circ\g_t|^2 + \s_0 H(h) \notag
\\&=\pa_t\psi-\frac12 \Big( G(h)\psi 
+ \frac{\la \grad_{\S^1} \psi , \grad_{\S^1} h \ra}{J(1+h)} \Big)^2
+ \frac{|\grad_{\S^1} \psi|^2}{2(1+h)^2} 
  + \s_0H(h). \label{first.line}
\end{align}
We are left to make the term $- \frac{\a_0}{2}\la\grad\Phi,\mJ x\ra\circ\g_h$ explicit. However, the same computation done in Lemma \ref{derivative.change.strip} gives 
\begin{equation}-\frac{\a_0}{2}\la\grad\Phi,\mJ x\ra\circ\g_h= -\frac{\a_0}{2}\mM\psi.
\end{equation}
Putting together this identity with \eqref{first.line}, we get 
\begin{align*}\pa_t \psi - 
& \frac12 \Big( G(h)\psi 
+ \frac{\la \grad_{\S^1} \psi , \grad_{\S^1} h \ra}{J(1+h)} \Big)^2
+ \frac{|\grad_{\S^1} \psi|^2}{2(1+h)^2} 
  + \s_0H(h)  \\
  &\qquad - \frac{\a_0}{2}\mM\psi - \frac{\a_0^2}{8}(1+h)^2 - \a_0 K(h)\psi \sim0
  \end{align*}
We can then choose a representative $\psi$ such that this equation is exactly solved for $(h,\psi)=(0,0)$ (imposing $K(0)0=0$). In such a case, at the right-hand side we would have $\s_0-\frac{\a_0^2}{8}$, and so we choose $\psi$ for which \eqref{psi.eq.} holds.

To prove that one gets \eqref{div.eq.01}-\eqref{kin.eq.01} from \eqref{h.eq.}-\eqref{psi.eq.}, let us define the set $\Om_t$, with boundary $\pa\Om_t$, as in \eqref{pa.Om.t}, and $\Phi(t,\cdot)$ in $\Om_t$ 
as the solution of the Laplace problem \eqref{Laplac.problem}.
Then, the vector field
\begin{equation*}u(t,x):=\grad\Phi(t,\cdot) - \frac{\a_0}{2}\mJ x\end{equation*}
satisfies both $\div\,u=0$ and $\text{curl}\,u=\a_0$ in $\Om_t$. 
By considering also \eqref{h.eq.}, 
equation \eqref{kin.eq.01} is also satisfied. 
From \eqref{psi.eq.}, using \eqref{h.eq.}, we obtain  \eqref{Bernoulli}. 
Now we define $p$ on the closure of $\Om_t$ as 
\begin{equation} \label{def.tilde.p.using.dyn.eq.02}
p := - \pa_t \Phi - \frac12 |\grad \Phi|^2  
\quad \ \text{in } \overline{\Om_t} = \Om_t \cup \pa \Om_t.
\end{equation}
Then the dynamics equation \eqref{Ber.law} in the open domain $\Om_t$ trivially holds. 
From \eqref{def.tilde.p.using.dyn.eq.02} at the boundary $\pa \Om_t$ 
and \eqref{Bernoulli} (which is an identity for points of the boundary $\pa \Om_t$)
we deduce that $p = \s_0 H_{\Om_t}$ on $\pa \Om_t$, that is \eqref{pressure.eq.01}.

\end{proof}

We finally write the equations \eqref{h.eq.}-\eqref{psi.eq.} on the torus $\T^1$.

\begin{theorem} \label{thm:ww.on.torus} In the notations above, the equations \eqref{h.eq.}-\eqref{psi.eq.} can be written on $(t,\th)\in\R\times\T^1$ as \eqref{xi.eq}-\eqref{chi.eq}, that is,
\begin{align*}&\pa_t\xi=e^{-2\xi}[\bar{G}(\xi)\chi + \frac{\a_0}{2}e^{2\xi}\xi'],
\\&\pa_t\chi=e^{-2\xi}\Big[\frac12\Big(\frac{\bar G(\xi)\chi + \xi'\chi'}{\sqrt{1+\xi'^2}}\Big)^2 - \frac12\chi'^2 + \s_0 e^{\xi}\Big(\Big(\frac{\xi'}{\sqrt{1+\xi'^2}}\Big)' - \frac{1}{\sqrt{1+\xi'^2}}\Big)\Big]
\\&\qquad+\frac{\a_0}{2}\chi' + \frac{\a_0^2}{8}e^{2\xi} + \a_0\tilde{K}(\xi)\chi + \s_0 - \frac{\a_0^2}{8}.
\end{align*}

\end{theorem}

\begin{proof} Given the equations \eqref{h.eq.}-\eqref{psi.eq.}, let us first transform the equation for $\pa_t h$.
Since for all $x=\phi(0,\th)\in\S^1$ for $\th\in\T^1$, we have $1+h(t,x)=1+h(t,\phi(0,\th))=e^{\xi(t,\th)}$, hence
\begin{equation}\label{pa.t.h.strip}\pa_t h(t,x)=e^{\xi(t,\th)}\pa_t\xi(t,\th).
\end{equation}
Then, by \eqref{pa.t.h.strip}, \eqref{nu.representation}, \eqref{J.th.}, \eqref{G.into.G.on.T1} and \eqref{mM.op.def}, we have
\begin{equation*}e^\xi\pa_t\xi=e^{-\xi}\tilde{G}(\xi)\chi + \frac{\a_0}{2}\xi',
\end{equation*}
from which we get \eqref{xi.eq}.

Let us now transform the equation for $\pa_t\psi$.
Since for all $x=\phi(\th,0)\in\S^1$ with $\th\in\T^1$, we have $\psi(t,x)=\psi(t,\cdot)\circ\phi(0,\th)=\chi(t,\th)$ and so by \eqref{G.into.G.on.T1}, \eqref{nu.representation}, equations \eqref{scalar.products.th}, \eqref{norm.grad.S} and \eqref{J.th.}, \eqref{mean.curvature} and \eqref{Neumann.Dirichlet.on.torus}, we have
\begin{align*}\pa_t\chi(t,\th)&=\pa_t\psi(t,\cdot)\circ\phi(0,\th)
\\&=\frac12\Big(e^{-\xi}(1+\xi'^2)^{-\frac12}\tilde G(\xi)\chi + e^{-\xi}(1+\xi'^2)^{-\frac12}\xi'\chi'\Big)^2-e^{-2\xi}\cdot\frac12\chi'^2 
\\&\qquad- \s_0 e^{-\xi}\cdot\Big( \Big(\xi'(1+\xi'^2)^{-\frac12}\Big)' - (\sqrt{1+\xi'^2})^{-\frac12}-e^\xi \Big)
\\&\qquad+\frac{\a_0}{2}\chi' + \frac{\a_0^2}{8}e^{2\xi} + \a_0\bar K(\xi)\chi  +\s_0  -\frac{\a_0^2}{8},
\end{align*}
which is exactly \eqref{chi.eq}.

\end{proof}

\begin{remark}\label{inherited.zero.average.Psi.choice} By \eqref{Neumann.Dirichlet.on.torus}, still $\bar K(\xi)\chi$ is defined up to additive constants. From now on, we suppose that
\begin{equation}\label{zero.average.Psi}\int_{\T^1}\bar K(\xi)\chi\,d\th=0.
\end{equation}
\end{remark}

\subsection{Hamiltonian structure for the drop equations with constant vorticity}

From now on, we consider the notation 
\begin{equation*}G(\xi)\chi:=\bar G(\xi)\chi,\qquad K(\xi)\chi:=\bar K(\xi)\chi.\end{equation*}

We now want to show that \eqref{xi.eq}-\eqref{chi.eq} has a Hamiltonian structure. Let us consider 
\begin{equation}\label{energy}\mH:=\mK + \s_0\mA(\pa\Om_h)-\Big(\s_0 - \frac{\a_0^2}{8}\Big)\mV(\Om_h),\end{equation}
where $\mK:=\frac12\int_{\Om_h}|u|^2dx$ is the kinetic energy, $\mA(\pa\Om_h):=\int_{\pa\Om_h}1\,d\mH^1$ the length of $\pa\Om_h$ and $\mV(\Om_h):=\int_{\Om_h}1\,dx$ the area of $\Om_h$. We want to describe each functional as a functional over $\xi,\chi$. To start with, we have
\begin{align}&\mA(\pa\Om_h)=\int_{\S^1}J(h)\,d\s=\int_{\T^1}e^\xi\sqrt{1+\xi'^2}\,d\th, \label{Area}
\end{align}
and 
\begin{align}\mV(\Om_h)
&=\frac{\s_0}{2}\int_{\pa\Om_h}\la x,\nu_h\ra\,d\mH^1 =\frac{\s_0}{2}\int_{\T^1}e^{2\xi}\,d\th.\label{volume}
\end{align}
We are left to make $\mK$ explicit. By \eqref{u.decomposition.potential}, we have
\begin{align}\mK
&=\frac12\int_{\Om_h}\Big|\grad\Phi - \frac{\a_0}{2}\mJ x\Big|^2\,dx \notag
\\&=\frac12\underbrace{\int_{\Om_h}|\grad\Phi|^2\,dx}_{=:(I)} -\underbrace{\frac{\a_0}{2}\int_{\Om_h}\la\grad\Phi,\mJ x\ra\,dx}_{=:(II)} + \underbrace{\frac{\a_0^2}{8}\int_{\Om_h}(x_1^2 + x_2^2)\,dx}_{=:(III)}. \label{kinetic.energy.trace}
\end{align}
In \cite{L}, it has been shown that
\begin{equation}\label{(I)}(I)=\frac12\int_{\S^1}\psi\,G(h)\psi\,J(h)\,d\s=\frac12\int_{\T^1}\chi\,G(\xi)\chi\,d\th.
\end{equation}
As for $(II)$, we have by \eqref{velocity.potential}, Divergence Theorem, \eqref{Neumann.problem.Psi}, 
\begin{align}(II)
&=\frac{\a_0}{2}\int_{\Om_h}\la\grad\Psi, x\ra\,dx 
=-\frac{\a_0}{4}\int_{\S^1}(1+h)^2\mM\psi\,d\s
=-\frac{\a_0}{4}\int_{\T^1}e^{2\xi}\,\chi'\,d\th. \label{(II)}
\end{align}
As for $(III)$, similarly we have by Divergence Theorem and \eqref{nu.h},
\begin{align}(III)
&=\frac{\a_0^2}{32}\int_{\Om_h}\div(x|x|^2)\,dx
=\frac{\a_0^2}{32}\int_{\T^1}e^{4\xi}\,d\th. \label{(III)}
\end{align}
Putting \eqref{Area}, \eqref{volume}, \eqref{(I)}, \eqref{(II)}, \eqref{(III)} together, we finally get \eqref{Hamiltonian.function}:
\begin{align}\mH(\xi,\chi)&=\frac{1}{2}\int_{\T^1}\chi\,G(\xi)\chi\,d\th + \s_0\int_{\T^1}e^\xi\,\sqrt{1+\xi'^2}\,d\th - \Big(\s_0 - \frac{\a_0^2}{8}\Big)\frac12\int_{\T^1}e^{2\xi}\,d\th \notag
\\&-\frac{\a_0}{4}\int_{\T^1}e^{2\xi}\,\chi'\,d\th + \frac{\a_0^2}{32}\int_{\T^1}e^{4\xi}\,d\th \notag
\\&=:\underbrace{\mK(\xi,\chi) + \s_0\mA(h) - \s_0\mV(\xi)}_{=:\mH_0(\xi,\chi)} + \frac{\a_0^2}{8}\mV(\xi) +\frac{\a_0}{2}\mI_0(\xi,\chi) + \frac{\a_0^2}{8}\mG(\xi), \label{alternative.Hamiltonian.function}
\end{align}
where
\begin{equation}\mV(\xi):=\frac12\int_{\T^1}e^{2\xi}\,d\th,\quad\mI_0(\xi,\chi):=-\frac12\int_{\T^1}e^{2\xi}\,\chi'\,d\th = \int_{\T^1}e^{2\xi}\,\xi'\chi\,d\th,\quad\mG(\xi):=\frac14\int_{\T^1}e^{4\xi}\,d\th. \label{angular.momentum}
\end{equation}
We now prove the following fact:
\begin{lemma} The capillary $2D$ drop equations with constant vorticity $\a_0$ \eqref{xi.eq}-\eqref{chi.eq} are the quasi-Hamiltonian equations
\begin{equation}\label{quasi.hamiltonian.equations}\pa_t\begin{pmatrix}\xi \\ \chi
\end{pmatrix}=J(\xi)\grad\mH(\xi,\chi) + \begin{pmatrix}0 \\ \a_0 K(\xi)\chi + \frac{\a_0^2}{4}e^{2\xi}
\end{pmatrix},
\end{equation}
with respect to the Poisson tensor $J(\xi):=\begin{pmatrix}0 & e^{-2\xi} \\ -e^{-2\xi} & 0
\end{pmatrix}$.
\end{lemma}

\begin{proof} In \cite{L}, the following formulae have been proved:
\begin{align}&\pa_\xi\mH_0(\xi,\chi)=-\frac12\Big(\frac{G(\xi)\chi + \xi'\chi'}{\sqrt{1+\xi'^2}}\Big)^2 + \frac12\chi'^2 - \s_0\Big(e^\xi\Big[\frac{\xi'}{\sqrt{1+\xi'^2}}\Big]'-\frac{e^\xi}{\sqrt{1+\xi'^2}}\Big) - \s_0 e^{2\xi}, \notag
\\&\pa_\xi\mI(\xi,\chi)=-e^{2\xi}\chi',\qquad\pa_\xi\mV(\xi,\chi)=e^{2\xi},\label{xi.derivative.known}
\\&\pa_\chi\mH_0(\xi,\chi)=G(\xi)\chi,\qquad\pa_\chi\mI(\xi,\chi)=e^{2\xi}\xi', \qquad\pa_\chi\mV(\xi,\chi)=0.\label{chi.derivative.known}
\end{align}
We also have
\begin{equation}\label{G.derivative}\pa_\xi\mG(\xi,\chi)=e^{4\xi},\qquad\pa_\chi\mG(\xi,\chi)=0.
\end{equation}
Putting all together these formulae, we get \eqref{quasi.hamiltonian.equations}.

\end{proof}

However, in the following we show that with some carefulness, \eqref{quasi.hamiltonian.equations} are indeed Hamiltonian:
\begin{theorem}\label{thm:hamiltionian.structure.1} The capillary $2D$ drop equations with constant vorticity $\a_0$ \eqref{xi.eq}-\eqref{chi.eq} can be written as \eqref{hamiltonian.equations}, namely
\begin{equation*}\pa_t\begin{pmatrix}\xi \\ \chi
\end{pmatrix} = J_{\a_0}(\xi)\grad\mH(\xi,\chi),
\end{equation*}
where $J_{\a_0}$ is
\begin{equation*}J_{\a_0}(\xi):=\begin{pmatrix}0 & e^{-2\xi} \\ -e^{-2\xi} & \a_0\pa_\th^{-1}\Pi_0^\perp
\end{pmatrix}.
\end{equation*}
Here, $\pa_\th^{-1}$ is the indefinite integral operator, while $\Pi_0^\perp$ is the $L^2-$projector operator on the zero-average functions subspace of $L^2$. 

\end{theorem}

\begin{proof} The equation \eqref{hamiltonian.equations} can be obtained from \eqref{quasi.hamiltonian.equations} by applying \eqref{G.is.K.derivative} and by observing that $\pa_\chi\mH(\xi,\chi)=(-K(\xi)\chi + \frac{\a_0}{4}e^{2\xi})'$ and also that the average of $\pa_\chi\mH(\xi,\chi)$ is zero for all $(\xi,\chi)$.
\end{proof}

\begin{remark}\label{remark:hamiltonian.if.volume.conserved} We point out that \eqref{hamiltonian.equations} are \emph{not} Hamiltonian on the phase space $H^{k_1}\times H_0^{k_2}$ because there we cannot invert $J_{\a_0}(\xi)$. Indeed, let us compute the formal inverse $J_{\a_0}^{-1}$ for $J_{\a_0}$, supposing that
\begin{align*}J_{\a_0}^{-1}(\xi):=\begin{pmatrix}Q(\xi) & -e^{2\xi} \\ e^{2\xi} & 0
\end{pmatrix},
\end{align*}
where $Q(\xi)=Q(\xi)[\cdot]$ is a bounded linear operator to compute. Then,
\begin{align*}J_{\a_0}^{-1}(\xi)\circ J_{\a_0}(\xi)
&=\begin{pmatrix}Q(\xi) & -e^{2\xi} \\ e^{2\xi} & 0
\end{pmatrix}\begin{pmatrix}0 & e^{-2\xi} \\ -e^{-2\xi} & \a_0\pa_\th^{-1}\Pi_0^\perp
\end{pmatrix}=\begin{pmatrix}1 & Q(\xi)e^{-2\xi} - \a_0 e^{2\xi}\pa_{\th}^{-1}\Pi_0^\perp \\ 0 & 1
\end{pmatrix}.
\end{align*}
Thus, $Q(\xi)$ must be such that for all $f$ smooth enough,
\begin{equation}\label{Q.equation}Q(\xi)[e^{-2\xi}f]=\a_0e^{2\xi}\pa_\th^{-1}\Pi_0^\perp f.
\end{equation}
This equation is formally satisfied by the operator
\begin{equation}\label{Q.operator}Q(\xi)f:=\a_0 e^{2\xi}\pa_\th^{-1}\Pi_0^\perp[ e^{2\xi}f].
\end{equation}
We eventually get that maps pairs of classes of functions into classes of functions, against the fact that $\chi$ is a class of functions while $\xi$ is not.

In such a case, to get \eqref{hamiltonian.equations} to be Hamiltonian, we must fix a constraint for $\xi$, which is naturally given by the conservation of volume $\mV(\xi)$, which is conserved along \eqref{hamiltonian.equations}:
\begin{equation*}M_V:=\{(\xi,\chi)\in H^{s_1}\times\dot H^{s_2}\colon\,\mV(\xi)=V\}.
\end{equation*}
\end{remark}

Despite Remark \ref{remark:hamiltonian.if.volume.conserved} suggests a phase space to work with, we want to avoid it because the volume $\mV(\xi)=\frac12\int_{\T^1}e^{2\xi}\,d\th$ is \emph{not} a linear constraint. Indeed, the following theorem shows that a certain change of variable allows us to get Hamiltonian equations over $H^{s_1}\times H^{s_2}$.

\begin{theorem}\label{thm:hamiltonian.structure.2} The capillary $2D$ drop equations with constant vorticity $\a_0$ \eqref{hamiltonian.equations} are equivalent to the Hamiltonian equations
\begin{equation}\label{new.hamiltonian.equations}\pa_t\begin{pmatrix}\zeta \\ \g\end{pmatrix} = J(\zeta)\grad\bar\mH(\zeta,\g),
\end{equation}
where we have \eqref{change.of.coordinates}, \eqref{new.hamiltonian.function} and \eqref{new.Poisson.tensor}:
\begin{align*}&\begin{pmatrix}\xi \\ \chi
\end{pmatrix}=\mC(\zeta,\g):=\begin{pmatrix}\zeta \\ \g + \frac{\a_0}{4}\pa_\th^{-1}\Pi_0^\perp e^{2\zeta}
\end{pmatrix}, 
\\&\bar\mH(\zeta,\g):=\mH\circ\mC(\zeta,\g),
\\&J(\zeta):=\begin{pmatrix}0 & e^{-2\zeta} \\ -e^{-2\zeta} & 0
\end{pmatrix}.
\end{align*}

\end{theorem}

\begin{proof} We look for a change of coordinates of the kind
\begin{equation}\label{convenient.change.of.coordinates}\begin{pmatrix}\xi \\ \chi
\end{pmatrix}=:\mC(\zeta,\g):=\begin{pmatrix}1 & 0 \\ \mQ(\cdot) & 1
\end{pmatrix}\begin{pmatrix}\zeta \\ \g
\end{pmatrix},
\end{equation}
where $\mQ(\cdot)$ is a smooth operator to fix. Then, we get 
\begin{align}&\pa_t\mC(\zeta,\g)=d\mC(\zeta)\circ\pa_t\begin{pmatrix}\zeta \\ \g
\end{pmatrix}=\begin{pmatrix}1 & 0 \\ d\mQ(\zeta)[\cdot] & 1
\end{pmatrix}\circ\pa_t\begin{pmatrix}\zeta \\ \g
\end{pmatrix}, \label{new.time.derivative}
\end{align}
Now, let us define the candidate Hamiltonian
\begin{equation}\label{hamiltonian.change.of.coordinates}\bar\mH(\zeta,\g):=\mH\circ\mC(\zeta,\g).
\end{equation}
Then, we have
\begin{align*}d\bar\mH(\zeta,\g)[\hat\zeta,\hat\g] 
&= d\mH(\mC(\zeta,\g))\circ d\mC(\zeta,\g)[\hat\zeta,\hat\g]
\\&=\la\pa_\xi\mH\circ\mC(\zeta,\g),\hat\zeta\ra_{L^2} + \la\pa_\chi\mH\circ\mC(\zeta,\g),d\mQ(\zeta)\hat\zeta\ra_{L^2} +\la\pa_\chi\mH\circ\mC(\zeta,\g),\hat\g\ra_{L^2}
\\&=\la\pa_\xi\mH\circ\mC(\zeta,\g) + d\mQ(\zeta)^*[\pa_\chi\mH\circ\mC(\zeta,\g)],\hat\zeta\ra_{L^2} + \la\pa_\chi\mH\circ\mC(\zeta,\g),\hat\g\ra_{L^2},
\end{align*}
where $d\mQ(\zeta)^*$ is the $L^2-$adjoint operator of $d\mQ(\zeta)$. Then
\begin{equation}\label{gradient.relations}\grad\mH\circ\mC(\zeta,\g)=(d\mC^{-1}(\zeta))^*\grad\bar\mH(\zeta,\g) =\begin{pmatrix}1 & -d\mQ(\zeta)^* \\ 0 & 1
\end{pmatrix}\grad\bar\mH(\zeta,\g)
\end{equation}
Therefore, thanks to \eqref{new.time.derivative}, \eqref{hamiltonian.change.of.coordinates} and \eqref{gradient.relations}, the equation \eqref{hamiltonian.equations} can be written in the new coordinates as
\begin{equation}\label{equations.to.make.canonical}\pa_t\begin{pmatrix}\zeta \\ \g
\end{pmatrix} = [d\mC^{-1}(\zeta)\circ J_{\a_0}(\zeta)\circ(d\mC^{-1}(\zeta))^*]\grad\bar\mH(\zeta,\g),
\end{equation}
where
\begin{align}d\mC^{-1}(\zeta)\circ J_{\a_0}(\zeta)\circ(d\mC^{-1}(\zeta))^*
&=\begin{pmatrix}1 & 0 \\ -d\mQ(\zeta) & 1
\end{pmatrix}\begin{pmatrix}0 & e^{-2\xi} \\ -e^{-2\xi} & \a_0\pa_\th^{-1}\Pi_0^\perp
\end{pmatrix}\begin{pmatrix}1 & -d\mQ(\zeta)^* \\ 0 & 1
\end{pmatrix} \notag
\\&=\begin{pmatrix}1 & 0 \\ -d\mQ(\zeta) & 1
\end{pmatrix}\begin{pmatrix}0 & e^{-2\zeta} \\ -e^{-2\zeta} & e^{-2\zeta}d\mQ(\zeta)^* + \a_0\pa_\th^{-1}\Pi_0^\perp
\end{pmatrix} \notag
\\&=\begin{pmatrix}0 & e^{-2\zeta} \\ -e^{-2\zeta} & -d\mQ(\zeta)\circ e^{-2\zeta} + e^{-2\zeta}d\mQ(\zeta)^* + \a_0\pa_\th^{-1}\Pi_0^\perp
\end{pmatrix}.
\label{poisson.tensor.to.make.canonical}
\end{align}
To turn \eqref{poisson.tensor.to.make.canonical} into $J(\zeta)$, we look for $\mQ(\zeta)$ satisfying for any $f\in L^2$,
\begin{equation}\label{operator.equation}-d\mQ(\zeta)[e^{-2\zeta}f] + e^{-2\zeta}d\mQ(\zeta)^*f + \a_0\pa_\th^{-1}\Pi_0^\perp f=0.
\end{equation}
This is equivalent to say that for smooth $f,g\in L^2$,
\begin{equation}\label{operator.equation.2}\a_0\la\pa_\th^{-1}\Pi_0^\perp f,g\ra_{L^2}=\la d\mQ(\zeta)[e^{-2\zeta}f],g\ra_{L^2} - \la f,d\mQ(\zeta)[e^{-2\zeta}g]\ra_{L^2}.
\end{equation}
Since the indefinite integral operator is skew-symmetric with respect to $\la\cdot,\cdot\ra_{L^2}$, and since we want to preserve the periodicity of the potential function, \eqref{operator.equation.2} is satisfied by all the $\mQ$ such that
\begin{equation}\label{good.differential}d\mQ(\zeta)f=\frac{\a_0}{2}\pa_\th^{-1}\Pi_0^\perp e^{2\zeta}f,
\end{equation}
and so we have
\begin{equation}\label{Q.is.translated.volume}\mQ(\zeta)=\frac{\a_0}{4}\pa_\th^{-1}\Pi_0^\perp e^{2\zeta}.
\end{equation}
This concludes the proof.
\end{proof}

\begin{remark}\label{rem:consequences.Wahlen.coordinates} We have along \emph{any} direction $f,g\in L^2$,
\begin{align*}\la d\mQ(\zeta)f,g\ra_{L^2}
&=\frac{\a_0}{2}\la\pa_\th^{-1}\Pi_0^\perp e^{2\zeta}f,g\ra_{L^2}=-\frac{\a_0}{2}\la f,e^{2\zeta}\Pi_0^{\perp}\pa_\th^{-1} g\ra_{L^2},
\end{align*}
from which we get 
\begin{equation}\label{adjoint.of.Q.differential}d\mQ(\zeta)^*=-\frac{\a_0}{2} e^{2\zeta}\Pi_0^\perp\pa_\th^{-1}.
\end{equation}

\end{remark}

\subsection{Symmetries and conserved quantities}

\medskip

In this framework, we can get conserved quantities from symmetries.
\begin{lemma}\label{lemma:symmetries} Given the Hamiltonian equations \eqref{hamiltonian.equations}, the following facts hold.

\medskip

$(i)$ The Hamiltonian $\bar\mH$ is invariant by the velocity potential translation 
\begin{equation}\label{translation.operator}\mS_a(\zeta,\g):=(\zeta,\g+a),\qquad\forall a\in\R.
\end{equation}
As a result, the volume 
\begin{equation}\label{volume.new.coordinates}\bar\mV(\zeta)=\frac12\int_{\T^1}e^{2\zeta}\,d\th\end{equation} is a conserved quantity for \eqref{new.hamiltonian.equations} whose Hamiltonian vector field is 
\begin{equation}\label{volume.vector.field}X_   {\bar\mV}(\zeta,\g):=J(\zeta)\grad\bar\mV(\zeta)=\frac{d}{da}\Big|_{a=0}\mS_a(\zeta,\g)=\begin{pmatrix}0 \\ 1
\end{pmatrix}.
\end{equation}

$(ii)$ The Hamiltonian $\bar\mH$ is invariant by the torus action
\begin{equation}\label{torus.action}\mT_\a(\zeta,\g)(\th):=(\zeta(\th+\a),\g(\th+\a)),\qquad\forall\a,\,\th\in\T^1.
\end{equation}
As a result, the total angular momentum
\begin{equation}\label{angular.momentum}\bar\mI(\zeta,\g):=-\frac12\int_{\T^1}e^{2\zeta}\g'\,d\th,
\end{equation}
where $\mI_0,\,\mG$ are defined in \eqref{angular.momentum}, is a conserved quantity for \eqref{hamiltonian.equations}, and its Hamiltonian vector field is
\begin{equation}\label{angular.momentum.vector.field}X_{\bar\mI}(\zeta,\g):=J(\zeta)\grad\bar\mI(\zeta,\g)=\frac{d}{d\a}\Big|_{\a=0}\mT_{\a}(\zeta,\g)=\begin{pmatrix}\zeta' \\ \g'
\end{pmatrix}.
\end{equation}

$(iii)$ The Hamiltonian $\bar\mH$ is invariant under the action of reflection operator
\begin{equation}\label{reversibility.action}\mR(\zeta,\g)(\th):=(\zeta(-\th),-\g(-\th)).
\end{equation}
As a result, the subspace
\begin{equation}\label{reversible.subspace}R:=\{(\zeta,\g)\in H^{s_1}\times\dot H^{s_2}\colon\,\mR(\zeta,\g)=(\zeta,\g)\}
\end{equation}
is dynamically invariant for \eqref{new.hamiltonian.equations}.    
\end{lemma}

\begin{proof} $(i)$ For all $(\zeta,\g)$, the function $f(a):=\bar\mH(\zeta,\g+a)$ is constant for all $a\in\R$, because by \eqref{alternative.Hamiltonian.function}, $G(\zeta)[\chi+a]=G(\zeta)\chi$ (Lemma \ref{lemma:G}, point $(iii)$) and $\la G(\zeta)\chi,1\ra_{L^2}=0$, we have
\begin{align*}f(a) 
&=\bar\mH(\zeta,\g+a)=\mH\Big(\zeta,\g+\frac{\a_0}{4}\pa_\th^{-1}\Pi_0^\perp e^{2\zeta} + a\Big)=\mH\Big(\zeta,\g+\frac{\a_0}{4}\pa_\th^{-1}\Pi_0^\perp e^{2\zeta}\Big)=f(0),
\end{align*}
which gives the invariance $\bar\mH\circ\mS_a=\bar\mH$. As a result, we have
\begin{align*}0
&=\frac{d}{da}\Big|_{a=0}f = d\bar\mH(\zeta,\g)[(0,1)] = \{\bar\mH,\bar\mV\}(\zeta,\g),
\end{align*}
where the last line holds because $\grad\bar\mI(\zeta)=(-e^{2\zeta}\g',e^{2\zeta}\zeta')$ and thus $X_{\bar\mI}(\zeta)=J(\zeta)\grad\bar\mI(\zeta)=(\zeta',\g')$.

\medskip

$(ii)$ The function $f(\a):=\bar\mH\circ\mT_\a(\zeta,\g)$ holds because of \eqref{G.translation.inv} and the change of variable $\b=\th+\a$:
\begin{align*}f(\a)
&=\bar\mH\circ\mT_\a(\zeta,\g)
\\&=\mH\Big(\mT_\a\zeta,\mT_\a\g + \frac{\a_0}{4}\pa_\th^{-1}\Pi_0^\perp e^{2\mT_\a\zeta} \Big)
\\&=\mH\Big(\zeta,\g + \frac{\a_0}{4}\pa_\th^{-1}\Pi_0^\perp e^{2\mT_\a\zeta} \Big)
\\&=f(0),
\end{align*}
which gives the invariance $\bar\mH\circ\mT_\a=\bar\mH$. As a result, we have
\begin{align*}0
&=\frac{d}{d\a}\Big|_{\a=0}f = d\bar\mH(\zeta,\g)[(\zeta',\g')] = \{\bar\mH,\bar\mI\}(\zeta,\g),
\end{align*}
where the last line holds because $\grad\bar\mI(\zeta)=(e^{2\zeta},0)$ and thus $X_{\bar\mV}(\zeta)=J(\zeta)\grad\bar\mV(\zeta)=(0,1)$.

\medskip

$(iii)$ The relation $\bar\mH\circ\mR=\bar\mH$ holds because of \eqref{alternative.Hamiltonian.function}, \eqref{G.inversion} and the change of variable $\a:=-\th$.

\end{proof}

\begin{remark}\label{rem:equivariances} Since we have the invariances
\begin{align}&\bar\mH\circ\mT_\a=\bar\mH,\quad\bar\mH\circ\mS_a=\bar\mH,\quad\bar\mH\circ\mR=\bar\mH, \label{Hamiltonian.invariances}
\\&\bar\mI\circ\mT_\a=\bar\mI,\quad\bar\mI\circ\mS_a=\bar\mI,\quad\bar\mI\circ\mR=\bar\mI, \label{Angular.momentum.invariances}
\\&\bar\mV\circ\mT_\a=\bar\mV,\quad\bar\mV\circ\mS_a=\bar\mV,\quad\bar\mV\circ\mR=\bar\mV, \label{Volume.invariances}
\end{align}
we also have the equivariances
\begin{align}&\grad\bar\mH\circ\mT_\a=\mT_\a\circ\grad\bar\mH,\quad\grad\bar\mH\circ\mS_a=\mS_a\circ\grad\bar\mH,\quad\grad\bar\mH\circ\mR=\mR\circ\grad\bar\mH, \label{Hamiltonian.equivariances}
\\&\grad\bar\mI\circ\mT_\a=\mT_\a\circ\grad\bar\mI,\quad\grad\bar\mI\circ\mS_a=\mS_a\circ\grad\bar\mI,\quad\grad\bar\mI\circ\mR=\mR\circ\grad\bar\mI, \label{Angular.momentum.equivariances}
\\&\grad\bar\mV\circ\mT_\a=\mT_\a\circ\grad\bar\mV,\quad\grad\bar\mV\circ\mS_a=\mS_a\circ\grad\bar\mV,\quad\grad\bar\mV\circ\mR=\mR\circ\grad\bar\mV. \label{Volume.equivariances}
\end{align}

\end{remark}

\begin{remark} The quantity $\bar\mI$ is really the total angular momentum of the drop on the plane of rotation. Indeed, by \eqref{u.decomposition.potential}, \eqref{(II)}, \eqref{(III)} and \eqref{change.of.coordinates}, we have
\begin{align*}\frac12\int_{\Om_h}x\wedge u\,dx
&=-\frac12\int_{\Om_h}\la u,\mJ x\ra\,dx
\\&=-\frac12\int_{\Om_h}\la\grad\Phi -\frac{\a_0}{2}\mJ x,\mJ x\ra\,dx
\\&=-\frac{1}{2}\int_{\T^1}e^{2\xi}\Big\{\chi' - \frac{\a_0}{8}e^{2\xi} \Big\}\,d\th
\\&=\bar\mI(\zeta,\g).
\end{align*}
This is different from the Water Waves equations on the flat torus (see for instance \cite{Wahlen, BBMM}), since the total angular momentum is given simply by that of the gradient component of $u$ up to an additive constant, which corresponds to our term $\mG(\xi)$.
\end{remark}
We have another conserved quantity: the fluid barycenter velocity 
\begin{align}\mB
&:=\int_{\Om_h}u\,dx  \notag
\\&=\int_{\Om_h}\grad\Phi\,dx - \frac{\a_0}{2}\int_{\Om_h}\mJ x\,dx  \notag
\\&=\int_{\Om_h}(\div(\Phi\textbf{e}_i))_{i=1,2}\,dx - \frac{\a_0}{4}\int_{\Om_h}\mJ(\div(x_i^2\textbf{e}_i))_{i=1,2}\,dx
\notag
\\&=\int_{\pa\Om_h}\Phi\nu_h\,d\mH^1 - \frac{\a_0}{4}\int_{\pa\Om_h}\mJ(x_i^2(\nu_h)_i)_{i=1,2}\,d\mH^1 \notag
\\&=\int_{\T}e^\rho\chi\begin{pmatrix}\cos\th \\ \sin\th
\end{pmatrix}\,d\th - \int_{\S^1}e^\rho\rho'\chi\begin{pmatrix}-\sin\th \\ \cos\th
\end{pmatrix}\,d\th  \notag
\\&- \frac{\a_0}{4}\int_{\T}\mJ\Big\{e^{3\rho}\begin{pmatrix}\cos^3\th \\ \sin^3\th\end{pmatrix} + e^{3\rho}\rho'\begin{pmatrix}\sin\th\cos^2\th \\ -\cos\th\sin^2\th\end{pmatrix}\Big\}\,d\th\notag
\\&=\int_\T e^\zeta\Big(\g' + \frac{\a_0}{4}\Pi_0^\perp e^{2\zeta} - \frac{\a_0}{6}e^{2\zeta} \Big)\begin{pmatrix}-\sin\th \\ \cos\th
\end{pmatrix}\,d\th.
\label{barycenter.velocity.alternative}
\end{align}
The fact that it is a conserved quantity for the free boundary problem \eqref{div.eq.01}-\eqref{kin.eq.01} does not depend on the presence of vorticity. Indeed, by Reynolds transport theorem (see for instance \cite{B.J.LM}, Lemma 3.3), \eqref{dyn.eq.01}, divergence theorem on $\R^2$, \eqref{pressure.eq.01} and \eqref{div.theorem.hypersurfaces}, we have
\begin{align*}\frac{d}{dt}\mB
&=\int_{\pa\Om_h}u\la u,\nu_{\Om_h\ra}\,d\mH^1 + \int_{\Om_h}\pa_t u\,dx
\\&=\int_{\pa\Om_h}u\la u,\nu_{\Om_h\ra}\,d\mH^1 - \int_{\Om_h}(u\cdot\grad u + \grad p)\,dx
\\&=-\int_{\pa\Om_h}p\nu_{\Om_h}\,d\mH^1
\\&=-\s_0\int_{\pa\Om_h}H_{\Om_h}(\la \textbf{e}_i,\nu_h\ra)_{i=1,2}\,d\mH^1
\\&=0.
\end{align*}
Another interesting fact is the following. Let us consider the barycenter vector-valued functional
\begin{align}\mP:&=\int_{\Om_h}x\,dx\notag
\\&=\frac12\int_{\Om_h}(\div(x_i\textbf{e}_i))_{i=1,2}\,dx \notag
\\&=\frac12\int_{\pa\Om_h}(x_i(\nu_{h})_i)_{i=1,2}\,d\mH^1 \notag
\\&=\frac12\int_{\T}\Big\{e^{3\zeta}\begin{pmatrix}\cos^3\th \\ \sin^3\th\end{pmatrix} + e^{3\zeta}\zeta'\begin{pmatrix}\sin\th\cos^2\th \\ -\cos\th\sin^2\th\end{pmatrix}\Big\}\,d\th \notag
\\&=\frac12\int_{\T}e^{3\zeta}\begin{pmatrix}\cos\th \\ \sin\th
\end{pmatrix}\,d\th.
\label{barycenter.position}
\end{align}
Differentiating in time, by Reynolds transport theorem, divergence theorem and \eqref{div.eq.01} that it is not generally a constant of motion:
\begin{align}\label{position.time.derivative}\frac{d}{dt}\mP=\int_{\pa\Om_h}x\la u,\nu_h\ra\,d\mH^1=\int_{\Om_h}u\,dx=\mB.
\end{align}
However, $\mB$ is a constant of motion, and if we set $\mB=(0,0)$, then $\mP$ becomes as well. This also implies that if in particular $\mP=(0,0)$, we get
\begin{align}\int_\T e^\zeta\Big(\g' + \frac{\a_0}{4}\Pi_0^\perp e^{2\zeta} \Big)\begin{pmatrix}-\sin\th \\ \cos\th
\end{pmatrix}\,d\th = 0.
\label{barycenter.velocity.0.rigidity}
\end{align}

\section{Existence of rotating waves}

\medskip

Let us consider the formal linearization of \eqref{new.hamiltonian.equations} at the rotating circle $(\zeta,\g)=(0,0)$: 
\begin{align}L_0[\zeta,\g]:&=J(0)d\grad\bar\mH(\om,0)[\zeta,\g] \notag
\\&=J(0)d\mC(0)^*\circ d\grad\mH(0)\circ d\mC(0)[\zeta,\g] \notag
\\&=\begin{pmatrix}0 & 1 \\ -1 & 0
\end{pmatrix}\begin{pmatrix}1 & -\frac{\a_0}{2}\Pi_0^\perp\pa_\th^{-1} \\ 0 & 1
\end{pmatrix}\begin{pmatrix}-\s_0\pa_{\th\th} + (- \s_0 + \frac{\a_0^2}{4}) &  -\frac{\a_0}{2}\pa_\th \\ \frac{\a_0}{2}\pa_\th & G(0)
\end{pmatrix}\begin{pmatrix}\zeta \\ \frac{\a_0}{2}\pa_\th^{-1}\Pi_0^\perp\zeta + \g\end{pmatrix}\notag
\\&=\begin{pmatrix}0 & 1 \\ -1 & 0
\end{pmatrix}\begin{pmatrix}-\s_0\pa_{\th\th} -\s_0 + \frac{\a_0^2}{4}\Pi_0 & -\frac{\a_0}{2}\pa_\th - \frac{\a_0}{2}\Pi_0^\perp\pa_\th^{-1} G(0) \\ \frac{\a_0}{2}\pa_\th & G(0)
\end{pmatrix}\begin{pmatrix}\zeta \\ \frac{\a_0}{2}\pa_\th^{-1}\Pi_0^\perp\zeta + \g\end{pmatrix}  \notag
\\&=\begin{pmatrix}\frac{\a_0}{2} G(0)\pa_\th^{-1}\Pi_0^\perp + \frac{\a_0}{2}\pa_\th & G(0)
\\\s_0\pa_{\th\th} + \s_0  - \frac{\a_0^2}{4}\Pi_0 + \frac{\a_0^2}{4}\Pi_0^\perp  +  \frac{\a_0^2}{4}\Pi_0^\perp\pa_\th^{-1}G(0)\pa_\th^{-1}\Pi_0^\perp & \frac{\a_0}{2}\Pi_0^\perp\pa_\th^{-1} G(0) + \frac{\a_0}{2}\pa_\th 
\end{pmatrix}\begin{pmatrix}\zeta \\ \g
\end{pmatrix}.
\label{linearization.at.0}
\end{align}
Let us expand $\zeta,\g$ along the $L^2-$basis \eqref{Fourier.basis}:
\begin{equation}\label{fourier.expansion}\zeta=\sum_{(\ell,m)\in\mT}\zeta_{\ell,m}\ph_{\ell,m},\qquad\g=\sum_{(\ell,m)\in\mT}\g_{\ell,m}\ph_{\ell,m}.
\end{equation}
Since 
\begin{align*}&\ph_{\ell,m}'=-m\ell\ph_{\ell,-m},\qquad\ph_{\ell,m}''=-\ell^2\ph_{\ell,m},\qquad G(0)\ph_{\ell,m}=\ell\ph_{\ell,m},
\\&\pa_\th^{-1}\ph_{0,0}=0,\qquad\pa_\th^{-1}\Pi_0^\perp\ph_{\ell,m}=\frac{m}{\ell}\ph_{\ell,-m},
\end{align*}
then
\begin{equation}\label{L.is.sum.of.blocks.lin}L_0[\eta,\b]=\sum_{(\ell,m)\in\mT}L_0^{(\ell,m)}\begin{pmatrix}\eta_{\ell,m}\ph_{\ell,m} \\ \b_{\ell,-m}\ph_{\ell,-m}
\end{pmatrix},
\end{equation}
where
\begin{align}\label{L.lm.blocks.0.linear}L_0^{(\ell,m)}:=\begin{cases}\begin{pmatrix}\frac{m\a_0}{2} -\frac{m\ell\a_0}{2} & \ell
\\-\s_0\ell^2 + (\s_0 + \frac{\a_0^2}{4}) - \frac{\a_0^2}{4\ell} & -\frac{m\a_0}{2} + \frac{m\ell\a_0}{2}  
\end{pmatrix} & \text{if}\quad(\ell,m)\ne(0,0)
\\\begin{pmatrix}0 & 0
\\\s_0 - \frac{\a_0^2}{4} & 0  
\end{pmatrix} & \text{if}\quad(\ell,m)=(0,0)
\end{cases}
\end{align}
The eigenvalue equation is then
\begin{align}&P(\lm,0)=\lm^2=0,\qquad\qquad\qquad\qquad\qquad\qquad\qquad(\ell,m)=(0,0),
\\&P(\lm,\ell)=\lm^2 + \ell\Big(\s_0\ell^2 + \frac{\a_0^2}{4} - \s_0 + \frac{\a_0^2}{4} \Big) = 0,\qquad(\ell,m)\ne(0,0).\label{eigenvalue.equation.linear}
\end{align}
We then observe that the rotating circle is linearly stable, since all the eigenvalues are purely imaginary. Therefore, small oscillations about this solution are expected.

\medskip

This motivates the fact that in this section, we look for the \emph{rotating wave solutions} \eqref{rotating.waves}:
\begin{equation*}\xi(t,\th):=\eta(\th+\om t),\qquad\chi(t,\th):=\b(\th + \om t).
\end{equation*}
By \eqref{Hamiltonian.equivariances}, \eqref{Angular.momentum.equivariances} and \eqref{Volume.equivariances}, the rotating waves are solutions of the equation
\begin{align}&\om\eta' = e^{-2\eta}\Big[G(\eta)\beta + \frac{\a_0}{2}e^{2\eta}\eta'\Big], \label{rw.1}
\\&\om\b'=- e^{-2\eta}\Big[-\frac12 \Big( \frac{G(\eta) \beta + \eta'\beta'}{ 
\sqrt{1+\eta'^2}} \Big)^2
+ \frac12\beta'^2 
-\s_0\Big(e^{\eta}\Big[\frac{\eta'}{\sqrt{1+\eta'^2}} \Big]' - \frac{e^{\eta}}{\sqrt{1+\eta'^2}} \Big)\Big] \notag
\\&\qquad\quad-e^{-2\eta}\Big[-\frac{\a_0}{2}e^{2\eta}\b' -\frac{\a_0^2}{8}e^{4\eta}\Big] + \a_0 K(\eta)\b + \s_0 - \frac{\a_0^2}{8}. \label{rw.2}
\end{align}

\subsection{The critical point structure of rotating wave equation}

We have the following characterization of rotating waves.
\begin{theorem}\label{thm:rotating.waves.as.critical.points} The following facts are equivalent.

\medskip

$(i)$ The function $u=(\eta,\b)$ solves \eqref{rw.1}-\eqref{rw.2}.

$(ii)$ The function $\bar u=\mC^{-1}(u)$ is a critical point for the functional
\begin{equation}\label{energy.volume.functional}\bar\mE(\om,\bar u):=\bar\mH(\bar u)-\om(\bar\mI(\bar u)-a)=\mH\circ\mC(\bar u) - \om(\bar\mI(\bar u)-a)
\end{equation}
for all $a\in\R$.

\end{theorem}

\begin{proof} Since \eqref{xi.eq}-\eqref{chi.eq} are equivalent to \eqref{new.hamiltonian.equations}, then by calling $(\bar\eta,\bar\b):=\mC^{-1}(\eta,\b)$ we get that \eqref{rw.1}-\eqref{rw.2} are equivalent to 
\begin{equation*}\om\begin{pmatrix}\bar\eta' \\ \bar\b'
\end{pmatrix}=J(\bar\eta)\grad\bar\mH(\bar\eta,\bar\b).
\end{equation*}
Multiplying both sides by $J(\bar\eta)^{-1}$, we get exactly $\grad\bar\mE(\om,\bar u)=0$.
\end{proof}

From now on, with an abuse of notation we set
\begin{align}&\mF(\om, u):=\begin{pmatrix}\mF_1(\om, u) \\ \mF_2(\om, u)\end{pmatrix}:=\grad\bar\mE(\om,u)=d\mC(u)^*[\grad\mH\circ\mC(u)] - \om\grad\bar\mI(u).\label{rw.vector.field}
\end{align}
Thanks to Lemma \ref{lemma:G} and the equivariances \eqref{Hamiltonian.equivariances}-\eqref{Angular.momentum.equivariances}, the following lemma holds:

\begin{lemma}\label{ww.rotating.wave.eq} Let us consider $\xi(t,\th),\chi(t,\th)$, $t\in\R,\,\th\in\T^1$, satisfying the ansatz \eqref{rotating.waves}, let $\kappa\in\N$ and $\mathfrak{s}>0,\,s_0,\,s>0$, $s \geq s_0 > 1$, and let 
\[
U := \{ u = (\eta, \b) : 
\eta \in H_\kappa^{\mathfrak{s},s+\frac32}, \ \ 
\beta \in  \dot H_{\kappa}^{\mathfrak{s},s+1}, \ \ 
\| \eta \|_{H_{\kappa}^{\mathfrak{s},s_0 + 1}} < \delta_0 \},
\]
where $\delta_0$ is the constant in Lemma \ref{lemma:G}, point $(vii)$. 

Then, 
$\mF_{1}(\om, u) \in H_{\kappa}^{\mathfrak{s},s-\frac12}$, $\mF_{2}(\om, u) \in H_{\kappa,0}^{\mathfrak{s},s}$ 
for all $u \in U$, $\om \in \R$, the map 
\[
\mF : \R \times U \to \dot H_{\kappa}^{\mathfrak{s},s-\frac12}\times H_{\kappa,0}^{\mathfrak{s},s}, 
\]
such that $\mF(\om,u):=(\mF_1(\om,u),\mF_2(\om,u))$,
is analytic, and the vector $(\eta(\th+\om t),\beta(\th+\om t))$ solves the equation \eqref{new.hamiltonian.equations} if and only if 
\begin{equation}\label{rotating.wave.equation.0}\mF(\om,u)=(0,0).\end{equation}
\end{lemma}

\begin{remark}\label{rem:F.codomain.zero.average}We observe by \eqref{lemma:G} that 
\begin{align*}\la\mF(\om,u),(0,1)\ra_{(L^2)^2}
&=\la G(\eta)\b + \frac{\a_0}{2}\pa_\th(e^{2\eta}), 1\ra_{(L^2)^2}
\\&=\la\grad\bar\mH(u),J(\eta)\grad\bar\mV(\eta)\ra_{(L^2)^2} - \om\la\grad\bar\mI(u),J(\eta)\grad\bar\mV(\eta)\ra_{(L^2)^2}
\\&=0,
\end{align*}
which explains why the image of $\mF$ is contained in $\dot H_{\kappa}^{\mathfrak{s},s-\frac12}\times H_{\kappa,0}^{\mathfrak{s},s}$. However, since $\b\in\dot H_\kappa^{\mathfrak{s},s+1}$, we can fix its mean to be zero, so that $\mF$ maps $H_\kappa^{\mathfrak{s},s+\frac32}\times H_{\kappa,0}^{\mathfrak{s},s+1}$ to $H_\kappa^{\mathfrak{s},s-\frac12}\times H_{\kappa,0}^{\mathfrak{s},s}$.
\end{remark}

From now on, we consider the case $\kappa=1$, since all the computations are similar in the general case.

\subsection{Variational Lyapunov-Schmidt decomposition and bifurcation from double eigenvalues}

\medskip

Similarly as done for getting the liner operator $L_0$, we linearize $\mF(\om,u)$ at the rotating circle $u=0$:
\begin{align}\mL_\om[\eta,\b]:&=d\mF(\om,0)[\eta,\b] \notag
\\&=d\mC(0)^*\circ d\grad\mH(0)\circ d\mC(0)[\eta,\b] - \om d\grad\bar\mI(0)[\eta,\b] \notag
\\&=\begin{pmatrix}-\s_0\pa_{\th\th} - \s_0  + \frac{\a_0^2}{4}\Pi_0 - \frac{\a_0^2}{4}\Pi_0^\perp  -  \frac{\a_0^2}{4}\Pi_0^\perp\pa_\th^{-1}G(0)\pa_\th^{-1}\Pi_0^\perp & -\frac{\a_0}{2}\Pi_0^\perp\pa_\th^{-1} G(0) + (\om-\frac{\a_0}{2})\pa_\th \\ \frac{\a_0}{2} G(0)\pa_\th^{-1}\Pi_0^\perp - (\om-\frac{\a_0}{2})\pa_\th & G(0)
\end{pmatrix}\begin{pmatrix}\eta \\ \b
\end{pmatrix}
\label{variational.linearization}
\end{align}
Let us expand $\eta,\beta$ as in \eqref{fourier.expansion} and then we get
\begin{equation}\label{L.is.sum.of.blocks}\mL_\om[\eta,\b]=\sum_{(\ell,m)\in\mT}\mL_\om^{(\ell,m)}\begin{pmatrix}\eta_{\ell,m}\ph_{\ell,m} \\ \b_{\ell,-m}\ph_{\ell,-m}
\end{pmatrix},
\end{equation}
where
\begin{align}\label{L.lm.blocks}\mL_\om^{(\ell,m)}:=\begin{cases}\begin{pmatrix}\s_0\ell^2 - (\s_0 + \frac{\a_0^2}{4}) + \frac{\a_0^2}{4\ell} & \frac{m\a_0}{2} + m\ell(\om-\frac{\a_0}{2}) \\ \frac{m\a_0}{2} + m\ell(\om-\frac{\a_0}{2}) & \ell
\end{pmatrix} & \text{if}\quad(\ell,m)\ne(0,0)
\\\begin{pmatrix}-\s_0 + \frac{\a_0^2}{4} & 0 \\ 0 & 0
\end{pmatrix} & \text{if}\quad(\ell,m)=(0,0)
\end{cases}
\end{align}

\begin{remark}  \label{rem:excluding.v00}
We notice that $\mathtt{v}_{0,0}:=(0,1)$ is formally in $\ker\mL_\om$ for all $\om\in\R$. However, since $\b\in \dot H^{\mathfrak{s},s+1}$, we can fix $\b_{0,0}=0$. It follows that $\mathtt{v}_{0,0}$ is not in $\ker\mL_\om$. 
\end{remark}

The kernel of $\mL_{\om}$ is nonzero if and only if there exists some $\ell\in\N$ satisfying the equation
\begin{equation}\label{resonance.equation}F(\s_0,\a_0,\om,\ell)=\det\mL_{\om}^{(\ell,m)}=\s_0\ell^2 - \Big(\om-\frac{\a_0}{2}\Big)^2\ell  - \Big[\s_0 + \a_0\Big(\om-\frac{\a_0}{2}\Big) + \frac{\a_0^2}{4}\Big]  = 0.
\end{equation}
To establish the dimension of $\ker\mL_{\om}$, we need to count the number of solutions of \eqref{resonance.equation}. First, let us notice that $F(\s_0,\a_0,\om,\ell)$ does not depend on $m$, so the dimension of $\ker\mL_{\om}$ is the double of the number of solutions of \eqref{resonance.equation}. Second, the number of integers $\ell$ satisfying \eqref{resonance.equation} is at most $2$ because $F$ is a quadratic function in $\ell$, and so $\dim(\ker\mL_{\om})\in\{0,2,4\}$. Third, the existence of frequencies solving \eqref{resonance.equation} is \eqref{existence.resonance.condition}, which we recall here:
\begin{equation*}\Delta:=(\ell-1)\Big(C\ell(\ell+1) - \frac14\Big)\ge0,\qquad C=\frac{\s_0}{\a_0^2}.
\end{equation*}
In such cases, the frequencies are
\begin{equation}\label{resonant.frequencies}\om_*^{(\pm)}=\frac{\a_0}{2} - \frac{\a_0}{2\ell} \pm \sqrt{\Big(\frac{\a_0}{2\ell}\Big)^2 + \frac{4\s_0(\ell^2-1)-\a_0^2}{4\ell}}.
\end{equation}
In the following lemma, we prove that for large values of $\ell$, the frequencies \eqref{resonant.frequencies} are double, but otherwise they can be also double.
\begin{lemma}\label{lemma:simple.eigenvalue} 

$(i)$ Fixed any $\s_0>0,\,\a_0>0,\,\om\in\R$ such that $\Delta>0$, let us set 
\begin{equation*}f(\ell):= \Big(\frac{\a_0}{2\ell}\Big)^2 + \frac{4\s_0(\ell^2-1)-\a_0^2}{4\ell},\qquad g_\pm(\ell):=-\frac{\a_0}{2\ell} \pm \sqrt {f(\ell)}.
\end{equation*}
Then, there exists $\ell_0\in\N$ such that for any $\ell\ge\ell_0$,
\begin{equation}\label{resonance.double.l.large}g_+'(\ell)>0,\quad g_-'(\ell)<0.
\end{equation}

\medskip

$(ii)$ There exists a choice $\s_0>0,\,a_0>0,\,\om\in\R$ such that the equation $F(\s_0,\a_0,\om,\ell)=0$ admits two distinct integer solutions $\ell_{\pm}\in\N$. 

\end{lemma}

\begin{proof} $(i)$ We have
\begin{equation*}g_\pm'(\ell)=\frac{\a_0}{2\ell^2}\Big[1 \pm \underbrace{\frac{\s_0\ell^2}{\a_0\sqrt{f(\ell)}}}_{\sim\frac{\s_0}{\a_0}\ell^\frac32}\underbrace{\Big(1 + \frac{\a_0^2 + 4\s_0}{4\s_0\ell^2} - \frac{\a_0^2}{\s_0\ell^3} \Big)}_{\sim1} \Big]
\end{equation*}
for $\ell$ large, which implies \eqref{resonance.double.l.large}.

\medskip

$(ii)$ Let us fix any $\a_0>0$, and let consider $c_1\in\R,\,c_2>0$ to be fixed so that $\om=c_1\a_0,\,\s_0=c_2\a_0^2$. Thus, we have
\begin{align*}\ell_{\pm}
&=\frac{c_1^2}{2c_2}\pm\frac{\sqrt{c_1^4 + 4c_2(1+2c_1+c_2)}}{2c_2}
=\frac{c_1^2}{2c_2}\pm\frac{\sqrt{c_1^4 + 4c_2(1+2c_1) + 4c_2^2}}{2c_2}.
\end{align*}
If we look for $c_1$ such that $-c_1^2=1+2c_1$, then $c_1=-1$. Therefore,
\begin{equation*}\ell_\pm=\frac{1}{2c_2}\pm\frac{1-2c_2}{2c_2}=\frac{1}{2c_2}\pm\Big(\frac{1}{2c_2}-1\Big)=\begin{cases}\ell_+=\frac{1}{c_2}-1, 
\\\ell_-=1.
\end{cases}
\end{equation*}
It is enough then to choose $c_2\in(0,\frac13]$ so that $\frac{1}{c_2}\in[3,+\infty)\cap\N$. Thus, we get point $(iii)$.
\end{proof}
We are now ready to do the Lyapunov-Schmidt decomposition. Let us assume that there exist some $\s_0,\,\a_0>0$ and $\ell\in\N$ such that $\om$ assume one of the two values in \eqref{resonant.frequencies}.Let us call $\om_*$ this value.

We call $S:=S_2:=\{(\ell_\s,m)\colon\,m\in\{-1,1\},\,\s=0\}$ if the resonances are double, and $S:=S_4:=\{(\ell_\s,m)\colon\,m\in\{-1,1\},\,\s\in\{-,+\}\}$ if the resonances have multiplicity $4$. In any case,
\begin{align}&V:=\ker\mL_{\om_*}=\Big\{v=\sum_{(\ell,m)\in S}v_{\ell,m}\mathtt{v}_{\ell,m} \colon\,v_{\ell,m}\in\R, \notag
\\&\qquad\qquad\qquad\qquad\qquad\mathtt{v}_{\ell,m}=\frac{1}{\sqrt{1 + (\om_*-\frac{\a_0}{2}+\frac{\a_0}{2\ell}})^2}\begin{pmatrix}\ph_{\ell,m} \\ -[\frac{m\a_0}{2\ell} +m(\om_*-\frac{\a_0}{2})]\ph_{\ell,-m}
\end{pmatrix}\Big\}, \label{V.kernel}
\\&W:=V^{\perp_{(L^2)^2}}=\Big\{u=\mJ v + \sum_{(\ell,m)\in\mT\setminus S}\begin{pmatrix}\eta_{\ell,m}\ph_{\ell,m} \\ \b_{\ell,-m}\ph_{\ell,-m}
\end{pmatrix}\colon\,v\in V,\quad(\eta_{\ell,m},\b_{\ell,-m})\in\ell^2(\N)\times\ell^2(\N) \Big\} \label{W}
\\&W^{\mathfrak{s},s}:=W\cap(H^{\mathfrak{s},s+\frac32}\times H_0^{\mathfrak{s},s+1}). \label{W.regular}
\end{align}
We now compute the range:
\begin{align}&\mL_{\om_*}\mJ\mathtt{v}_{\ell,m}=\ell\Big[\frac{m\a_0}{2\ell} + m\Big(\om_*-\frac{\a_0}{2}\Big)\Big]^2\mJ\mathtt{v}_{\ell,m}, \label{range.1}
\\&\mL_{\om_*}\begin{pmatrix}\eta_{\ell,m}\ph_{\ell,m} \\ \b_{\ell,-m}\ph_{\ell,-m}
\end{pmatrix}=\begin{pmatrix}[\s_0\ell^2 - \s_0 - \frac{\a_0^2}{4} + \frac{\a_0^2}{4\ell}]\eta_{\ell,m}\ph_{\ell,m} + [\frac{m\a_0}{2} + m\ell(\om_*-\frac{\a_0}{2})]\b_{\ell,-m}\ph_{\ell,m} \\ [\frac{m\a_0}{2} + m\ell(\om_*-\frac{\a_0}{2})]\eta_{\ell,m}\ph_{\ell,-m} + \ell\b_{\ell,-m}\ph_{\ell,-m}
\end{pmatrix}. \label{range.2}
\end{align}
Therefore, the range and its orthogonal complement in $L^2\times L^2$ are
\begin{align}&R:=\mL_{\om_*}W=W,\qquad R^{\mathfrak{s},s}:=R\cap(H^{\mathfrak{s},s-\frac12}\times H_0^{\mathfrak{s},s}), \label{R}
\\&Z:=R^{\perp_{(L^2)^2}}.
\end{align}
By Remark \ref{rem:excluding.v00} and \eqref{range.1}, \eqref{range.2}, we have $Z=V$.
We now show that $V,W$ are invariant under torus and reversibility actions.
\begin{lemma} For all $\a\in\T^1$, we have
\begin{align}&\mT_\a V=V,\qquad\mT_\a W= W,\label{torus.action.invariance.barV}
\\&\mR V= V,\qquad\mR W= W. \label{reversibility.action.invariance.barW}
\end{align}    
Moreover, we have
\begin{equation}\label{matrix.actions}\mT_\a\begin{pmatrix}\mathtt v_{\ell,1} \\ \mathtt v_{\ell,-1}
\end{pmatrix}=\begin{pmatrix}\cos(\ell\a) & -\sin(\ell\a) \\ \sin(\ell\a) & \cos(\ell\a)
\end{pmatrix} \begin{pmatrix}\mathtt v_{\ell,1} \\ \mathtt{v}_{\ell,-1}
\end{pmatrix},\qquad\mR\begin{pmatrix}\mathtt v_{\ell,1} \\ \mathtt{v}_{\ell,-1}
\end{pmatrix}=\begin{pmatrix}1 & 0 \\ 0 & -1
\end{pmatrix}\begin{pmatrix}\mathtt v_{\ell,1} \\ \mathtt{v}_{\ell,-1}
\end{pmatrix}
\end{equation}
\end{lemma}

\begin{proof} Let us start showing \eqref{torus.action.invariance.barV}. We have
\begin{align*}\mT_\a\mathtt{v}_{\ell,1}(\th)
&=\mathtt{v}_{\ell,1}(\th + \a)
\\&=\frac{1}{\sqrt{1 + (\om_* - \frac{\a_0}{2} + \frac{\a_0}{2\ell})^2}}\begin{pmatrix}\cos(\ell\th + \ell\a) \\ -[\frac{\a_0}{2\ell} + (\om_* - \frac{\a_0}{2})]\sin(\ell\th + \ell\a)
\end{pmatrix}
\\&=\cos(\ell\a)\,\mathtt{v}_{\ell,1}(\th) - \sin(\ell\a)\,\mathtt{v}_{\ell,-1},
\end{align*}
and similarly
\begin{align*}\mT_\a\mathtt{v}_{\ell,-1}(\th)=\sin(\ell\a)\,\mathtt{v}_{\ell,1}(\th) + \cos(\ell\a)\,\mathtt{v}_{\ell,-1}.\end{align*}
As a result, both for $\dim V=2$ and $\dim V=4$, we get the first identity in \eqref{torus.action.invariance.barV}. The second identity follows by $W=V^{\perp_{(L^2)^2}}$.

We now show the identities \eqref{reversibility.action.invariance.barW}. We have
\begin{equation}\label{R.action.kernel}\mR\mathtt{v}_{\ell,1}(\th)
=\mathtt{v}_{\ell,1},\quad\mR\mathtt{v}_{\ell,-1}(\th)
=\begin{pmatrix}-\sin(\ell\th) \\ -[\frac{\a_0}{2\ell} + (\om_* - \frac{\a_0}{2})\cos(\ell\th)
\end{pmatrix} = -\mathtt{v}_{\ell,-1},
\end{equation}
and so we have $\mR V=V$. Reasoning as above, we also get the second identity in \eqref{reversibility.action.invariance.barW}, and we conclude. 
\end{proof}
We now define the norms $\|\cdot\|_{W^{\mathfrak{s},s}}:=\|\cdot\|_{H^{\mathfrak{s},s+\frac32}\times H^{\mathfrak{s},s+1}},\quad\|\cdot\|_{R^{\mathfrak{s},s}}:=\|\cdot\|_{H^{\mathfrak{s},s-\frac12}\times H^{\mathfrak{s},s}}$.
The following lemma holds:
\begin{lemma}\label{L.invertibility} In the notations above, the linear operator $\mL_{\om_*}|_{W^{\mathfrak{s},s}}\colon W^{\mathfrak{s},s}\longmapsto R^{\mathfrak{s},s}$ is well-defined and bijective; in particular, it satisfies the estimate
\begin{equation}\label{L.homeo.est.}\|\mL_{\om_*}(\eta,\beta)\|_{R^{\mathfrak{s},s}}\lesssim_{\s_0,\a_0,\om_*} \|(\eta,\beta)\|_{W^{\mathfrak{s},s}},\qquad\forall(\eta,\beta)\in W^{\mathfrak{s},s}.
\end{equation}
As a result, the inverse operator of $\mL_{\om_*}$, namely $(\mL_{\om_*}|_{W^{\mathfrak{s},s}})^{-1}\colon R^{\mathfrak{s},s}\longmapsto W^{\mathfrak{s},s}$, is continuous.
\end{lemma}

\begin{proof} We prove \eqref{L.homeo.est.} only for $\mathfrak{s}=0$, as in the other cases the proof follows identically. Fixed $(\eta,\beta)\in W^s$, by \eqref{range.2} we get
\begin{align*}\|\mL_{\om_*}(\eta,\beta)\|_{R^s}^2 & = \sum_{(\ell,m)\in\mT\setminus S}\ell^{2(s-\frac12)}\Big\{\Big[\s_0\ell^2 - \s_0 - \frac{\a_0^2}{4} - \frac{\a_0^2}{4\ell}\Big]\eta_{\ell,m} + \Big[\frac{m\a_0}{2} + m\ell\Big(\om_*-\frac{\a_0}{2}\Big)\Big]\b_{\ell,-m}\Big\}^2 
\\&\qquad\qquad+ \sum_{(\ell,m)\in\mT\setminus S}\ell^{2s}\Big\{\Big[\frac{m\a_0}{2} + m\ell\Big(\om_* - \frac{\a_0}{2}\Big)\Big]\eta_{\ell,m} + \ell\b_{\ell,-m}\Big\}^2
\\&\lesssim\sum_{(\ell,m)\in\mT\setminus S}\ell^{2(s-\frac12)}\Big\{\s_0^2\ell^4\eta_{\ell,m}^2 + \Big(\om_*-\frac{\a_0}{2}\Big)^2\ell^2\b_{\ell,-m}^2\Big\} 
\\&\qquad\qquad+ \sum_{(\ell,m)\in\mT\setminus S}\ell^{2s}\Big\{\Big(\om_*-\frac{\a_0}{2}\Big)^2\ell^2\eta_{\ell,m}^2 + \ell^2\b_{\ell,-m}^2\Big\}
\\&\lesssim_{\s_0,\a_0,\om_*}\sum_{(\ell,m)\in\mT\setminus S}[\ell^{2(s+\frac32)}\eta_{\ell,m}^2 + \ell^{2(s+1)}\beta_{\ell,m}^2]
\\&\lesssim_{\s_0,\a_0,\om_*}\|(\eta,\beta)\|_{W^s}.
\end{align*}
We conclude by Open Mapping Theorem.
\end{proof}    

We are now ready to make the Lyapunov-Schmidt decomposition of \eqref{rotating.wave.equation.0}. Each $u\in U$ can be uniquely written as $u=v+w$ for some $v\in V,\,w\in W^{\mathfrak{s},{s}}$. 
Defining the projection maps $\Pi_{R^{\mathfrak{s},s}}\colon H^{\mathfrak{s},s-\frac12}\times H_0^{\mathfrak{s},s}\longmapsto R^{\mathfrak{s},s}$ and $\Pi_{Z}\colon\,H^{\mathfrak{s},s-\frac12}\times  H_0^{\mathfrak{s},s}\longmapsto Z$, the equation \eqref{rotating.wave.equation.0} is equivalent to the system of equations
\begin{align}&\Pi_{R^{\mathfrak{s},s}}\mF(\omega;v+w)=0, \label{Range.eq.}
\\&\Pi_{Z}\mF(\omega;v+w)=0.\label{Bif.eq.}
\end{align}
For any fixed $\epsilon>0$, let
\begin{equation}\label{U.eps}\mU_\epsilon:=\{(\om,v)\in\R\times V\colon\,|\om-\om_*|<\epsilon,\,|v|<\epsilon\}.
\end{equation}
We now solve the range equation \eqref{Range.eq.}.
\begin{lemma}\label{Range.eq.solv.}Let $\mF=(\mF_1,\mF_2)$ as in Lemma \ref{ww.rotating.wave.eq}, and let $\om_*$ satisfy \eqref{resonant.frequencies}.

Then, there exists $\epsilon_0>0$, such that for all $(\om,v)\in\mU_{\epsilon_0}$, there exists one and only one analytic function $(\om,v)\in\mU_{\epsilon_0}\longmapsto w(\om,v)\in W^{\mathfrak{s},s}$ such that 
\begin{equation}\label{zeros.implicit.func.}\Pi_{R^{\mathfrak{s},s}}\mF(\om,v+w(\om,v))=0.\end{equation}
Moreover, for all $(\om,v)\in\mU_{\epsilon_0}$ and $\b\in\T^1$,
\begin{align}&w(\om,0)=0,\qquad\pa_v w(\om,0)=0, \label{w.values}
\\&w(\om,v)=O(|v|^2+|v|\cdot|\om-\om_*|),\label{w.est.}
\\&w(\om,\mT_\b v) = \mT_\b w(\om,v), \label{w.torus.action}
\\&w(\om,\mR v) = \mR w(\om,v). \label{w.reversibility.action}
\end{align}

\end{lemma}

\begin{proof} The solvability of the range equation \eqref{Range.eq.} and the analiticity of $w(\om;v)$ in $\mU_{\epsilon_0}$ can be immediately deduced by Lemma \ref{L.invertibility}, thanks to which we can apply Implicit Function Theorem.
The first identity in \eqref{w.values} still follows from Implicit Function Theorem. The second identity follows from the fact that differentiating both sides of \eqref{zeros.implicit.func.} with respect to $v$ along any direction $\tilde v\in V$, we get
\begin{equation*}0=\Pi_{R^{\mathfrak{s},s}}d\mF(\om_*;0)[\tilde v + dw(\om_*;0)\tilde v]=\Pi_{R^{\mathfrak{s},s}}\mL_{\om_*}|_{W^{\mathfrak{s},s}}[d w(\om_*;0)\tilde v],
\end{equation*}
which, by Lemma \ref{L.invertibility}, necessarily implies $d w(\om_*;0)\tilde v=0$ for all $\tilde v\in V$ and so $\pa_v w(\om_*;0)=0$.
The estimate \eqref{w.est.} follows by the Taylor expansion of $w$ about $(\om_*;0)$. In particular, since $\mF(\om;u)$ is linear in $\om$, necessarily $\pa_\om^k w(\om_*;0)=0$ for all $k\in\N$.
As for \eqref{w.torus.action}, by \eqref{zeros.implicit.func.} and by \eqref{Hamiltonian.equivariances}-\eqref{Volume.equivariances}, we have that for any $\alpha\in\T^1$ and for any $(\om,v)\in\mU_{\epsilon_0}$,
\begin{equation*}0=\mT_\alpha\Pi_{R^{\mathfrak{s},s}}\mF(\om;v+w(\om;v))=\Pi_{R^{\mathfrak{s},s}}\mT_\alpha\mF(\om;v+w(\om;v))=\Pi_{R^{\mathfrak{s},s}}\mF(\om;\mT_\alpha v+\mT_\alpha w(\om;v)).
\end{equation*}
Moreover, since $(\om,\mT_\alpha v)\in\mU_{\epsilon_0}$, we also have
\begin{equation*}\Pi_{R^{\mathfrak{s},s}}\mF(\om;\mT_\alpha v + w(\om;\mT_\alpha v))=0,
\end{equation*}
By uniqueness of $w(\om;v)$, one gets \eqref{w.torus.action}. Similarly, we get \eqref{w.reversibility.action}
    
\end{proof}

We are now left with solving the bifurcation equation
\begin{equation}\label{bif.eq.reduced}\Pi_Z\mathcal F(\omega;v + w(\omega,v))=0.
\end{equation}
Let us consider the following quantity:
\begin{align}&\Phi(\om,v):=\bar\mE(\om,v+w(\om,v)). \label{action.on.V}
\end{align}
Now, we get the asymptotics for $\Phi(\om,v)$ about the bifurcation points $(\om_*,0)$.
\begin{lemma}\label{lemma:taylor.expansions} For all $(\om,v)\in\mU_{\e_0}$, we have
\begin{align}&\pa_\om\mL_\om(\om_*)=\begin{pmatrix}0 & \pa_\th \\ -\pa_\th & 0
\end{pmatrix}, \label{linearization.om.derivative}
\\&\Phi(\om,v)=\frac12(\om-\om_*)\la\pa_\om\mL_\om(\om_*)v,v\ra_{L^2\times L^2} + O(|v|^3),\label{action.expansion}
\\&\Phi(\omega,\mR v) = \Phi(\omega,v),\qquad\Phi(\omega,\mathcal{T}_\alpha v)=\Phi(\omega,v),\qquad\forall\alpha\in\T^1. \label{reduced.energy.torus.invariance}
\end{align}

\end{lemma}

\begin{proof} \eqref{linearization.om.derivative} is a direct check.

Let us now prove \eqref{action.expansion}. First of all, we observe that for all $(\om,v)\in\mU_{\e_0},\,\hat v\in V$,
\begin{align}d\Phi(\om,v)\hat v
&=d(\bar\mE(v+w(\om,v)))\hat v \notag
\\&=d\bar\mE(v+w(\om,v))[\hat v + dw(\om,v)\hat v] \label{action.first.derivative},
\end{align}
and
\begin{align}d^2\Phi(\om,v)[\hat v, \hat v]=d^2\bar\mE(v+w(\om,v))[\hat v + dw(\om,v)\hat v, \hat v + dw(\om,v)\hat v] + d\bar\mE(v+w(\om,v))d^2 w(\om,v)[\hat v, \hat v]. \label{action.second.derivative}
\end{align}
By these computations, by Lemma \ref{Range.eq.solv.}, by $\mL_{\om_*}|_V\equiv0$ and by linearity of $\mL_\om$ in $\om$, it follows that for all $v\in V$,
\begin{align*}&d\Phi(\om,0)v=0,
\\&d^2\Phi(\om,0)[v,v]=d^2\bar\mE(0)[v,v]=\la\mL_\om v,v\ra_{L^2\times L^2}=(\om-\om_*)\la\pa_\om \mL_\om(\om_*)v,v\ra_{L^2\times L^2}.
\end{align*}
Putting all together, we get \eqref{action.expansion}.

The identities \eqref{reduced.energy.torus.invariance} follow from \eqref{Hamiltonian.invariances}, \eqref{Angular.momentum.invariances} and \eqref{w.reversibility.action}, \eqref{w.torus.action}.
\end{proof}

Let us now observe the variational nature of \eqref{bif.eq.reduced}.
\begin{lemma}\label{lemma:bif.eq.is.variational} We have that for $(\om,v)\in\mU_{\e_0}$,
\begin{equation}\label{bif.equation.is.variational}\Pi_Z\mF(\om,v+w(\om,v))=\grad_v\Phi(\om,v),
\end{equation}
where $\grad_v\Phi$ is the gradient of $\Phi$ with respect to $v$. As a consequence, $(\om,v)\in\mU_{\e_0}$ solves the bifurcation equation \eqref{bif.eq.reduced} if and only if it is a critical point for $\Phi(\om,v)$. Also, we have the equivariances
\begin{equation}\label{reduced.bif.equiv.}\mR\circ\Pi_Z\mF(\omega,\cdot) = \Pi_Z\mF(\omega,\cdot)\circ\mR,\qquad\mT_\a\circ\Pi_Z\mF(\omega,\cdot) = \Pi_Z\mF(\omega,\cdot)\circ\mT_\a,\qquad\forall\a\in\T^1.
\end{equation}
\end{lemma}

\begin{proof} For all $(\om,v)\in\mU_{\e_0},\,\hat v\in V$, we have by \eqref{action.first.derivative} and \eqref{zeros.implicit.func.}
\begin{align*}d\Phi(\om,v)\hat v
&=d\bar\mE(\om,v+w(\om,v))[\hat v + dw(\om,v)\hat v]
\\&=\la\grad\bar\mE(\om,v+w(\om,v)),\hat v\ra_{(L^2)^2} + \la\grad\bar\mE(\om,v+w(\om,v)),dw(\om,v)\hat v\ra_{(L^2)^2}
\\&=\la\mF(\om,v+w(\om,v)),\hat v\ra_{(L^2)^2} + \la\mF(\om,v+w(\om,v)), dw(\om,v)\hat v\ra_{(L^2)^2}
\\&=\la\Pi_Z\mF(\om,v+w(\om,v)),\hat v\ra_{(L^2)^2} + \la\Pi_{R^{\mathfrak{s},s}}\mF(\om,v+w(\om,v)), dw(\om,v)\hat v\ra_{(L^2)^2}
\\&=\la\Pi_Z\mF(\om,v+w(\om,v)),\hat v\ra_{(L^2)^2},
\end{align*}
from which we get \eqref{bif.equation.is.variational}.
\end{proof}

We now prove the existence of rotating waves when $\om$ has multiplicity $2$.

\begin{theorem} \label{thm:rw.crandall.rabinowitz} 
Let $\kappa\in\N$, let $\a_0>0,\,\mathfrak{s}\ge0,\,s_0>0$ with $s\ge s_0>1$, let $\ell\in\N$ be fixed and, finally, let
\begin{align*}&\om_*=\om_*^{(+)}:=\frac{\a_0}{2} - \frac{\a_0}{2\kappa\ell} + \frac{1}{2\kappa\ell}\sqrt{(\kappa\ell-1)(4\s_0\kappa\ell(\kappa\ell+1)-\a_0^2)}.
\end{align*} 
Let us suppose that \eqref{existence.resonance.condition} holds,
and that $\om_*^{(+)}$ has multiplicity 2.
Then there exist $\e_0 > 0$, $C > 0$ such that, for every $\e \in (-\e_0, \e_0)$, 
there exists one and only $\mT_\a-$orbit
\begin{equation}
\{ (\eta_\e , \beta_\e )\circ\mT_\alpha:\,\alpha \in \T^1 \}
\end{equation}
of nonzero solutions of the rotating wave equations \eqref{rotating.wave.equation.0} satisfying the $\kappa-$fold symmetry, having angular velocity $\om = \om_\e$ and 
\begin{equation}
|\om_\e - \om_*| + \| \eta_\e \|_{\mathfrak{s},s+\frac32} + \| \beta_\e \|_{\mathfrak{s},s+1} = O(\e).
\end{equation}
Moreover, the $\mT_\a-$orbit depends analytically on $\e$ in the interval $(0, \e_0)$, and it is generated by the rotating wave
\begin{equation}\label{reflection.invariant.rw}u_\e:=\begin{pmatrix}\eta_\e \\ \b_\e
\end{pmatrix}=\e\mathtt v_{\ell,1} + O(\e^2),
\end{equation}
which is $\mR-$invariant (that is, $\mR u_\e=u_\e$).

The same holds for bifurcation frequencies
\begin{align*}&\om_*^{(-)}:=\frac{\a_0}{2} - \frac{\a_0}{2\kappa\ell} - \frac{1}{2\kappa\ell}\sqrt{(\kappa\ell-1)(4\s_0\kappa\ell(\kappa\ell+1)-\a_0^2)}
\end{align*} 
when $\om_*^{(-)}$ has multiplicity 2.

\end{theorem}

\begin{proof} 
By the equivariances \eqref{reduced.bif.equiv.}, $V$ can be restricted to the subspace generated by $\mathtt v_{\ell,1}$, over which we can solve the bifurcation equation \eqref{bif.eq.reduced}. We are then reduced to check the Crandall-Rabinowitz transversality condition.

We can now check that for all $(\ell,m)\in S$,
\begin{align}&\pa_\om\mL_\om(\om_*){\mathtt{v}}_{\ell,m}
=\frac{1}{\sqrt{1+(\om_*-\frac{\a_0}{2})^2}}\begin{pmatrix}0 & \pa_\th \\ -\pa_\th & 0
\end{pmatrix}\begin{pmatrix}\ph_{\ell,m} \\-[\frac{m\a_0}{2\ell} + m(\om_*-\frac{\a_0}{2})]\ph_{\ell,-m}
\end{pmatrix}  \notag
\\&\qquad\qquad\qquad\qquad=\frac{1}{\sqrt{1+(\om_*-\frac{\a_0}{2})^2}}\begin{pmatrix}[-\frac{\a_0}{2} - \ell(\om_*-\frac{\a_0}{2})]\ph_{\ell,m} \\ m\ell\ph_{\ell,-m}
\end{pmatrix}, \label{pa.om.L.action}
\\&\la\pa_\om\mL_\om(\om_*)\mathtt{v}_{\ell,m},\mathtt{v}_{\ell,m}\ra_{L^2\times L^2}=\frac{1}{1+(\om_*-\frac{\a_0}{2})^2}\la \begin{pmatrix}[- \frac{\a_0}{2} - \ell(\om_*-\frac{\a_0}{2})]\ph_{\ell,m} \\ m\ell\ph_{\ell,-m}
\end{pmatrix} \notag
\\&\qquad\qquad\qquad\qquad\qquad\qquad\qquad\qquad\qquad\qquad\qquad\quad,\begin{pmatrix}\ph_{\ell,m} \\ -[\frac{m\a_0}{2\ell} + m(\om_*-\frac{\a_0}{2})]\ph_{\ell,-m}
\end{pmatrix}\ra_{L^2\times L^2} \notag
\\&\qquad\qquad\qquad\qquad\qquad\qquad\quad\quad=-\frac{\a_0 + 2\ell(\om_*-\frac{\a_0}{2})}{1+(\om-\frac{\a_0}{2})^2}\neq0, \label{pa.om.L.scalar.product}
\end{align}
if and only if $\om_*\ne\frac{\a_0}{2}-\frac{\a_0}{2\ell}$. Thus, we have the Crandall-Rabinowitz transversality condition for simple eigenvalues, which concludes the proof.

\end{proof}

\subsection{Existence of rotating waves by angular momentum parametrization}

In this section, we find rotating waves as bifurcations from eigenvalue of multiplicity $4$ by Moser-Weinstein approach. 
In particular, we want to prove the following theorem.
\begin{theorem} \label{thm:rw.angular.momentum} 
Let $\kappa\in\N$ and $\a_0>0,\,\mathfrak{s}\ge0,\,s_0>0$ with $s\ge s_0>1$, let $\ell_0\in\N$ be fixed and, finally, let us suppose that \eqref{existence.resonance.condition} holds.
Let
\begin{align*}&\om_*=\om_*^{(+)}:=\frac{\a_0}{2} - \frac{\a_0}{2\kappa\ell_0} + \frac{1}{2\kappa\ell_0}\sqrt{(\kappa\ell_0-1)(4\s_0\kappa\ell_0(\kappa\ell_0+1)-\a_0^2)},
\end{align*} 
and suppose also that $\om_*>\frac{\a_0}{2}$ and its multiplicity is 4.
Then, there exist $a_0 > 0$, $C > 0$ such that, for every $a \in (0, a_0)$, 
there exists at least two distinct $\mT_\a-$orbits
\begin{equation}\label{orbit}
\{ (\eta_a^{(i)} , \beta_a^{(i)} )\circ\mT_\alpha : \alpha \in \T^1,\,i=1,2 \}
\end{equation}
of nonzero solutions of the rotating traveling wave equations \eqref{rotating.wave.equation.0} satisfying the $\kappa-$fold symmetry, having angular velocity $\om = \om_a$, 
angular momentum 
\begin{equation}
\mI(\eta_a , \beta_a) = a
\end{equation}
and 
\begin{equation}
|\om_a - \om_*| + \| \eta_a \|_{\mathfrak{s},s+\frac32} + \| \beta_a \|_{\mathfrak{s},s+1} 
\leq C \sqrt{a}.
\end{equation}

\end{theorem}

From now on, we suppose that $\dim V=4$, and we denote by $|\cdot|$ the $(L^2)^2-$norm over $V$, and $\|\cdot\|$ the $W^{\mathfrak{s},s}-$norm over $W^{\mathfrak{s},s}$

We consider the following functions:
\begin{align}&f(\om,v):=\Pi_Z\mF(v+w(\om,v)), \label{F.projection.onto.Z}
\\&g(\om,v):=\grad\bar\mI(v+w(\om,v)), \label{grad.modified.I.}
\\&F(\om,v):=\la f(\om,v),g(\om,v)\ra_{(L^2)^2}. \label{F.def.}
\end{align}
Close to bifurcation points, they can be described as follows.
\begin{lemma}\label{lemma:taylor.expansions.for.moser.weinstein}For all $(\om,v)\in\mU_{\e_0}$, the following facts hold.

\medskip

$(i)$ $f\colon\,\mU_{\e_0}\to Z$ and $g\colon\,\mU_{\e_0}\to(L^2)^2$ are both analytic in $\mU_{\e_0}$.

\medskip

$(ii)$ The following formulae in $\mU_{\e_0}$ hold:
\begin{align}&f(\om,0)=0,\quad\pa_v f(\om_*,0)=0,\quad\pa_{\om v}f(\om_*,0)=\Pi_Z\pa_\om\mL_\om(\om_*), \label{f.expansions}
\\&g(\om,0)=0,\quad\pa_v g(\om_*,0)=d\grad\bar\mI(0)=-\pa_\om\mL_\om(\om_*), \label{g.expansions}
\\&F(\om,0)=0\quad\pa_vF(\om,0)=0,\quad\pa_{vv}F(\om_*,0)=0, \notag
\\&\la\pa_{\om vv}F(\om_*,0)v,v\ra_{(L^2)^2}= -2|\Pi_Z\pa_\om\mL_\om(\om_*)v|^2,\qquad\forall v\in V \label{F.expansions}
\end{align}

\end{lemma}

\begin{proof} $(i)$ It follows by Lemma \ref{Range.eq.solv.}.

$(ii)$ The formulae \eqref{f.expansions}, \eqref{g.expansions} can be obtained with similar computations done in Lemma \ref{lemma:taylor.expansions}. As for \eqref{F.expansions}, we have $F(\om,0)=0,\,\pa_vF(\om,0)=0,\,\pa_{vv}F(\om_*,0)=0$ by \eqref{f.expansions} and \eqref{g.expansions}, since
\begin{align*}&\pa_vF(\om,v)=\la\pa_vf(\om,v),g(\om,v)\ra_{(L^2)^2} + \la f(\om,v),\pa_vg(\om,v)\ra_{(L^2)^2},
\\&\pa_{vv}F(\om,v)=\la\pa_{vv}f(\om,v),g(\om,v)\ra_{(L^2)^2} + 2\la\pa_vf(\om,v),\pa_vg(\om,v)\ra_{(L^2)^2} + \la f(\om,v),\pa_{vv}g(\om,v)\ra_{(L^2)^2}.
\end{align*}
By these two expressions and \eqref{f.expansions}, \eqref{g.expansions}, we also observe that
\begin{align*}\pa_{\om vv}F(\om_*,0)[v,v]
&=2\la\pa_{\om v}f(\om_*,0)v,\pa_v g(\om_*,0)v\ra_{(L^2)^2}
\\&=-2|\Pi_Z\pa_\om\mL_\om(\om_*)v|^2
\end{align*}

\end{proof}

We then make the following choice of $\om=\om(v)$.
\begin{lemma}
\label{lemma:choice.of.omega} 
In the assumptions of Lemma \ref{lemma:taylor.expansions.for.moser.weinstein}, there exist $\e_1 \in (0, \e_0]$, $b_1, C > 0$ 
and a function $\om : B_{V}(\e_1) \to \R$, $v \longmapsto \om(v)$, 
where $B_{V}(\e_1) := \{ v \in V_N : |v| < \e_1 \}$, 
which is Lipschitz continuous in $B_{V}(\e_1)$, 
analytic in $B_{V}(\e_1) \setminus \{ 0 \}$, 
such that $\om(0) = \om_*$,
\begin{equation}  \label{IFT.om}
F(\om(v), v) = 0
\end{equation}
for all $v \in B_{V_N}(\e_1)$, 
and, if $(\om,v)$ satisfies $F(\om,v) = 0$ 
with $|\om - \om_*| < b_1$, $v \in B_{V_N}(\e_1)$, 
then $\om = \om(v)$. 
Moreover, the graph $\{ (\om(v) , v) : v \in B_{V_N}(\e_1) \} \subset \R \times V_N$ 
is contained in the open set $\mU_{\e_0} \subset \R \times V$ 
where the function $w$ constructed in Lemma \ref{Range.eq.solv.} is defined, and 
\begin{equation} \label{om.estimate}
|\om(v)-\om_*| = O(|v|)
\quad \forall v \in B_{V_N}(\e_1).
\end{equation}
Also, for all $\b \in \T^1$, we have 
\begin{equation} \label{om.group.action}
\om(\mT_\b v) = \om(v),\quad\om(\mR v)=\om(v).
\end{equation}
\end{lemma}

\begin{proof}
By \eqref{F.expansions}, we have the following Taylor expansion of $F(\om,v)$ about $(\om,0)$ in $\mU_{\e_0}$:
\begin{equation}\label{F.Taylor.in.v}F(\om,v)=-\frac12\la\pa_{ vv}F(\om_*,0)v,v\ra_{(L^2)^2} + O(|v|^3).
\end{equation}
Consider $v \in V_N$, 
and introduce polar coordinates 
$\rho = |v|$, $y = v / |v|$ on $V_N \setminus \{ 0 \}$. 
The function  
\[
\Phi(\om, \rho, y) := \rho^{-2} F(\om, \rho y)
\]
is defined for $|\om - \om_*| < \e_0$, 
$\rho \in (0, \e_0)$, $y$ in the unit sphere $\{ |y| = 1 \}$ of $V_N$.
In fact, replacing $v$ with $\rho y$ in the converging Taylor series of $F(\om,v)$ 
centered at $(\om_*, 0)$, 
one obtains that $\Phi(\om, \rho, y)$ is a well-defined, converging power series 
in the open set $\mD := \{ (\om, \rho, y) : 
|\om - \om_*| < \e_0$, $\rho \in (0, \frac12 \e_0)$, $|y| < 2\}$. 
By \eqref{F.Taylor.in.v}, $\Phi(\om, \rho, y)$ converges to $\frac12 \pa_{vv} F(\om,0)[y,y]$ 
as $\rho \to 0$, and therefore $\Phi$ has a removable singularity at $\rho = 0$. 
Hence $\Phi$ has an analytic extension to the open set 
$\mD_1 := \{ (\om, \rho, y) : |\om - \om_*| < \e_0$, $|\rho| < \frac12 \e_0$, $|y| < 2\}$. 
We also denote this extension by $\Phi$.

By \eqref{F.expansions} and \ref{pa.om.L.scalar.product}, we have 
\begin{align*}
\Phi(\om_*, 0, y_0) 
& = \tfrac12 \la\pa_{vv}F(\om_*,0)v,v\ra_{(L^2)^2} = 0, 
\\ 
\pa_\om \Phi (\om_*, 0, y_0) 
& = - |\Pi_Z\pa_\om\mL_\om(\om_*)v|^2\neq 0
\end{align*}
for any $|y_0| = 1$. 
Hence, by the implicit function theorem for analytic functions, 
there exists a function $\Om(\rho, y)$ such that $\Om(0, y_0) = \om_*$ and 
\[
\Phi( \Om(\rho, y), \rho, y) = 0
\]
for all $(\rho,y)$ with $|\rho| < \e_1$, $|y - y_0| < \e_1$, 
for some $\e_1 \in (0, \e_0]$. Moreover, $\e_1$ is independent of $y_0$  
because the unit sphere of $V_N$ is compact. 
Finally, given $v \in V_N$, $|v| < \e_1$, we define 
$\om(v) := \Om(\rho, y)$ with $\rho = |v|$ and $y = v/|v|$ for $v \neq 0$,
and $\om(0) := \om_*$. 
The Lipschitz estimate \eqref{om.estimate} holds because 
$|\om(v) - \om_*| = |\Om(\rho,y) - \Om(0,y)| \leq C \rho$. 

By \eqref{F.projection.onto.Z}, by \eqref{torus.action.invariance.barV}, by the equivariances \eqref{Hamiltonian.equivariances}-\eqref{Volume.equivariances},
and \eqref{w.torus.action},
we have that for any $\alpha\in\T^1$,
\[
\mT_\alpha f(\om, v) = f(\om, \mT_\alpha v), \quad \ 
\mT_\alpha g(\om, v) = g(\om, \mT_\alpha v).
\]
Hence, by \eqref{F.def.}, one has 
\begin{equation} \label{F.group.action}
F(\om, \mT_\alpha v) = F(\om, v)
\end{equation}
because $\mT_\alpha^*  = \mT_\alpha^{-1}$.
By \eqref{IFT.om} with $\mT_\alpha v$ in place of $v$, one has 
$F(\om(\mT_\alpha v), \mT_\alpha v) = 0$. On the other hand, 
by \eqref{F.group.action}, one also has that 
$F(\om(\mT_\alpha v), \mT_\alpha v) = F(\om(\mT_\alpha v), v)$. 
Therefore 
\[
F(\om(\mT_\alpha v), v) = 0,
\]
and, by the uniqueness property of the implicit function, 
we obtain \eqref{om.group.action}. 
\end{proof}

Now, given $B_{V}(\e_1)$ as in Lemma \ref{lemma:choice.of.omega}, let us define the following objects for all $v\in B_{V}(\e_1)$ and $a\in\R$:
\begin{align}&u(v):=v+w(\om(v);v), \label{u.v.function}
\\&\bar\mH_V(v):=\bar\mH(u(v)),\qquad\bar\mI_V(v):=\bar\mI(u(v)),\qquad\bar\mE_V(v):=\bar\mE(u(v)), \label{restricted.functionals}
\\&\mS_V(a):=\{v\in B_{V}(\e_1)\colon\,\bar\mI_V(v)=a\}.  \label{restricted.constraint}
\end{align}
These quantities are nothing less but those already introduced before, but restricted to $V$; the new one is only the constraint $\mS_V(a)$ of the conservation of the angular momentum restricted on $V$. The following lemma describes the regularity of \eqref{u.v.function} and the functionals in \eqref{restricted.functionals}.

\begin{lemma}\label{lemma:regularity}
Given the notations and the assumptions made above, the following facts hold.

\medskip

$(i)$ The function $B_{V}(\e_1) \to W^{\mathfrak{s},s}$, $v \longmapsto u(v)$ 
is analytic in $B_{V}(\e_1) \setminus \{ 0 \}$,
it is of class $C^1 ( B_{V}(\e_1) )$, 
and its differential is Lipschitz continuous in $B_{V}(\e_1)$. 
Moreover $w(\om(\mT_\b v), \mT_\b v) = \mT_\b w(\om(v), v)$, $w(\om(\mR v),\mR v)=\mR w(\om(v),v)$ for all $\b \in \T^1$, and $w(v)$ satisfies the estimate
\begin{equation}\label{w.comp.estimate}\| w(\om(v), v) \| = O(|v|^2).\end{equation}
Same holds for $u(v)$.

\medskip

$(ii)$ The functionals $\bar\mH_V:\,B_{V}(\e_1) \to \R$ 
and $\bar\mI_V:\,B_{V}(\e_1) \to \R$ 
are analytic in $B_{V}(\e_1) \setminus \{ 0 \}$
and of class $C^2(B_{V}(\e_1))$, 
and their second order differentials are Lipschitz continuous in $B_{V}(\e_1)$.
Moreover $\bar\mH_V \circ \mT_\b = \bar\mH_V,\,\bar\mH_V\circ\mR=\bar\mH_V$ 
and $\bar\mI_V \circ \mT_\b = \bar\mI_V,\,\bar\mI_V\circ\mR=\bar\mI_V$ for all $\b \in \T^1$. 
    
\end{lemma}

\begin{proof} $(i)$ The function $v \longmapsto w(\om(v),v)$ is Lipschitz continuous in $B_{V}(\e_1)$
and analytic in $B_{V}(\e_1) \setminus \{ 0 \}$ 
by Lemmas \ref{Range.eq.solv.} and \ref{lemma:choice.of.omega}.
By \eqref{w.est.} and \eqref{om.estimate}, one has 
\begin{equation} \label{w.comp.estimate}
\| w(\om(v), v) \| = O(|v|^2)
\end{equation}
for all $v \in B_{V}(\e_1)$. 
Hence the function $v \longmapsto w(\om(v),v)$ is differentiable at $v=0$ 
with zero differential. 
Its differential at any point $v \in B_{V}(\e_1) \setminus \{ 0 \}$ 
in direction $\tilde v \in V$ is 
\begin{equation}  \label{der.w.comp}
d\{ w( \om(v) , v) \}\tilde v 
= d w( \om(v) , v)[d\om(v)\tilde v] 
+ dw( \om(v), v)\tilde v.
\end{equation}
For $v \to 0$, one has 
$(\om(v), v) \to (\om_*,0)$ because the function $\om(v)$ in Lemma \ref{lemma:choice.of.omega}
is continuous. Moreover,
$(\pa_v w)(\om(v), v) \to \pa_v w(\om_*, 0) = 0$ 
and $(\pa_\om w)(\om(v), v) \to \pa_\om w(\om_*, 0) = 0$,  
because the function $w$ in Lemma \ref{Range.eq.solv.} is analytic. 
By \eqref{w.values}, $d\om(v)\tilde v$, which is defined for $v \neq 0$, remains bounded as $v \to 0$
because $\om(v)$ is Lipschitz. 
Hence $\pa_v \{ w( \om(v) , v) \} \to 0$, which implies that $w(\om(v), v)$ 
is of class $C^1 ( B_{V}(\e_1) )$. 
Moreover, 
$|d\om(v)\tilde v| \leq C |\tilde v|$ for $0 < |v| < \e_1$ 
because $\om(v)$ is Lipschitz,
and 
\[
\| (\pa_\om w)( \om(v) , v) \| \leq C |v|, \quad \  
\| d w( \om(v), v)\tilde v \| \leq C |v||\tilde v|,
\]
because the function $w$ in Lemma \ref{Range.eq.solv.} is analytic 
and by \eqref{om.estimate}.

Hence, by \eqref{der.w.comp}, 
\[
\| d\{ w( \om(v) , v) \}\tilde v \|
\leq C |v||\tilde v|
\]
for $v \in B_{V}(\e_1) \setminus \{ 0 \}$, 
so that the map $v \longmapsto \pa_v \{ w( \om(v) , v) \}$ is Lipschitz continuous around $v=0$. 

The fact that $w(\om(\mT_\alpha v), \mT_\alpha v) = \mT_\alpha w(\om(v), v)$ follows from \eqref{w.torus.action} and \eqref{om.group.action}. Finally, the same regularity properties hold for $u(v)$ because of \eqref{u.v.function}.

\medskip

$(ii)$ From Lemma
\ref{ww.rotating.wave.eq} and the identities \eqref{Hamiltonian.invariances}-\eqref{Volume.invariances} and \eqref{Hamiltonian.equivariances}-\eqref{Volume.equivariances},
we deduce the analyticity, the $C^1$ regularity with Lipschitz differentials 
and the invariance with respect to the group action $\mT_\alpha$ for both $\bar\mH_V,\,\bar\mI_V$. Now, we want to prove the higher regularity: we will do it for $\bar\mI_V$, since for $\bar\mH_V$ the computations are very similar.
By \eqref{restricted.functionals}, 
the differential of $\bar\mI_{V}$ at a point $v \in B_{V}(\e_1)$ in direction $\tilde v \in V$ is 
\[
d\bar\mI_{V}(v)\tilde v 
= d\bar\mI(u(v))[ du(v)\tilde v], 
\quad \ 
du(v)\tilde v = \tilde v + d\{ w(\om(v),v) \}\tilde v.
\]
The map $v \longmapsto d\bar\mI_{V}(v)$ is Lipschitz. 
At $v \neq 0$, its differential in direction $\tilde z \in V$ is 
\[
d^2\bar\mI_{V}(v)[\tilde z, \tilde v] 
= d^2\bar\mI(u(v))[ du(v)\tilde z, du(v)\tilde v ] 
+ d\bar\mI(u(v))[ d^2 u(v)[ \tilde z, \tilde v] ].
\]
As $v \to 0$, one has $u(v) \to 0$, $du(v)\tilde v \to \tilde v$, 
$d\bar\mI(u(v)) \to d\bar\mI(0) = 0$, and $d^2\bar\mI(u(v)) \to d^2\bar\mI(0)$, 
while the bilinear map $d^2u(v) = d^2\{ w( \om(v), v ) \}$ 
remains bounded, i.e., $\| du(v)[ \tilde z, \tilde v] \| \leq C |\tilde z||\tilde v|$
uniformly as $v \to 0$, because the map $v \longmapsto du(v)$ is Lipschitz. 
Hence, $d^2\bar\mI_{V}(v)[\tilde z, \tilde v]$ converges to 
$d^2\bar\mI_V(0)[\tilde z, \tilde v]$ as $v \to 0$. 
Similarly, one proves that 
\[
| d\bar\mI_{V}(v)\tilde v - d^2\bar\mI_V(0)[v, \tilde v] | 
\leq C |v|^2|\tilde v|
\]
for $v \neq 0$. This implies that the map $v \longmapsto d\bar\mI_{V}(v)$ is differentiable at $v=0$, 
with differential $d^2\bar\mI_{V}(0) = d^2\bar\mI_V(0)$. 
Also, from the limit already proved it follows that $\bar\mI_{V}$ is of class $C^2$. 
The Lipschitz estimate for the second order differential
\[
| d^2\bar\mI_{V}(v)[ \tilde z, \tilde v] - d^2\bar\mI_{V}(0)[ \tilde z, \tilde v] | 
\leq C |v||\tilde z||\tilde v|
\]
is proved similarly.

\end{proof}

Now, we want to provide a geometrical description of the constraint $\mS(a)$ by a Taylor expansion of $\mI(v)$ about $v=0$.
\begin{lemma}\label{constraint.top.desc.}
In the above notations and assumptions, let us further suppose $\om_*>\frac{\a_0}{2}$. We have the following facts.

\medskip

$(i)$ The functional $\bar\mI_V : B_{V}(\e_1) \to \R$ defined in \eqref{restricted.functionals} 
satisfies
\begin{equation}\label{IZ.approx}
\bar\mI_V (v) = \bar\mI_0 (v) + R(v), 
\quad \ |R(v)| = O(|v|^3),
\end{equation}
where 
\begin{equation} \label{def.mI.0}
\bar\mI_0(v) := 
-\frac12 \la \pa_\om\mL_\om(\om_*) v , v \ra_{(L^2)^2} 
=\,\sum_{(\ell,m)\in S}\frac{\a_0+2\ell(\om_*-\frac{\a_0}{2})}{2(1+(\om_*-\frac{\a_0}{2} + \frac{\a_0}{2\ell})^2)}|\mathtt v_{\ell,m}|^2.
\end{equation}

$(ii)$ There exist a constant $\e_2 \in (0, \e_1]$ 
and a map $\psi:\,B_{V}(\e_2) \to B_{V}(\e_1)$ such that 
\begin{equation} \label{mI.VN.psi}
\bar\mI_V(\psi(v)) = |v|^2 
\end{equation}
for all $v \in B_{V}(\e_2)$. 
The map $\psi$ is analytic in $B_{V}(\e_2) \setminus \{ 0 \}$, 
it is of class $C^1( B_{V}(\e_2) )$, with Lipschitz differential,  
it is a diffeomorphism of $B_{V}(\e_2)$ onto its image $\psi( B_{V}(\e_2) )$, 
which is an open neighborhood of $\psi(0)=0$ contained in $B_{V}(\e_1)$,   
and 
\begin{equation} \label{psi.group.action}
\psi \circ \mT_\b = \mT_\b \circ \psi,\qquad\psi\circ\mR=\mR\circ\psi
\end{equation}
for all $\b \in \T$. 
Moreover, there exists $a_0 > 0$ such that, for all $a \in (0, a_0)$, 
the set $\mS_V(a)$ in \eqref{restricted.constraint} is 
\[
\mS_V(a) = \psi \big( \mathtt{S}(a) \big) 
\]
where
\begin{equation} \label{def.mS.0.VN.a}
\mathtt{S}(a) := \{ v \in V : |v|^2 = a \},
\end{equation}
and thus $\mS_V(a)$ is an analytic connected compact manifold of dimension $3$ embedded in $V$.
Its tangent and normal space at a point $v \in \mS_V(a)$ are
\begin{align} 
T_v( \mS_V(a) ) 
& = \{ \tilde v \in V:\,\la \grad \bar\mI_V (v) , \tilde v \ra_{(L^2)^2} = 0 \}, 
\label{tangent.space}
\\
N_v( \mS_V(a) ) 
& = \{ \lm \grad \bar\mI_V (v):\,\lm \in \R \},
\label{normal.space}
\end{align}
where $\grad \bar\mI_V(v) \neq 0$ for all $v \in \mS_V(a)$, all $a \in (0, a_0)$.

\end{lemma}

\begin{proof} $(i)$ By doing a Taylor expansion of $\mI_V$ about $v=0$ and observing that 
\begin{equation*}\bar\mI_V(0)=0,\quad\grad\bar\mI_V(0)=0,\quad d^2\bar\mI_V(0)[v,v]=\la d\grad\bar\mI_V(0)v,v\ra_{(L^2)^2}=\la-\pa_\om\mL(\om_*)v,v\ra_{(L^2)^2},\end{equation*}
we get \eqref{IZ.approx}.
The expression \eqref{def.mI.0} is obtained by the computations done in \ref{pa.om.L.scalar.product}, and so $\bar\mI_0$ is a positive-definite quadratic form whose levels are diffeomorphic to $\S^3$.

\medskip

$(ii)$ Define the diagonal linear map 
\begin{equation} \label{def.Lm}
\Lm:\,V \to V, \quad 
\Lm \mathtt{v}_{m} :=\lambda= \lambda_{\ell,m}\mathtt{v}_{\ell,m}, \quad 
\lm_{\ell,m} :=\Big\{\frac{\a_0+2\ell(\om_*-\frac{\a_0}{2})}{2(1+(\om_*-\frac{\a_0}{2} + \frac{\a_0}{2\ell})^2)}\Big\}^{-\frac12},\quad (\ell,m)\in S.
\end{equation}  
By \eqref{IZ.approx}, one has 
\begin{equation} \label{hat.square}
\bar\mI_0(\Lm v) = |v|^2.
\end{equation}
Given $v$, we look for a real number $\mu$ such that 
\begin{equation} \label{19:22}
\bar\mI_{V}( (1+\mu) \Lambda v ) = |v|^2. 
\end{equation}
We look for $\mu$ in the interval $[-\delta, \delta]$, 
with $\delta = \frac14$, 
and we assume that $|v| < \e_2$, where $\e_2 \in (0, \e_1]$ is such that 
$\frac54 |\Lm| \e_2 < \e_1$, so that, for $v \in B_{V}(\e_2)$, 
the point $(1+\mu) \Lm v$ is in the ball $B_{V}(\e_1)$ where $\bar\mI_{V}$ is defined. 
Let $R$ be the remainder in \eqref{IZ.approx}. 
By \eqref{hat.square}, one has 
$\bar\mI_0( (1+\mu) \Lm v ) = (1+\mu)^2 |v|^2$,  
and, since $\bar\mI_{V}= \bar\mI_0 + R$, 
\eqref{19:22} becomes 
\begin{equation} \label{19:23}
(2 \mu + \mu^2) |v|^2 + R( (1+\mu) \Lm v ) = 0.
\end{equation}
For $v \neq 0$, \eqref{19:23} is the fixed point equation $\mu = \mK(\mu)$
for the unknown $\mu$, where 
\begin{equation} \label{def.mK}
\mK(\mu) := - \frac{\mu^2}{2} - \frac{R( (1+\mu) \Lm v)}{2 |v|^2}.
\end{equation}
By the estimate in \eqref{IZ.approx}, 
for some constants $C_1, C_2$ one has
\[
| \mK(\mu) | \leq \frac{\delta^2}{2} + C_1 \e_2 \leq \delta, 
\quad \ 
| \mK'(\mu) | \leq \delta + C_2 \e_2 \leq \frac12 
\]
for all $|\mu| \leq \delta$, $|v| < \e_2$,  
provided $\e_2$ is sufficiently small, namely
$C_1 \e_2 \leq \frac12 \delta$ and $C_2 \e_2 \leq \frac14$. 
Hence, by the contraction mapping theorem, 
in the interval $[- \delta, \delta]$ there exists a unique fixed point of $\mK$, 
which we denote by $\mu(v)$. 
Hence, 
\begin{equation}  \label{def.psi.in.the.proof}
\mI_{N}( \psi (v) ) = |v|^2, \quad \ 
\psi(v) := (1 + \mu(v)) \Lm v,  
\end{equation}
for all $v \in B_{V}(\e_2) \setminus \{ 0 \}$. 
From the implicit function theorem 
applied to equation \eqref{19:23} around any pair $(v, \mu(v))$ 
it follows that the map $v \longmapsto \mu(v)$ is analytic in $B_{V}(\e_2) \setminus \{ 0 \}$. 
Moreover, 
$|\mu(v)| = |\mK(\mu(v))| \leq \frac12 \mu^2(v) + C_1 |v| \leq \frac18 |\mu(v)| + C_1 |v|$, 
whence 
\[
|\mu(v)| \leq C |v|.
\]
Thus, defining $\mu(0) := 0$, the function $\mu(v)$ is also Lipschitz in $B_{V}(\e_2)$. 
As a consequence, the function $\psi$ is analytic in $B_{V_N}(\e_2) \setminus \{ 0 \}$
and Lipschitz in $B_{V}(\e_2)$. In addition, $|\psi(v) - \psi(0) - \Lm v| \leq C |v|^2$, 
which means that $\psi$ is differentiable also at $v=0$, 
with differential $d\psi(0)\tilde v = \Lm \tilde v$. 
At $v \neq 0$, the differential is 
\[
d\psi(v)\tilde v = (d\mu(v)\tilde v) \Lm v + (1 + \mu(v)) \Lm \tilde v,
\]
and $d\psi(v) \to d\psi(0) = \Lm$ as $v \to 0$ because $\mu(v) \to 0$, $\Lm v \to 0$, 
and $|d\mu(v)\tilde v| \leq C |\tilde v|$ uniformly as $v \to 0$.
Thus, $\psi$ is of class $C^1$ in $B_{V}(\e_2)$. 
Moreover, 
\[
|d\psi(v)\tilde v - d\psi(0)\tilde v| \leq C |v||\tilde v|,
\]
i.e., the differential map $v \longmapsto d\psi(v)$ is Lipschitz continuous. 

The function $\psi$ is a diffeomorphism of open sets of $V$, and, 
for each $a \in (0, a_0)$ sufficiently small, one has 
\begin{align*}
\mS_{V}(a) 
& = \{ v \in B_{V}(\e_1) : \bar\mI_{V}(v) = a \} 
\\ 
& = \{ v = \psi(y) : y \in B_{V}(\e_2), \ a = \bar\mI_{V}(v) = \bar\mI_{V}(\psi(y)) = |y|^2 \}
\\ 
& = \psi ( \{ y \in V : |y|^2 = a \} ),
\end{align*}
namely $\mS_{V}(a)$ is the image of the sphere $\{ |y|^2 = a \}$ by the diffeomorphism $\psi$. 

By \eqref{matrix.actions}, we have
\begin{equation} \label{Lm.group.action} 
\Lm \mT_\alpha = \mT_\alpha \Lm,\qquad\Lm\circ\mR=\mR\circ\Lm,
\end{equation} 
and also
\begin{equation} \label{norm.group.action}
|\mT_\alpha v|^2 = |v|^2,\qquad|\mR v|^2=|v|^2.
\end{equation}  
From \eqref{Lm.group.action}, \eqref{hat.square}, and \eqref{norm.group.action}, we have 
\[
\bar\mI_0( \mT_\alpha \Lm v ) = \bar\mI_0( \Lm \mT_\alpha v ) 
= |\mT_\alpha v|^2 = |v|^2 = \bar\mI_0( \Lm v)
\] 
for all $v \in V$, and same for $\mR$. This implies that $\bar\mI_0 \circ \mT_\alpha = \bar\mI_0,\quad\bar\mI_0\circ\mR=\bar\mI_0$ 
because $\{ \Lm v : v \in V \} = V$.
By Lemma \ref{lemma:regularity}, the functional $\bar\mI_{V}$ has the same invariance property,  
and therefore the difference $R = \bar\mI_{V} - \bar\mI_0$ also satisfies
\begin{equation}  \label{mR.group.action}
R \circ \mT_\alpha = R,\qquad R\circ\mR=R
\end{equation}
for all $\alpha \in \T^1$. 
Now denote by $\mK(\mu, v)$ the scalar quantity $\mK(\mu)$ in \eqref{def.mK}. 
By uniqueness of fix point, we have proved that $\mK(\mu, v) = \mu$ if and only if $\mu=\mu(v)$. 
Moreover, by \eqref{norm.group.action} and \eqref{mR.group.action}, one has
\begin{equation*}\mK(\mu, \mT_\alpha v) = \mK(\mu, v),\qquad\mK(\mu,\mR v)=\mK(\mu,v)\end{equation*}
for all pairs $(\mu, v)$, all $\alpha \in \T^1$. 
Then 
\[
\mu(v) = \mK( \mu(v), v) = \mK( \mu(v), \mT_\alpha v),
\]
whence $\mu(v) = \mu( \mT_\alpha v)$, and similarly $\mu(v)=\mu(\mR v).$ 
This identity, together with \eqref{Lm.group.action}, 
gives \eqref{psi.group.action}.
\end{proof}

\begin{remark} By following straightforward the computations done in the proof of Lemma \ref{constraint.top.desc.}, it holds true also when $\dim V=2$. In this case, $\mS_V(a)$ is diffeomorphic to a unit circle $\S^1$.
\end{remark}

We are now ready to get the rotating waves.
\begin{lemma}
In the notations and assumptions made above, the following facts hold.

\medskip

$(i)$ For all $a \in \R$, the functional $\bar\mE_V\colon\, B_{V}(\e_1) \to \R$ defined in \eqref{restricted.functionals}
is Lipschitz continuous in $B_{V}(\e_1)$ and analytic in $B_{V}(\e_1) \setminus \{ 0 \}$.
Moreover, for all $\b \in \T^1$, one has  
\begin{equation} \label{mE.VN.group.action}
\bar\mE_V \circ \mT_\b = \bar\mE_V,\quad\bar\mE_V\circ\mR=\bar\mE_V
\end{equation}

\medskip

$(ii)$ For any $v \in \mS_V(a)$, one has 
\begin{equation} \label{grad.mE.V}
\grad \bar\mE_V(v) = \Pi_{Z} \mF (\om(v), u(v)).
\end{equation}

\medskip

$(iii)$ If $v\in\mS_V(a)$ is a constrained critical point for $\bar\mE_V$ along the constraint $\mS_V(a)$, then for all $a\in(0,a_0)$ with $\a_0$ as in Lemma \ref{constraint.top.desc.}, one has that also 
\begin{equation}\grad\bar\mE_V(v)=0.
\end{equation}
As a result, $v$ solves the bifurcation equation \eqref{bif.eq.reduced}.

\medskip

$(iv)$ There exist at least two orbits of solutions for \eqref{bif.eq.reduced} which are distinct with respect to both the actions $\mT_\b,\,\mR$.

\end{lemma}

\begin{proof} $(i)$ It follows from Lemma \ref{lemma:regularity}.

\medskip

$(ii)$ For a matter of simplicity of notation, we set $\la\cdot,\cdot\ra:=\la\cdot,\cdot\ra_{(L^2)^2}$.

For each point $v\in B_V(\e_1)\setminus\{0\}$ and for each direction $\hat v\in V$, we have 
\begin{align*}d\bar\mE_V(v)\hat v
&=d\bar\mH_V(v)\hat v - \om(v)d\bar\mI_V(v)\hat v - (d\om(v)\hat v)(\bar\mI_V(v)-a)
\\&=\la\grad\bar\mH_V(v)-\om(v)\grad\bar\mI_V(v),du(v)\hat v\ra - (d\om(v)\hat v)(\bar\mI_V(v)-a)
\\&=\la\mF(\om(v),u(v)),\hat v\ra + \la\mF(\om(v),u(v)),dw(\om(v),v)\hat v\ra - (d\om(v)\hat v)(\bar\mI_V(v)-a)
\\&=\la\mF(\om(v),u(v)),\hat v\ra - (d\om(v)\hat v)(\bar\mI_V(v)-a)
\\&\quad+\la\Pi_Z\mF(\om(v),u(v)),dw(\om(v),v)\hat v\ra + \la\Pi_{R^{\mathfrak{s},s}}\mF(\om(v),u(v)),dw(\om(v),v)\hat v\ra
\\&=\la\mF(\om(v),u(v)),\hat v\ra - (d\om(v)\hat v)(\bar\mI_V(v)-a).
\end{align*}
The last identity can be deduced by \eqref{zeros.implicit.func.} and by the fact that $dw(v)\hat v\in R^{\mathfrak{s},s}\perp Z$. If more $v\in\mS_V(a)$, then we deduce \eqref{grad.mE.V}.

\medskip

$(iii)$ If $v\in\mS_V(a)$ is a constrained critical point for $\bar\mE_V$ along the constraint $\mS_V(a)$, then there exists a constant $\lambda\in\R$ such that
\begin{equation*}\label{lagrange.multip.}\grad\bar\mE_V(v)=\lambda\grad\bar\mI_V(v).
\end{equation*}
Thus, using \eqref{IFT.om}, \eqref{F.def.}, \eqref{F.projection.onto.Z}, \eqref{grad.modified.I.}, 
the definition of $u(v)$ in \eqref{u.v.function},
\eqref{grad.mE.V}, and \eqref{lagrange.multip.}, 
we obtain
\begin{align} \label{0=lm.scal.prod}
0 & = F(\om(v), v) 
= \la \Pi_{Z} \mF (\om(v), u(v)) , (\grad \bar\mI_V)(u(v)) \ra_{(L^2)^2}
= \lm \la \grad \bar\mI_V(v) , \grad \bar\mI(u(v)) \ra_{(L^2)^2}.
\end{align}
By \eqref{IZ.approx} and \eqref{w.comp.estimate}, we get 
\[
\grad \bar\mI_V(v) = \grad \bar\mI_0(v) + O(|v|^2), 
\quad 
\grad \bar\mI(u(v)) = \grad \bar\mI_0(v) + O(|v|^2),
\]
and 
\[
\la \grad \bar\mI_V(v) , (\grad \bar\mI)(u(v)) \ra_{(L^2)^2} 
= |\grad \bar\mI_0(v)|^2 + O(|v|^3) 
\geq \tfrac12 |\grad \bar\mI_0(v)|^2 
> 0. 
\]
This means that the coefficient of $\lm$ in \eqref{0=lm.scal.prod} is nonzero, 
whence $\lm = 0$. Thus, by \eqref{lagrange.multip.}, $\grad \bar\mE_V(v) = 0$, 
then \eqref{bif.eq.reduced} follows from \eqref{grad.mE.V}.

\medskip

$(iv)$ It is enough to observe that by Lemma \ref{constraint.top.desc.}, the constraint $\mS_V(a)$ is a connected compact smooth manifold with no boundary, so any $C^1$ function over it attains at least one maximum and one minimum. If they coincide, then the whole constraint $\mS_V(a)$ is made of critical points, otherwise they generate two distinct critical  orbits by \eqref{mE.VN.group.action}.

\end{proof}

\begin{remark} The lower bound $2$ for the number of critical points coincide with that provided by the Benci $\S^1-$index, see for instance \cite{Mawhin.Willem}, Lemma 6.10, Section 6.4. This bound is optimal for bifurcations from eigenvalues of multiplicity 4: for an example, see \cite{BBMM}, Remark 5.8.
    
\end{remark}

\subsection{Existence rotating waves by angular frequency approach}

In this section, we prove the following theorem.

\begin{theorem}\label{thm:rw.frequency}Let $\kappa\in\N$ and $\a_0>0,\,\mathfrak{s}\ge0,\,s_0>0$ with $s\ge s_0>1$, let $\ell_0\in\N$ be fixed and, finally, let us suppose that \eqref{existence.resonance.condition} holds.
Then, the following facts hold.

\medskip

$(i)$ Let
\begin{align*}&\om_*:=\om_*^{(+)}:=\frac{\a_0}{2} - \frac{\a_0}{2\kappa\ell_0} + \frac{1}{2\kappa\ell_0}\sqrt{(\kappa\ell_0-1)(4\s_0\kappa\ell_0(\kappa\ell_0+1)-\a_0^2)},
\end{align*}     
and suppose that $\om_*<\frac{\a_0}{2}$ and $\om_*$ has multiplicity $4$.
Then, one of the following alternatives can occur.

\medskip

$(i.1)$ $u=0$ is a non-isolated solution for \eqref{rotating.wave.equation.0}.

$(i.2)$ There exists $\tilde\e>0$ and a one-sided neighbourhood $\mU_{\tilde\e}^{(s)}$ of $\om_*$ such that for all $\om\in\mU_{\e_0}^{(s)}$, there exist at least two distinct $\mT_\a-$orbits
\begin{equation*}\{(\eta_\om^{(i)},\b_\om^{(i)})\circ\mT_\a\colon\,\a\in\T^1\},\qquad i\in\{1,2\}
\end{equation*}
of nonzero solutions of \eqref{rotating.wave.equation.0}
satisfying the $\kappa-$fold symmetry, having angular frequency $\om$, and 
\begin{equation*}\|\eta_\om^{(i)}\|_{\mathfrak{s},s+\frac32} + \|\b_{\om}^{(i)}\|_{\mathfrak{s},s+1}\to0
\end{equation*}
as $\om\to\om_*$.

$(i.3)$ There exists $\tilde\e>0$ such that for all $\om\in(\om_*-\tilde\e,\om_*+\tilde\e)$, there exist at least one $\mT_\a-$orbit
\begin{equation*}\{(\eta_\om,\b_\om)\circ\mT_\a\colon\,\a\in\T^1\}
\end{equation*}
of nonzero solutions of \eqref{rotating.wave.equation.0} satisfying the $\kappa-$fold symmetry, having angular frequency $\om$, and satisfying
\begin{equation*}\|\eta_\om\|_{\mathfrak{s},s+\frac32} + \|\b_{\om}\|_{\mathfrak{s},s+1}\to0
\end{equation*}
as $\om\to\om_*$.

\medskip

$(ii)$ Similarly, all the theses hold for 
\begin{align*}&\om_*:=\om_*^{(-)}:=\frac{\a_0}{2} - \frac{\a_0}{2\kappa\ell_0} - \frac{1}{2\kappa\ell_0}\sqrt{(\kappa\ell_0-1)(4\s_0\kappa\ell_0(\kappa\ell_0+1)-\a_0^2)},
\end{align*}
supposing only that $\om_*$ ha multiplicity $4$.

\end{theorem}

We recall that by \eqref{bif.equation.is.variational}, the nonzero rotating waves are nonzero critical points for $\Phi(\om,v)$ in $\mU_{\e_0}$. Furthermore, by \eqref{action.expansion} and \ref{pa.om.L.scalar.product}, we have in $\mU_{\e_0}$
\begin{equation*}\Phi(\om,v)=-(\om-\om_*)\sum_{(\ell,m)\in S}\frac{\a_0-2\ell(\om_*-\frac{\a_0}{2})}{2(1+(\om_*-\frac{\a_0}{2} + \frac{\a_0}{2\ell})^2)}|\mathtt v_{\ell,m}|^2 + R(v),
\end{equation*}
where $R(v)=O(|v|^3)$.

The case $(i.1)$ is that $v=0$ is not an isolated critical point for $\bar\Phi(\om_*,\cdot)$. By definition of non-isolated critical point, the following fact holds.
\begin{lemma}\label{lemma:nonisolated.critical.point.case} If $v=0$ is not an isolated critical point for $\bar\Phi(\om_*,\cdot)$, then there exists a sequence $(v_n)_{n\in\N}\subset V$ such that $v_n\to0$ and each $v_n$ is a critical point for $\bar\Phi(\om_*,\cdot)$. As a result, each $u_n:=v_n+w(\om_*,v_n)$ is a solution for the rotating wave equation $\mF(\om,u)=0$. 

As a result, point $(i.1)$ holds.
\end{lemma}

From now on, we suppose that $v=0$ is an isolated critical point for $\bar\Phi(\om_*,\cdot)$. With an abuse of notations, we identify $V$ with the space of the vectors $(v_{\ell,m}\colon\,(\ell,m)\in\mS)$ of coefficients of $v\in V$, and we consider the Hilbert space $(V,|\cdot|)$, where $|\cdot|^2$ is the sum of the squares of such coefficients. Moreover, we set $B(\rho)$ to be the ball of $V$ centered at $v=0$ and having radius $\rho$. Finally, we consider the map $f_{\om}\colon\,v\in B(\e_0)\longmapsto f_\om(v):=\Phi(\om,v)\in\R$, for any $\om\in(\om_*-\e_0,\om_*+\e_0)$.

\medskip

We want to prove point $(i.2)$ first.

\begin{lemma}\label{lemma:point.ii} Let us suppose that $v=0$ is a local strict maximum or minimum for $f_{\om_*}$. Then, we have that for $\tilde\e>0$ small enough and $\om\in(\om_*-\tilde\e,\om_*)$, there exist at least two nonzero critical points $\overline v(\om),\,\underline{v}(\om)$ for $f_\om$.
As a result, point $(i.2)$ holds.
\end{lemma}

\begin{proof} There exists $\e_1\in(0,\e_0)$ such that 
\begin{equation*}\max_{\pa B(\e_1)}f_{\om_*}\le-M<0.
\end{equation*}
By continuity of $\bar\Phi$ with respect to $\om$, up to restrictions we have that for all $\om\in(\om_*-\e_1,\om_*)$,
\begin{equation}\label{f.om.is.negative}\max_{\pa B(\e_1)}f_{\om}\le-M<0.
\end{equation}
However, for all $\om<\om_*$, $v=0$ is a strict local minimum for $f_\om$, and so there exists $\e_2\in(0,\e_1)$ and $\overline{m}=\overline{m}(\om)>0$ such that $f_\om>0$ in $\overline{B(\e_2)}\setminus\{0\}$ and
\begin{equation}\label{0.is.local.minimum}\overline{m}(\om)=\max_{\pa B(\e_2)}f_\om= \max_{\overline{B(\e_2)}}f_\om
\end{equation}
As a result, since $f_\om$ is positive in $B(\e_2)$ and negative in $\pa B(\e_1)$, then we can suppose $\overline{m}(\om)$ to be realized over a critical point $\overline{v}(\om)\in B(\e)\setminus\{0\}$.

Let us now set
\begin{align}&\Gamma:=\{\g\in C([0,1],\overline{B(\e_1)})\colon\,\g(0)=0,\,\g(1)\in\pa B(\e_1)\}, \label{curve.space}
\\&\underline{m}(\om):=\inf_{\g\in\Gamma}\max_{t\in[0,1]} f_\om(\g(t)).\label{mountain.pass.critical.value}
\end{align}
Let us notice that since any curve $\g\in\Gamma$ crosses $\pa B(\e_2)$, we have that 
\begin{align}&\underline{m}(\om)\le\overline{m}(\om), \label{underline.m.is.less.than.overline.m}
\\&\underline{m}(\om)\ge\min_{\pa B(\e_2)}f_\om>0. \label{underline.m.is.positive}
\end{align}
As a result, the fact that $\underline{m}(\om)$ is a critical value for any $\om\in(\om_*-\tilde\e,\om_*)$, with $\tilde\e$ small enough, results from the existence of Palais-Smale sequence (see \cite{BBMM}, Section 5.1, Case $(ii)$).

\end{proof}

Finally, to prove point $(i.3)$, we point out that if $v=0$ is neither a maximum nor a minimum for $f_{\om_*}$, the Mountain pass critical set \eqref{mountain.pass.critical.value} is not compact anymore. However, following \cite{BBMM}, we define the Conley block and the exit set respectively as
\begin{align*}&W:=\{v\in B(\e_0)\colon\,|R(v)|\le\g,\,\,f_\om(v)\le\mu\},\qquad\g\in(0,\e_0),\,\,\mu>0,
\\&W_-:=\{v\in W\colon\,R(v)=-\g\}.
\end{align*}

Let us now call $\eta_t$ the gradient flow map generated by $R$, that is,
\begin{equation}\label{gradient.R.flow}\frac{d}{dt}\eta_t(v)=-\grad R\circ\eta_t(v),\qquad \eta_0(v)=v.
\end{equation}
We now prove the existence of a Conley block for $f_\om$ (with $\om$ sufficiently close to $\om_*$), which is nothing less but $W$. 

\begin{proposition}[Proposition A.2, \cite{BBMM}]\label{thm:conley.block.existence} For $\g\in(0,\e_0)$ small enough, the set $W$ satisfies the following properties.

\medskip

$(i)$ For all $v\in W$, either $R(v)=-\g$ or $\eta_t(v)\in W$ for all $t>0$ small enough.

\medskip

$(ii)$ $W_-$ is the exit set of $W$ with respect to $\eta_t$, or else,
\begin{equation}\label{exit.set.char}W_-=\{v\in W\colon\,\eta_t(v)\notin W,\,\,\forall t>0\,\,\text{small enough}\}.
\end{equation}
If $R$ attains negative values arbitrarily close to $v=0$, then $W_-\ne\emptyset$ and there exists $v\in W_-$ such that $\eta_t(v)\to0$ as $t\to-\infty$. 

\medskip

$(iii)$ One has
\begin{align}&\text{int}(W)=\{v\in B(\e_0)\colon\,|R(v)|<\g,\,|g(v)|<c\}, \label{W.interior}
\\&\pa W= W_-\cup\mW, \label{W.boundary}
\end{align}
where
\begin{align}&\mW:=\mW_1\cup\mW_2, \label{non.exit.set}
\\&\mW_1:=\{v\in B(\e_0)\colon\,R(v)\in(-\g,\g],\,g(v)=c\}\subset B(\e)\setminus\overline{B\Big(\frac{\e}{2}\Big)}, \label{non.exit.set.1}
\\&\mW_2:=\{v\in B(\e_0)\colon\,R(v)=\g,\,g(v)<c\}. \label{non.exit.set.2}
\end{align}

\medskip

$(iv)$ If $F\colon\,B(\e_0)\to\R$ is $C^1-$close enough to $R$ on $\pa W\setminus W_-$, then for all $v\in W\setminus W_-$, the gradient flow $\eta_t^{(F)}$ generated by $F$ satisfies $\eta_t^{(F)}(v)\in W$, for all $t>0$ small enough.
    
\end{proposition}

We are ready to prove the final lemma of this section.

\begin{lemma} Let us suppose that $f_{\om_*}$ attains both positive and negative values about $v=0$. Then, for $\g,\mu>0$ small enough, we have that for all $\tilde\e>0$ small enough and $\om\in(\om_*-\tilde\e,\om_*+\tilde\e)$, there exists at least one nonzero critical point $v(\om)$ for $f_\om$.
As a result, point $(i.3)$ holds.

\end{lemma}

\begin{proof} We consider only the case $\om>\om_*$, because in the other case, we do similar arguments for $-f_{\om}$ with $\om<\om_*$.

First of all, by Proposition \ref{thm:conley.block.existence}, there exists a (invariant-by-deformation) Conley block $W$ and an exit set $W_-$ for $f_{\om_*}$ where $f_{\om_*}|_{W_-}<0$.
Let us consider 
\begin{align}&\Gamma:=\{\g\in C([0,1],W)\colon\,\g(0)=0,\,\g(1)\in W_-\},
\\&m(\om):=\inf_{\g\in\Gamma}\max_{t\in[0,1]}f_\om(\g(t)).\label{mountain.pass.value.1}
\end{align}
By the same Proposition, item $(ii)$, there is a continuous trajectory which connects $0$ and $W_-$, and thus $\Gamma\ne\emptyset$. Moreover, let us observe that $m(\om)>0$: indeed, like in Lemma \ref{lemma:point.ii}, there exists some $\tilde\e>0$ such that $B(\tilde\e)\subset W$ and
\begin{equation*}\max_{\pa B(\tilde\e)}f_\om(v)>0.
\end{equation*}
Then, by the existence of a Palais-Smale sequence inside $W$ at the level $m(\om)$ (compare \cite{BBMM}), we conclude.
    
\end{proof}

\section{Conditional energetic stability analysis for rotating circles}

\medskip

At the beginning of Section 4, we showed that the rotating circle is linearly stable. However, a more precise analysis involving the Hamiltonian structure shows that nonlinear stability is not obvious, if not generally false.

In this section, we show that rotating circle $(\zeta,\g)=(0,0)$ is always conditionally stable in the space $E:=H^1\times H_0^\frac12$ if we restrict to the constraints $\mV(\zeta)=\mV(0)$ and $\mB(\zeta,\g)=(0,0)$, but conditionally unstable in the same space (and thus, in all the higher norms) if the \emph{modified Bond number} $C=\frac{\s_0}{\a_0^2}>\frac14$ and we do not make any further restriction.

\subsection{Linearization of the Hamiltonian at the rotating circle solution}

\medskip

Let us compute the Hessian operator of $\bar\mH$ at the rotating circle solution $(\zeta,\g)=(0,0)$:
\begin{align}\mL_0[\zeta,\g]:&=d\grad\bar\mH(\om,0)[\zeta,\g] \notag
\\&=d\mC(0)^*\circ d\grad\mH(0)\circ d\mC(0)[\zeta,\g] \notag
\\&=\begin{pmatrix}1 & -\frac{\a_0}{2}\Pi_0^\perp\pa_\th^{-1} \\ 0 & 1
\end{pmatrix}\begin{pmatrix}-\s_0\pa_{\th\th} + (- \s_0 + \frac{\a_0^2}{4}) &  -\frac{\a_0}{2}\pa_\th \\ \frac{\a_0}{2}\pa_\th & G(0)
\end{pmatrix}\begin{pmatrix}\zeta \\ \frac{\a_0}{2}\pa_\th^{-1}\Pi_0^\perp\zeta + \g\end{pmatrix}\notag
\\&=\begin{pmatrix}-\s_0\pa_{\th\th} -\s_0 + \frac{\a_0^2}{4}\Pi_0 & -\frac{\a_0}{2}\pa_\th - \frac{\a_0}{2}\Pi_0^\perp\pa_\th^{-1} G(0) \\ \frac{\a_0}{2}\pa_\th & G(0)
\end{pmatrix}\begin{pmatrix}\zeta \\ \frac{\a_0}{2}\pa_\th^{-1}\Pi_0^\perp\zeta + \g\end{pmatrix}  \notag
\\&=\begin{pmatrix}(-\s_0\pa_{\th\th} - \s_0 + \frac{\a_0^2}{4}\Pi_0 - \frac{\a_0^2}{4}\Pi_0^\perp - \frac{\a_0^2}{4}\Pi_0^\perp\pa_\th^{-1}G(0)\pa_\th^{-1}\Pi_0^\perp)\zeta - (\frac{\a_0}{2}\pa_\th + \frac{\a_0}{2}\Pi_0^\perp\pa_\th^{-1}G(0))\g \\ (\frac{\a_0}{2}\pa_\th + \frac{\a_0}{2}G(0)\pa_\th^{-1}\Pi_0^\perp)\zeta + G(0)\g
\end{pmatrix}  \notag
\\&=\begin{pmatrix}-\s_0\pa_{\th\th} - \s_0  + \frac{\a_0^2}{4}\Pi_0 - \frac{\a_0^2}{4}\Pi_0^\perp  -  \frac{\a_0^2}{4}\Pi_0^\perp\pa_\th^{-1}G(0)\pa_\th^{-1}\Pi_0^\perp & -\frac{\a_0}{2}\Pi_0^\perp\pa_\th^{-1} G(0) -\frac{\a_0}{2}\pa_\th \\ \frac{\a_0}{2} G(0)\pa_\th^{-1}\Pi_0^\perp + \frac{\a_0}{2}\pa_\th & G(0)
\end{pmatrix}\begin{pmatrix}\zeta \\ \g
\end{pmatrix}
\label{variational.linearization.at.0}
\end{align}
Expanding $\zeta,\g$ using \eqref{fourier.expansion}, we get
\begin{equation}\label{L.is.sum.of.blocks.0}\mL_0[\eta,\b]=\sum_{(\ell,m)\in\mT}\mL_0^{(\ell,m)}\begin{pmatrix}\eta_{\ell,m}\ph_{\ell,m} \\ \b_{\ell,-m}\ph_{\ell,-m}
\end{pmatrix},
\end{equation}
where
\begin{align}\label{L.lm.blocks.0}\mL_0^{(\ell,m)}:=\begin{cases}\begin{pmatrix}\s_0\ell^2 - (\s_0 + \frac{\a_0^2}{4}) + \frac{\a_0^2}{4\ell} & \frac{m\a_0}{2} - \frac{m\ell\a_0}{2} \\ \frac{m\a_0}{2} -\frac{m\ell\a_0}{2} & \ell
\end{pmatrix} & \text{if}\quad(\ell,m)\ne(0,0)
\\\begin{pmatrix}-\s_0 + \frac{\a_0^2}{4} & 0 \\ 0 & 0
\end{pmatrix} & \text{if}\quad(\ell,m)=(0,0)
\end{cases}
\end{align}
The eigenvalue equation is then
\begin{align}&P(\lm,0)=\lm\Big(\lm + \s_0 - \frac{\a_0^2}{4} \Big)=0,\qquad(\ell,m)=(0,0),
\\&P(\lm,\ell)=\lm^2 - \Big(\s_0\ell^2 + \ell - \s_0 - \frac{\a_0^2}{4} + \frac{\a_0^2}{4\ell} \Big)\lm + \Big(\s_0\ell^2 - \frac{\a_0^2}{4}\ell^2 - \s_0 + \frac{\a_0^2}{4} \Big)\ell = 0,\qquad(\ell,m)\ne(0,0).\label{eigenvalue.equation}
\end{align}
We immediately spot that when $(\ell,m)=(0,0)$,
\begin{equation}\label{eigenvalue.zero.frequency}\lm_1(0)=-\s_0 + \frac{\a_0^2}{4},\qquad \lm_2(0)=0.
\end{equation}
Instead, for $(\ell,m)\ne(0,0)$, since the constant term is the product of the roots of \eqref{eigenvalue.equation}, then it is always nonnegative for all $\ell\in\N$. Precisely, it is zero if and only if $\ell=1$, otherwise the eigenvalues are strictly positive and they are given by
\begin{align}&\lm_\pm(\ell)=\frac12\Big(\s_0\ell^2 + \ell - \s_0 - \frac{\a_0^2}{4} + \frac{\a_0^2}{4\ell} \Big)\pm\frac12\sqrt{\Delta(\ell)}, \notag
\\&\Delta(\ell):=\Big(\s_0\ell^2 + \ell - \s_0 - \frac{\a_0^2}{4} + \frac{\a_0^2}{4\ell}\Big)^2 - 4\ell\Big(\s_0\ell^2 - \frac{\a_0^2}{4}\ell^2 - \s_0 + \frac{\a_0^2}{4} \Big).
\label{eigenvalues.formula}
\end{align}
From this formula, it is easy to check that 
\begin{equation}\label{eigenvalues.asymptotics}\lim_{\ell\to+\infty}\frac{\lm_+(\ell)}{\ell^2}=\s_0,\qquad\lim_{\ell\to+\infty}\frac{\lm_-(\ell)}{\ell}=\s_0.
\end{equation}

\subsection{Conditional energetic stability with volume constraint}

\medskip

Here, we show conditional energetic stability no matter what the value of $C$ is, provided the energy functional $\bar\mH$ is uniformly strictly convex.

To this term, we consider the constraint $\{\mV(\zeta)=\mV(0)\}$ and $\{\mB(\zeta,\g)=\mB(0,0)=(0,0),\quad\mP(\zeta)=\mP(0)=(0,0)\}$. Indeed, we have the following fact.
\begin{lemma}\label{lemma:smallness.on.constraints} Given the expansions
\begin{equation*}\zeta=\sum_{(\ell,m)\in\mT}\zeta_{\ell,m}\ph_{\ell,m},\qquad\g=\sum_{(\ell,m)\in\mT}\g_{\ell,m}\ph_{\ell,m},\end{equation*}
the following asymptotics holds.

\medskip

$(i)$ If $\mV(\zeta)=\mV(0)$, then
\begin{equation}\zeta_{0,0}=O(\|\zeta\|_1^2).\label{average.smallness}
\end{equation}

$(ii)$ If $\mB(\zeta,\g)=(0,0)$ and $\mP(\zeta,\g)=(0,0)$, then
\begin{equation}\zeta_{1,\pm1}=O(\|\zeta\|_1^2),\qquad\g_{1,\pm1}=O(\|\zeta\|_1\|\g\|_\frac12 + \|\zeta\|_1^2).\label{barycenter.average.smallness}
\end{equation}

\end{lemma}

\begin{proof} $(i)$ If $\mV(\zeta)=\mV(0)$, then by \eqref{volume.new.coordinates}
\begin{equation*}\int_\T e^{2\zeta}\,d\th = \int_{\T^1}1\,d\th.
\end{equation*}
If we Taylor expand $\e^{2\zeta}$, by Parseval theorem we get
\begin{equation*}2\int_\T\zeta\,d\th = O(\|\zeta\|_0^2),
\end{equation*}
and thus $\zeta_{0,0}=O(\|\zeta\|_0^2)$. Since for $\|\zeta\|_0$ small enough we have
\begin{equation*}\|\zeta\|_0^2=|\zeta_{0,0}|^2 + \sum_{(\ell,m)\in\mT,\ell\ge1}\zeta_{\ell,m}\ph_{\ell,m}\le\frac12\|\zeta\|_0^2 + \sum_{(\ell,m)\in\mT,\ell\ge1}\zeta_{\ell,m}\ph_{\ell,m},
\end{equation*}
we get that $\|\zeta\|_0^2\sim\|\Pi_0^\perp\zeta\|_0^2\lesssim\|\zeta\|_1^2$, and we then deduce the estimate \eqref{average.smallness}.

\medskip

$(ii)$ By \eqref{position.time.derivative}, $\mP$ is conserved and thus the statement is well-posed. By \eqref{barycenter.position} and Taylor expansion of $e^{3\zeta}$ we get the first estimate in \eqref{barycenter.average.smallness}. Moreover, by \eqref{barycenter.velocity.0.rigidity}, Taylor expansion of $\zeta$ and integration by part we have
\begin{equation*}\int_{\T}\g'\begin{pmatrix}-\sin\th \\ \cos\th
\end{pmatrix}\,d\th = \frac{\a_0}{4}\int_\T\zeta\begin{pmatrix}-\sin\th \\ \cos\th
\end{pmatrix}\,d\th + O(\|\zeta\|_1\|\g\|_0 + \|\zeta\|_0^2) = O(\|\zeta\|_1\|\g\|_0 + \|\zeta\|_0^2),
\end{equation*}
which is $\g_{1,\pm1}=O(\|\zeta\|_1\|\g\|_0 + \|\zeta\|_1^2)$. Since $\g\in H_0^\frac12$, we have $\|\g\|_0=\|\Pi_0^\perp\g\|_0\lesssim\|\g\|_\frac12$, which gives the second estimate in \eqref{barycenter.average.smallness}.
    
\end{proof}

Then, we are ready to show conditional energetic stability of rotating circles in the constraint $\{\mV(\zeta)=\mV(0)\}\cap\{\mB(\zeta,\g)=(0,0),\,\mP(\zeta)=(0,0)\}$.
\begin{theorem}\label{thm:conditional.stability} Assume that the Cauchy problem 
\begin{equation}\label{Cauchy.problem}\frac{d}{dt}\begin{pmatrix}\zeta \\ \g
\end{pmatrix} = J(\zeta)\grad\bar\mH(\zeta,\g),\qquad(\zeta,\g)|_{t=0}=(\zeta_0,\g_0)
\end{equation}
is locally well posed in the Banach space $C([0,T],E)$ in the weak sense described in \cite{VWW}, page 9. Then, for any $\e>0$, there exists $\d_0=\d_0(\e)>0$ such that for any $d\in(0,\d_0)$ and for any initial datum $u_0=(\zeta_0,\g_0)\in E$ such that $\mV(\zeta_0)=\mV(0),\,\mB(u_0)=(0,0),\,\mP(\zeta_0)=(0,0)$ and $\|u_0\|_E\le\d$, we have $\|u(t)\|_E\le\e$ for any $t>0$. Here, $u(t)=(\zeta(t),\g(t))$ is the solution of \eqref{Cauchy.problem} with initial datum $u_0$.

\end{theorem}

\begin{proof} Consider any $u:=(\zeta,\g)\in E$. Then, we have the expansion for $\bar\mH$ about the rotating circle:
\begin{equation}\label{Hamiltonian.expansion}\bar\mH(u)=\bar\mH(0) + \la\mL_0 u,u\ra + O(\|u\|_{E}^3).
\end{equation}
Now, by the analysis led in Section 3.3, we have
\begin{align}\la\mL_0u,u\ra
&=\lm_1(0)|\zeta_{0,0}|^2 + \sum_{(\ell,m)\in\mT,\ell\ge2}\lm_{+}(\ell)|\zeta_{\ell,m}|^2 + \sum_{(\ell,m)\in\mT,\ell\ge2}\lm_{-}(\ell)|\g_{\ell,m}|^2 \notag
\\&\gtrsim\lm_1(0)|\zeta_{0,0}|^2 + \sum_{(\ell,m)\in\mT,\ell\ge2}(\ell^2|\zeta_{\ell,m}|^2 + \ell|\g_{\ell,m}|^2)
\label{second.order.Hamiltonian}
\end{align}
If we now assume $\mV(\zeta)=\mV(0)$, by \eqref{average.smallness} we get $|\zeta_{0,0}|^2=O(\|\zeta\|_1^4)=O(\|u\|_E^4)$, and if we assume $\mB(\zeta,\g)=(0,0)$, by \eqref{barycenter.average.smallness} we get $|\zeta_{1,\pm1}|^2,|\g_{1,\pm1}|^2=O(\|u\|_E^4)$. Altogether, this implies that $\la\mL_0 u,u\ra\gtrsim\|u\|_E^2$, and so coming back to \eqref{Hamiltonian.expansion} we get
\begin{equation}\label{Hamiltonian.convex}\bar\mH(u)-\bar\mH(0)\gtrsim\|u\|_E^2.
\end{equation}
Then, let us suppose that the statement does not hold, namely, there exists $\e_0>0$ such that for all $\delta>0$, for some initial datum $u_0=(\zeta_0,\g_0)$ satisfying $\|u_0\|_E\le\d$ and $\mV(\zeta_0)=\mV(0),\,\mB(u_0)=(0,0)$, there exists a time $T_\d>0$ such that $\|u(T_\d)\|_E\ge\e_0$. By \eqref{Hamiltonian.convex} and Lipschitz-continuity of $\bar\mH$ on balls in $E$, we have
\begin{equation*}\|u(T_\d)\|_E^2\lesssim\bar\mH(u(T_\d))-\bar\mH(0)=\bar\mH(u_0)-\bar\mH(0)\lesssim\|u_0\|_E^2\le\d^2.
\end{equation*}
Choosing $\d>0$ in a suitable way, we get a contradiction. This concludes the proof.

\end{proof}

\begin{remark} We remark here that the proof is independent of how large is the modified Bond number $C=\frac{\s_0}{\a_0^2}$. Indeed, despite the fact that the (possibly negative) eigenvalue $\lm_1(0)$ arises, its corresponding direction is negligible thanks to the constraints $\mV(\zeta)=\mV(0)$, $\mB(\zeta,\g)=(0,0)$ and $\mP(\zeta)=(0,0)$.
\end{remark}

\begin{flushright}

\textbf{Giuseppe La Scala}

Mathematical and Physical Sciences for Advanced Materials and Technologies

Scuola Superiore Meridionale

Via Mezzocannone, 4, 80138 Naples, Italy

giuseppe.lascala-ssm@unina.it

\end{flushright}


\begin{thebibliography}{99}


\bibitem{B.J.LM} {\sc P. Baldi, V. Julin, D.A. La Manna.} \emph {Liquid drop with capillarity and rotating traveling waves}, Archive for Rational Mechanics and Analysis 250 (1), 4 (2026).

\bibitem{BLS} {\sc P. Baldi, D.A. La Manna, G. La Scala}, \emph{Bifurcation from Multiple Eigenvalues of Rotating Traveling Waves on a Capillary Drop}, preprint, arxiv:2504.01555.

\bibitem{Benci} {\sc V. Benci}, \emph{A geometrical index for the group $\S^1$ and some applications to the study of periodic solutions of ordinary differential equations}, Comm. Pure Appl. Math. \textbf{34} (1981), 393-432.

\bibitem{BBMM} {\sc T. Barbieri, M. Berti, A. Maspero, M. Mazzucchelli}, \emph{Bifurcation of gravity-capillary Stokes waves with constant vorticity}, Journal of Differential Equations, Volume 451 (2026).

\bibitem{BMV} {\sc M. Berti, A. Maspero, P. Ventura}, \emph{On the analyticity of the Dirichlet-Neumann operator and Stokes waves}, Atti Accad. Naz. Lincei Cl. Sci. Fis. Mat. Natur. 33 no. 3, 611-650 (2022).

\bibitem{Beyer.Gunther} {\sc K. Beyer, M. G\"unther}, \emph{On the Cauchy Problem for a Capillary Drop. Part I: Irrotational Motions}, Math. Meth. Appl. Sci. 21, 1149-1183 (1998).

\bibitem{Chow.Hale} {\sc S.N. Chow, J.K. Hale}, \emph{Methods of Bifurcation Theory}, Springer-Verlag (1982).

\bibitem{Constantin.Strauss} {\sc A. Constantin, W. Strauss}, \emph{Exact steady periodic water waves with vorticity}, Comm. Pure Appl. Math. 57, no. 4, 481-527 (2004).

\bibitem{CSV} {\sc A. Constantin, W. Strauss, E. Varvaruca}, \emph{Global bifurcation of steady gravity water waves with critical layers}, Acta Math. 217, no. 2, 195–262 (2016).

\bibitem{CS}{\sc D. Coutand, S. Shkoller.}\emph{ Well-posedness of the free-surface incompressible Euler equations with or without surface tension}, J. Amer. Math. Soc. {\bf 20} , 829–930 (2007).

\bibitem{Crandall.Rabinowitz} {\sc M.G. Crandall. P.H. Rabinowitz}, \emph{Bifurcation from simple eigenvalues}, J. Funct. Anal. \textbf{8} (1971), 321–340.

\bibitem{Dyachenko} {\sc S. Dyachenko}, \textit{Traveling capillary waves on the boundary of a fluid disc}, Stud. Appl. Math. 148 (2022).

\bibitem{GHL} {\sc S. Gallot, D. Hulin, J. Lafontaine}, \emph{Riemannian Geometry}, Springer-Verlag (1987).

\bibitem{GSS.I} {\sc M. Grillakis, J. Shatah, W. Strauss}, \emph{Stability Theory of Solitary Waves in the Presence of Symmetry I}, Journal of Functional Analysis \textbf{74}, 160-197 (1987).

\bibitem{Iooss-Plotnikov-Toland} {\sc G. Iooss, P. Plotnikov, J. Toland}, \emph{Standing Waves on an Infinitely Deep Perfect Fluid Under Gravity}, Arch. Rational Mech. Anal., Volume 177, pages 367–478 (2005).

\bibitem{JL}{\sc V. Julin, D.A. La Manna}.\emph{ A Priori Estimates for the Motion of Charged Liquid Drop: A
Dynamic Approach via Free Boundary Euler Equations.} J. Math. Fluid Mech. {\bf 26}: 48,(2024).

\bibitem{K.L}{\sc v. Kozlov, e. Lokharu}, \emph{Global Bifurcation and Highest Waves on Water of Finite Depth}, Arch. Ration. Mech. Anal. 247, 5, 98 (2023). 

\bibitem{Lannes} {\sc D. Lannes}, \emph{The Water Waves Problem: Mathematical Analysis and Asymptotics}, American
Mathematical Soc. (2013).

\bibitem{Lannes1} {\sc D. Lannes}, \emph{Well-posedness of the water-waves equations}, J. Amer. Math. Soc., 3, 605–654,
18 (2005).

\bibitem{L} {\sc G. La Scala}, \emph{Two-dimensional capillary liquid drop: Craig-Sulem formulation on $\T^1$ and bifurcations from multiple eigenvalues of rotating waves}, preprint, arXiv:2505.11650.

\bibitem{Mawhin.Willem} {\sc J.\ Mawhin, M.\ Willem}, \emph{Critical Point Theory and Hamiltonian Systems}, Springer, 1989.

\bibitem{Martin} {\sc C.I. Martin}, \emph{Local bifurcation and regularity for steady periodic capillary-gravity water waves with constant vorticity}, Nonlinear Anal.: Real World Applications, 14, 131-149 (2013).

\bibitem{MNS} {\sc D. Meyer, L. Niebel, C. Seis}, 
\emph{Steady Bubbles and Drops in Inviscid Fluids}, Calc. Var. (2025), 64:299.

\bibitem{Moser} {\sc J. Moser}, \emph{Periodic orbits near an equilibrium and a theorem by Alan Weinstein}, Comm. Pure Appl. Math., 29 (1976), pp. 727–747.

\bibitem{Lord.Rayleigh.jets} {\sc L. Rayleigh}, \emph{On the capillary phenomenon of jets}, Proc. R. Soc. Lond. \textbf{29} (1879).

\bibitem{Shao.initial.notes} {\sc C. Shao}, 
\emph{Longtime Dynamics of Irrotational Spherical Water Drops: Initial Notes}, 
preprint available at https://arxiv.org/pdf/2301.00115.

\bibitem{Shao.G.paralinearization} {\sc C. Shao}, \emph{On the Cauchy Problem of Spherical Capillary Water Waves}, preprint available at https://arxiv.org/pdf/2310.07113.

\bibitem{Shatah.Zeng} {\sc J. Shatah, C. Zeng}, 
\emph{Geometry and a priori estimates for free boundary problems of the Euler’s equation}, 
Comm. Pure Appl. Math. 61 (5), 698-744, 2008.

\bibitem{VWW} {\sc K. Varholm, E. Wahlén, S. Walsh}, \emph{On the Stability of Solitary Water Waves with a Point Vortex}, Communications on Pure and Applied Mathematics, Vol. LXXIII, 2634–2684 (2020).

\bibitem{Wahlen} {\sc E. Wahlén}, \emph{Steady periodic capillary waves with vorticity}, Ark. Mat. 44 (2006), 367-387.

\bibitem{Wahlen.1} {\sc E. Wahlén}, \emph{A Hamiltonian formulation of water waves with constant vorticity}, Letters in
Mathematical Physics, \textbf{79} no.3, 303-315(2007).

\bibitem{Wahlen.Weber} {\sc E. Wahlén, J. Weber}, \emph{Large-amplitude steady gravity water waves with general vorticity and critical layers}, Duke Math. J. 173 (11), 2197-2258 (2024).

\bibitem{Weinstein} {\sc A. Weinstein}, \emph{Normal modes for nonlinear hamiltonian systems}, Invent. Math., 20 (1973), pp. 47–57.





















\end{thebibliography}
\end{document}